\documentclass{article}
\usepackage[utf8]{inputenc}
\usepackage{graphicx}
\usepackage{wrapfig}
\graphicspath{ {./images/} }
\usepackage{hyperref}
\usepackage{url}

\usepackage{amsmath}
\usepackage{bm}
\usepackage{mathtools}
\usepackage{amssymb}
\usepackage{amsfonts}
\usepackage{amsthm}
\usepackage[normalem]{ulem}

\usepackage{xcolor}
\usepackage{graphicx}
\usepackage{caption}
\usepackage{subcaption}
\usepackage{wrapfig}
\usepackage{float}
\usepackage[justification=centering]{caption}
\usepackage{geometry}
\usepackage{pgf,tikz,pgfplots}
\pgfplotsset{compat=1.15}
\usetikzlibrary{arrows}

\usepackage{circuitikz}
\usetikzlibrary{arrows.meta, 
calc, decorations.markings, positioning, shapes.geometric}
\tikzset{baseline={($ (current bounding box.west) - (0,1ex) $)}, auto}
\tikzset{vertex/.style={circle, inner sep=1.5pt, fill}, edge/.style={thick, line join=bevel}}

\usepackage{color}
\definecolor{purple}{rgb}{.5, 0, .635} 
\definecolor{yellow}{rgb}{1, 0.851, 0.188}

\definecolor{blue}{rgb}{0.133, 0.745, 0.89}
\definecolor{darkblue}{rgb}{0.08, 0.29, 0.37}

\definecolor{red}{rgb}{0.996, 0.427, 0.451}
\definecolor{darkred}{rgb}{0.52, 0.14, 0.15}

\definecolor{green}{rgb}{0.278, 0.525, 0.357}
\definecolor{darkgreen}{rgb}{0.15, 0.26, 0.16}

\usepackage{comment}

\usepackage{booktabs,tabularx,graphicx}
\newcolumntype{C}{>{\centering\arraybackslash}X}
\usepackage{graphbox}

\DeclarePairedDelimiter\norm{\lVert}{\rVert}%

\newcommand{\C}{\mathbb C}
\newcommand{\R}{\mathbb R}
\newcommand{\Z}{\mathbb Z}

\newcommand{\T}{\mathbb T}
\newcommand{\Sp}{\mathbb S}
\newcommand{\Sc}{\mathcal{S}}
\newcommand{\moon}[2]{\mathcal{P}(#1, #2)}
\newcommand{\sample}[2]{\mathcal{S}_{#1}(#2)}

\DeclareMathOperator{\MMT}{\mathcal{M}}

\theoremstyle{plain}
\newtheorem{thm}{Theorem}[section]
\newtheorem{lem}[thm]{Lemma}
\newtheorem{cor}[thm]{Corollary}

\newtheorem{prop}[thm]{Proposition}
\newtheorem{fact}[thm]{Fact}
\newtheorem{question}{Question}

\theoremstyle{definition}
\newtheorem{defn}[thm]{Definition}

\theoremstyle{remark}
\newtheorem{rem}[thm]{Remark}

\title{Curve Stitching and Dancing Planets}
\author{Frances Herr \\
\small University of Chicago \\
\small \tt herrf@uchicago.edu}
\date{}

\begin{document}

\maketitle

\begin{abstract}
    Curve stitching is a classic educational activity where one constructs elegant curves from a family of straight lines. We perform curve stitching around a circle to make a \emph{modular stitch graph}. Take $m$ points equally spaced around a circle, choose an integer multiplier $a$, and draw a chord from point $p$ to $a p \mod m$. What design will appear as the envelope of these chords? We connect these discrete objects to a continuous-time dynamical system and apply a topological perspective to understand the answer to this question.
\end{abstract}

\section{Modular Curve Stitching}\label{sec:mod-mult-tables}

Start by drawing a circle. Carefully place 12 evenly spaced points around the perimeter and label them $0$ to $11$. Now, add lines to the picture by following a simple rule: for each point $p$, draw a line between $p$ and $2p$, taking $2p$ modulo $12$ if needed. What design appears?

Increase the number of points around the circle and repeat the experiment; can you see the shape emerging? Figure $\ref{fig:2times}$ shows the result for $12$ points, $30$ points, and finally $100$ points. 

\begin{figure}[H]
    \centering
    \includegraphics[scale=0.1]{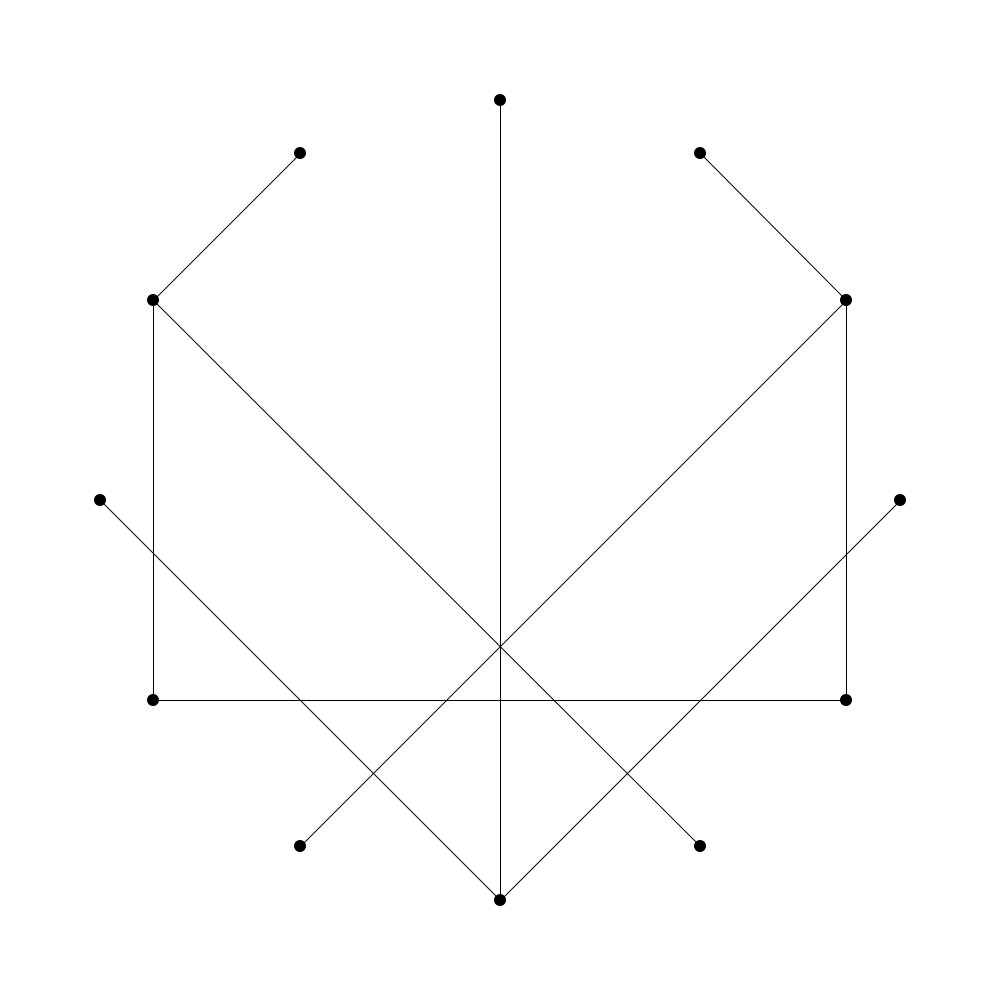}
    \includegraphics[scale=0.1]{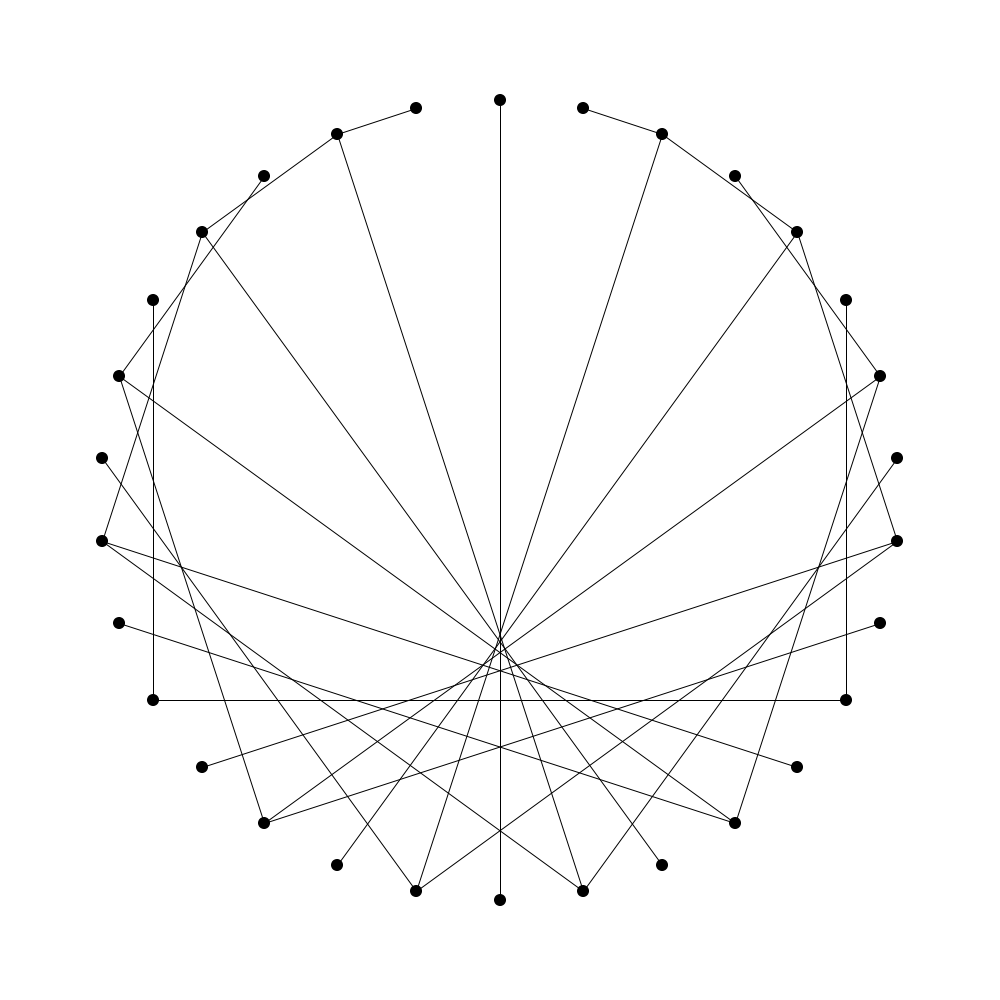}
    \includegraphics[scale=0.1]{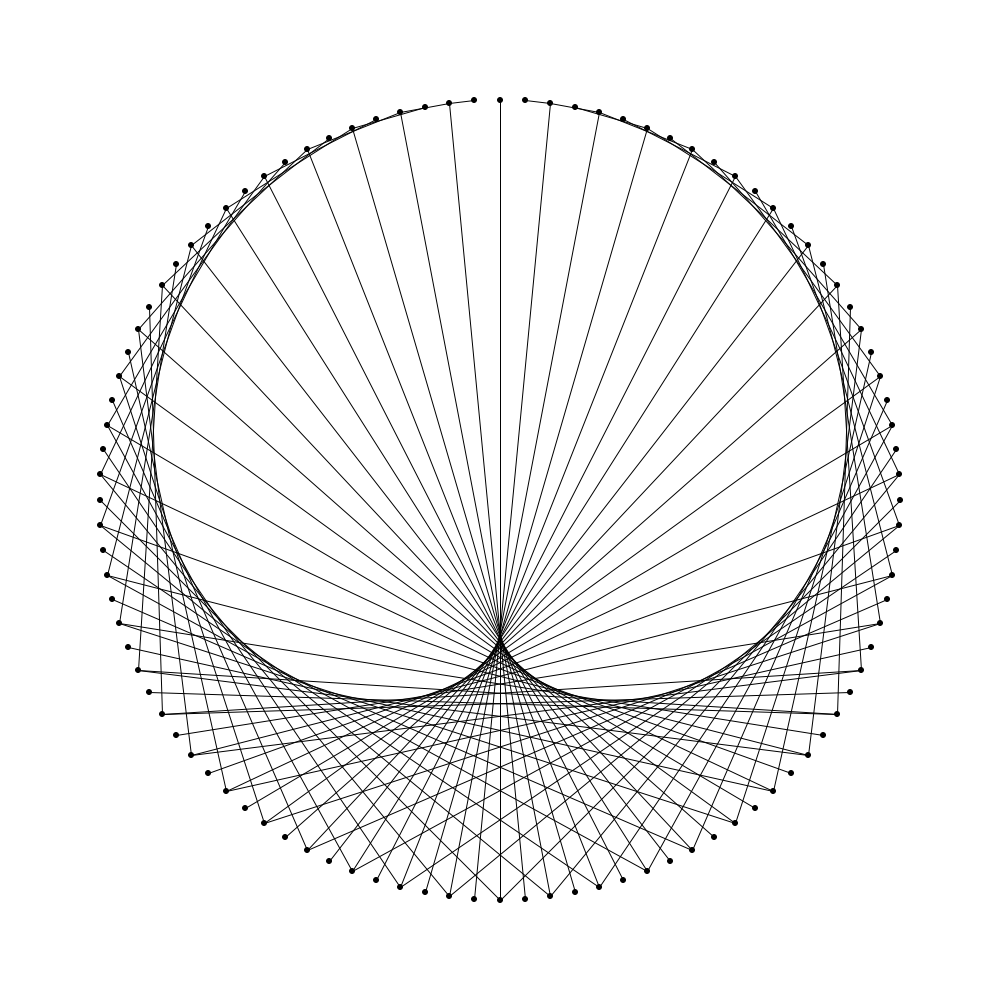}
    \caption{Increasing values for the modulus $m$ with multiplier $a = 2$. From left to right, $m = 12, \: 30, \: 100$.}
    \label{fig:2times}
\end{figure}

Mary Everest Boole experimented in this same way almost 200 years ago. Around 1840, a young Boole began lacing thread through holes in punched cards, a process she called \textit{curve stitching}.  In her 1909 book, \emph{Philosophy and Fun of Algebra} \cite{Boole}, she recommends this craft for teaching geometry to children. Thanks to Boole, in the 21$^{\text{st}}$ century geometry classroom, many children enjoy the same surprise that we saw in Figure \ref{fig:2times}: an elegant curve appears from carefully drawn straight lines!

This kind of drawing may already be familiar to the reader. Such objects appear throughout recreational mathematics and go by many names--- ``string art'', ``light caustics'', or ``spirographs'', to name a few. One popular appearance is in a YouTube video by the channel \emph{Mathologer}, ``Times Tables, Mandelbrot, and the Heart of Mathematics'' \cite{Polster}. We will call these objects \emph{modular stitch graphs} to honor Boole's influence in this history. 

Given a positive integer $m$ (the \emph{modulus}) and another integer $a$ (the \emph{multiplier}), we define the \emph{modular stitch graph} $\MMT(m,a)$ to be a discrete set of directed chords in the circle. For $m$ evenly spaced points around the circle labeled $0$ through $m-1$, connect $p$ to $ap \mod m$ with a directed chord. Then $\MMT(m,a)$ is the set of all these chords. Although $a$ can be any integer, note that $\MMT(m, a_1)$ and $\MMT(m, a_2)$ have the same set of chords if $a_1 \equiv a_2 \mod{m}$. Hence, we typically choose multipliers $a$ with $0 \leq a < m$.

\begin{quote}
    \textbf{Big Question.} \emph{Given values for $m$ and $a$, what design will $\MMT(m,a)$ create?}
\end{quote}

Before reading on, I encourage the reader to explore these objects for themselves. It only takes a little experimenting to stumble upon the zoo of designs showcased in Figure \ref{fig:examples-of-MMTs}. To generate these pictures, I recommend a tool by Mathias Lengler \cite{Lengler}. My own code to draw these objects is available online \cite{Herr}.

\begin{figure}[H]
     \centering
     \begin{subfigure}[b]{0.2\textwidth}
         \centering
         \includegraphics[width=\textwidth]{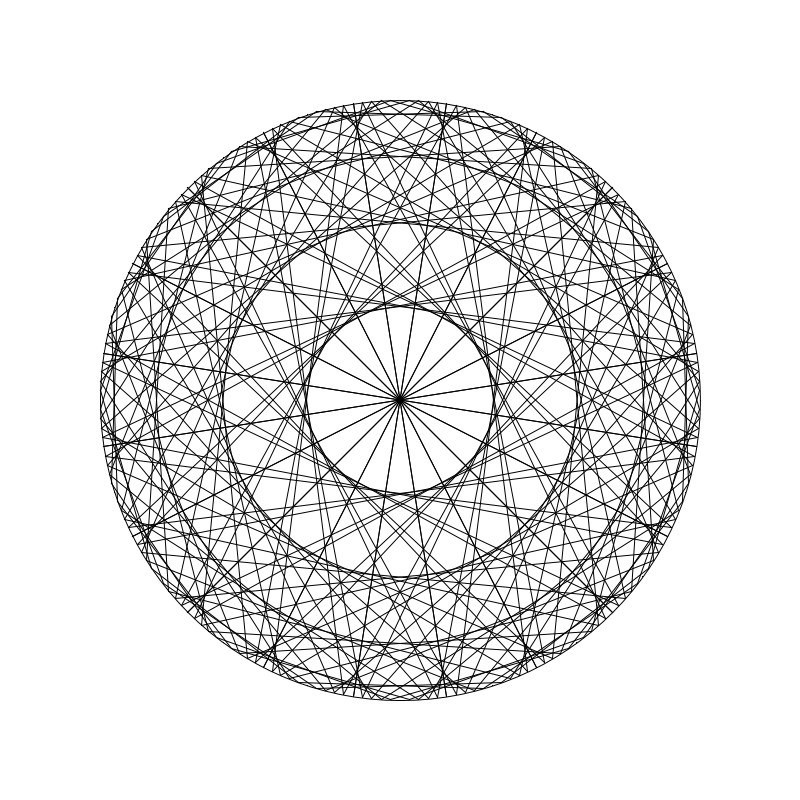}
         \caption{$\MMT(200,21)$}
         \label{fig:MMT-200-21}
     \end{subfigure}
    \hfill
     \begin{subfigure}[b]{0.2\textwidth}
         \centering
         \includegraphics[width=\textwidth]{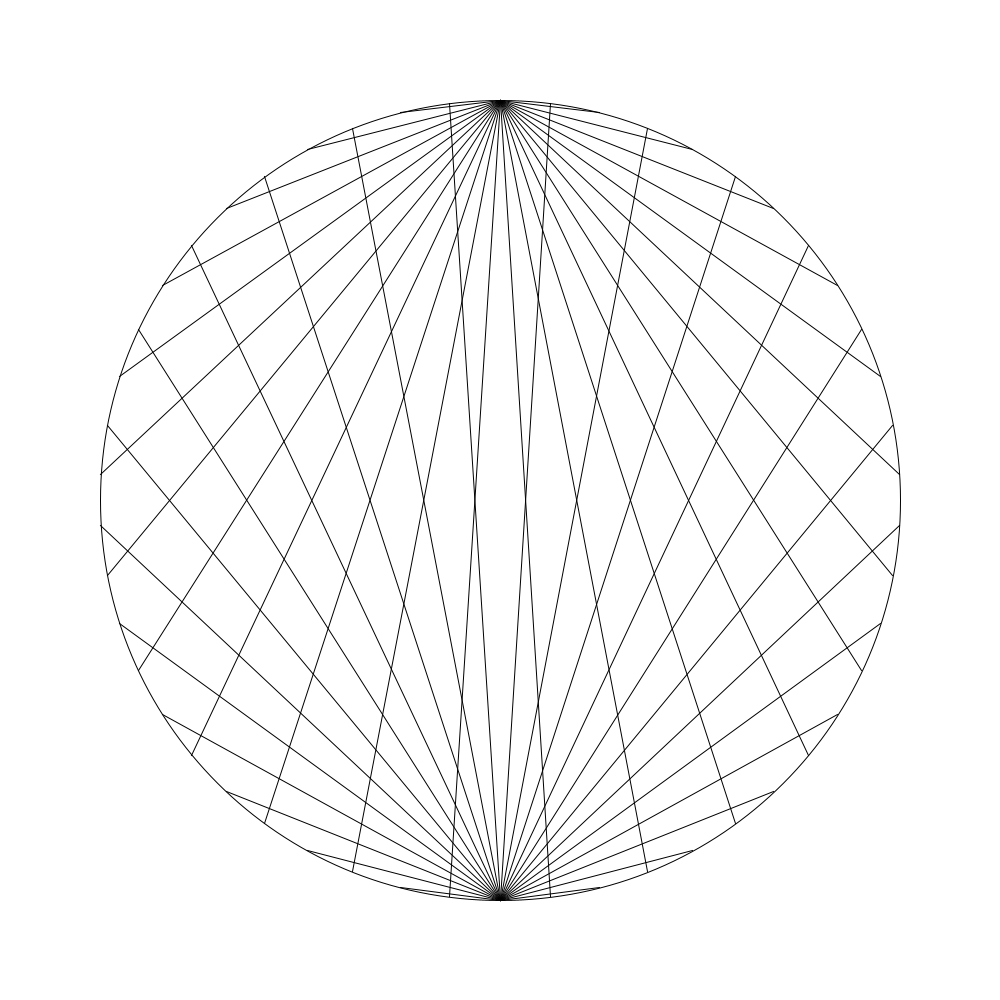}
         \caption{$\MMT(50,25)$}
         \label{fig:MMT-50-25}
     \end{subfigure}
    \hfill
     \begin{subfigure}[b]{0.2\textwidth}
         \centering
         \includegraphics[width=\textwidth]{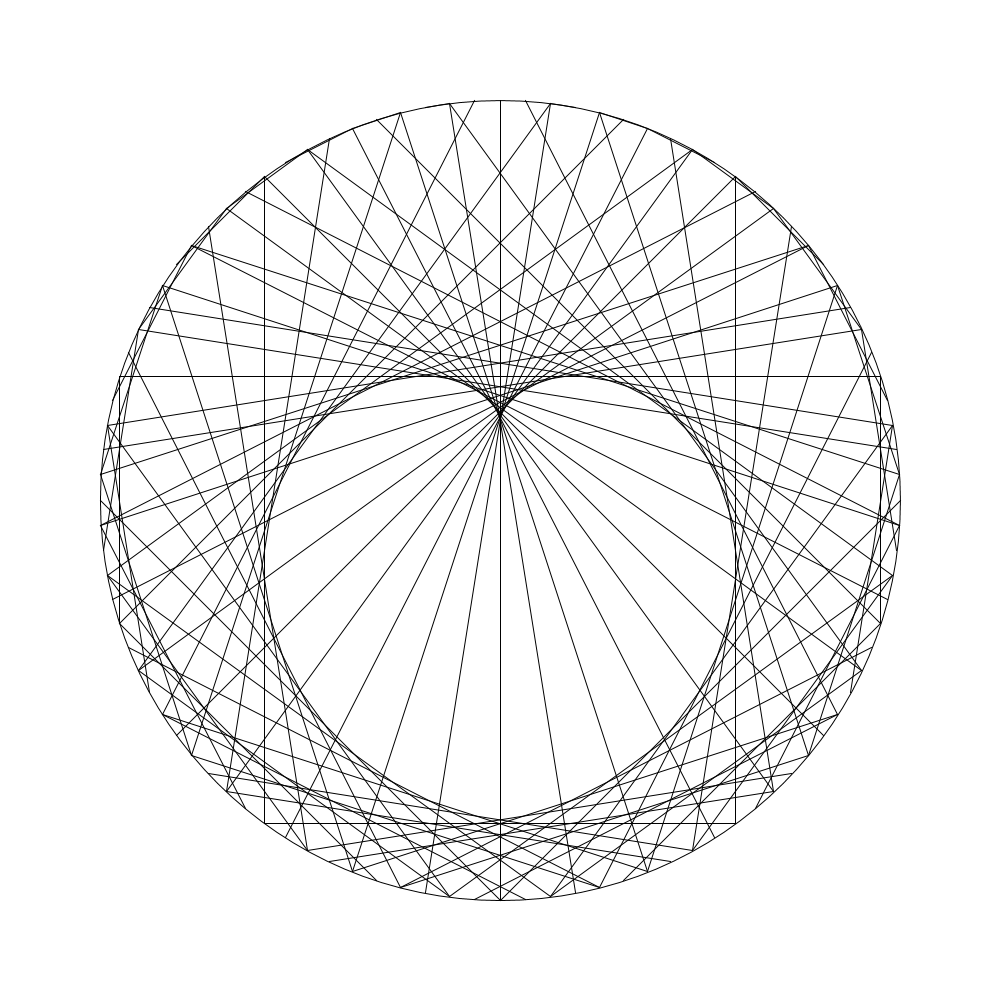}
         \caption{$\MMT(100,34)$}
         \label{fig:MMT-100-34}
     \end{subfigure}
     \hfill
     \begin{subfigure}[b]{0.2\textwidth}
         \centering
         \includegraphics[width=\textwidth]{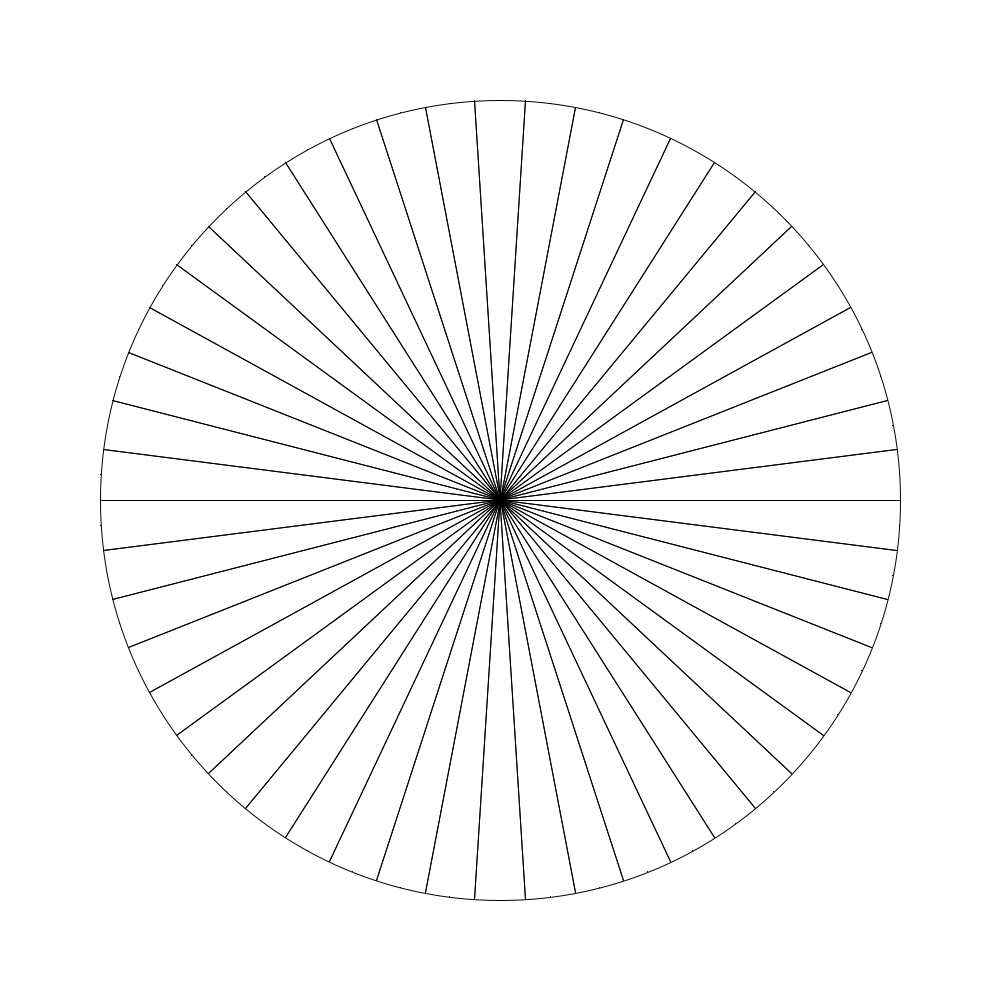}
         \caption{$\MMT(100, 51)$}
         \label{fig:MMT-100-51}
     \end{subfigure} \\
    \begin{subfigure}[b]{0.2\textwidth}
         \centering
         \includegraphics[width=\textwidth]{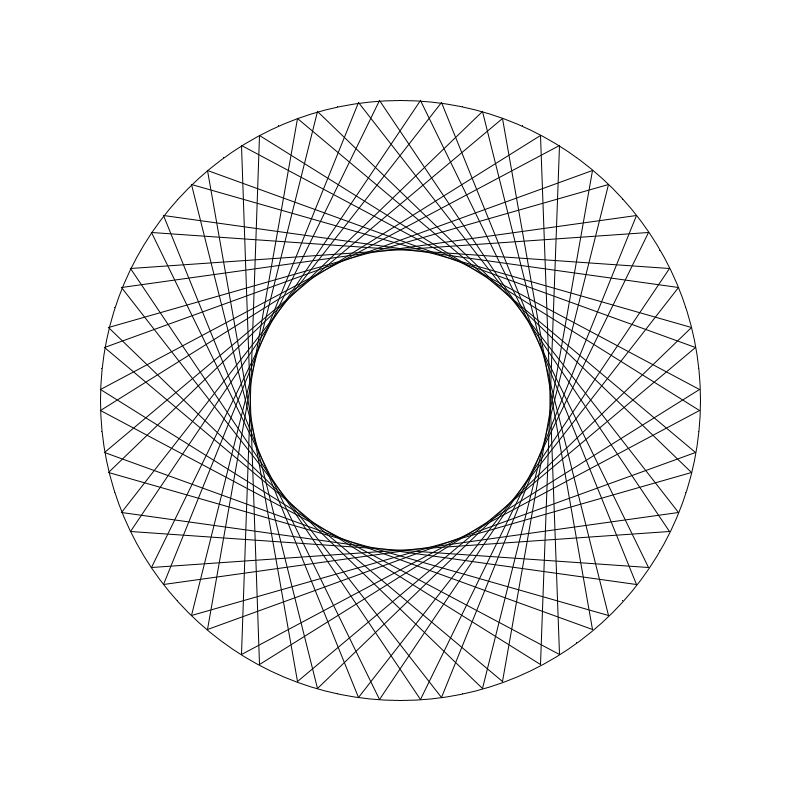}
         \caption{$\MMT(90,31)$}
         \label{fig:MMT-90-31}
     \end{subfigure}
     \hfill
     \begin{subfigure}[b]{0.2\textwidth}
         \centering
         \includegraphics[width=\textwidth]{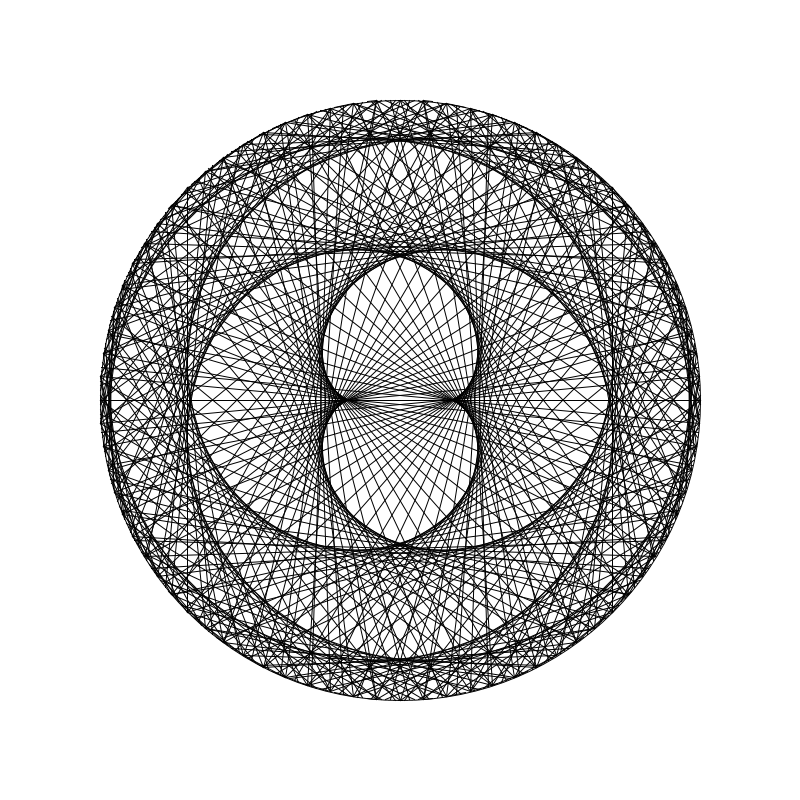}
         \caption{$\MMT(400,115)$}
         \label{fig:MMT-400-115}
    \end{subfigure}
    \hfill
     \begin{subfigure}[b]{0.2\textwidth}
         \centering
         \includegraphics[width=\textwidth]{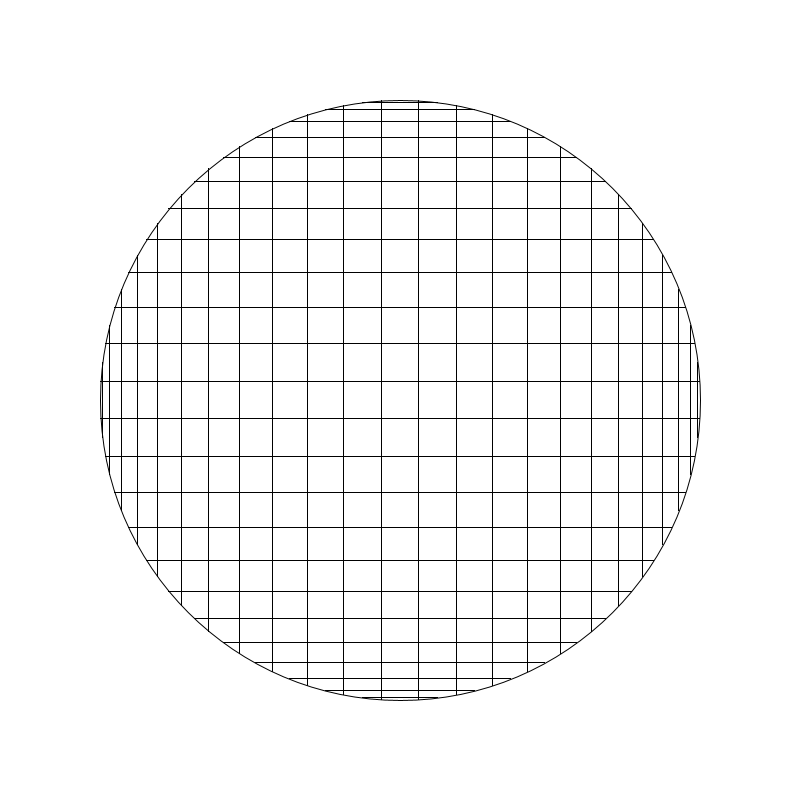}
         \caption{$\MMT(100,49)$}
         \label{fig:MMT-100-49}
     \end{subfigure}
    \hfill
         \begin{subfigure}[b]{0.2\textwidth}
         \centering
         \includegraphics[width=\textwidth]{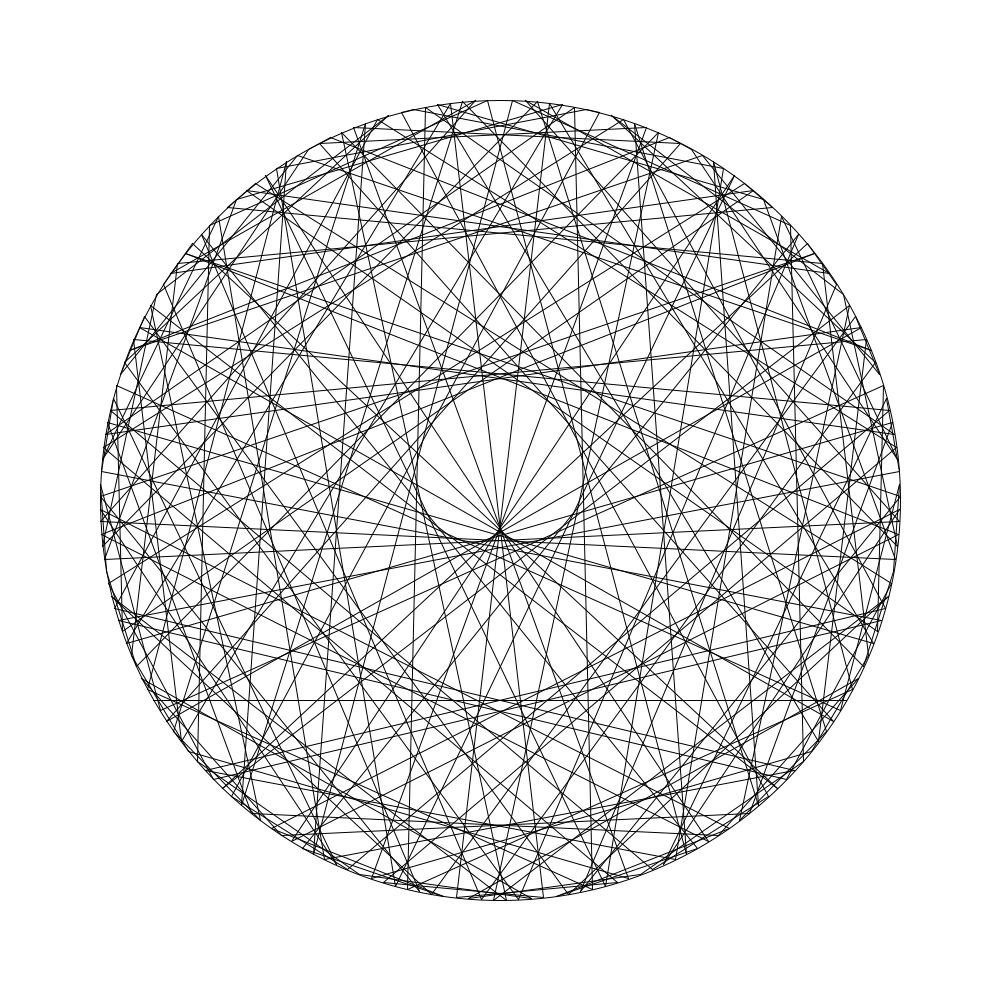}
         \caption{$\MMT(206, 21)$}
         \label{fig:MMT-206-21}
     \end{subfigure} 
        \caption{Modular stitch graphs with various values for the modulus $m$ and the multiplier $a$.}
        \label{fig:examples-of-MMTs}
\end{figure}

\subsection{Some first patterns}\label{sec:basic-patterns}

What do we mean by the ``design'' of a modular stitch graph? One candidate for a concrete mathematical notion is an \emph{envelope}.

\begin{defn}[Envelope]
    An \emph{envelope} of a collection of lines $L$ in the plane is a planar curve $C$ such that each $\ell \in L$ is tangent to $C$ at some point.
\end{defn}

By this definition, a collection of lines can possibly have many envelopes. Sometimes, an envelope curve for a modular stitch graph can strongly resemble its design. Other times, however, it is not a good representation of the graph's appearance. We can see this phenomenon with some examples.

Just like our first experiment, we will fix a small value for the multiplier $a$ and increase the modulus $m$. When $a = 2$, we see a heart-shaped curve called a \emph{cardioid} (Figure \ref{fig:2times}). This cardioid is an envelope of the chords in $\MMT(m, 2)$ for all $m$. Repeating the process with $a = 3, 4, 5, \dots$, we notice that for large modulus, $\MMT(m,a)$ resembles a curve with $a-1$ petals between $a-1$ cusps (Figure \ref{fig:fixed-m-increasing-a}). These curves are called \emph{epicycloids}; we will return to them in more detail in the next section.

\begin{figure}[H]
     \centering
      \begin{subfigure}[b]{0.2\textwidth}
         \centering
         \includegraphics[width=\textwidth]{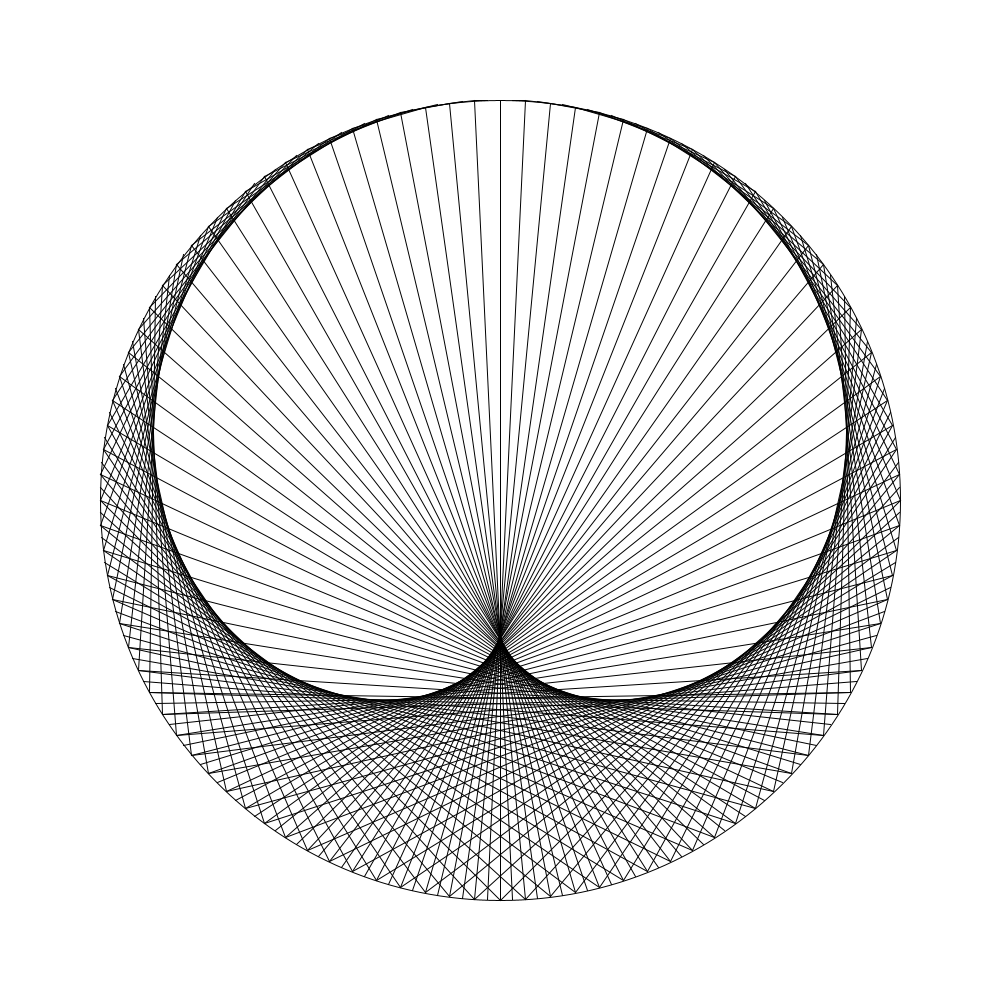}
         \caption{$\MMT(200,2)$}
         \label{fig:MMT-200-2}
     \end{subfigure}
     \begin{subfigure}[b]{0.2\textwidth}
         \centering
         \includegraphics[width=\textwidth]{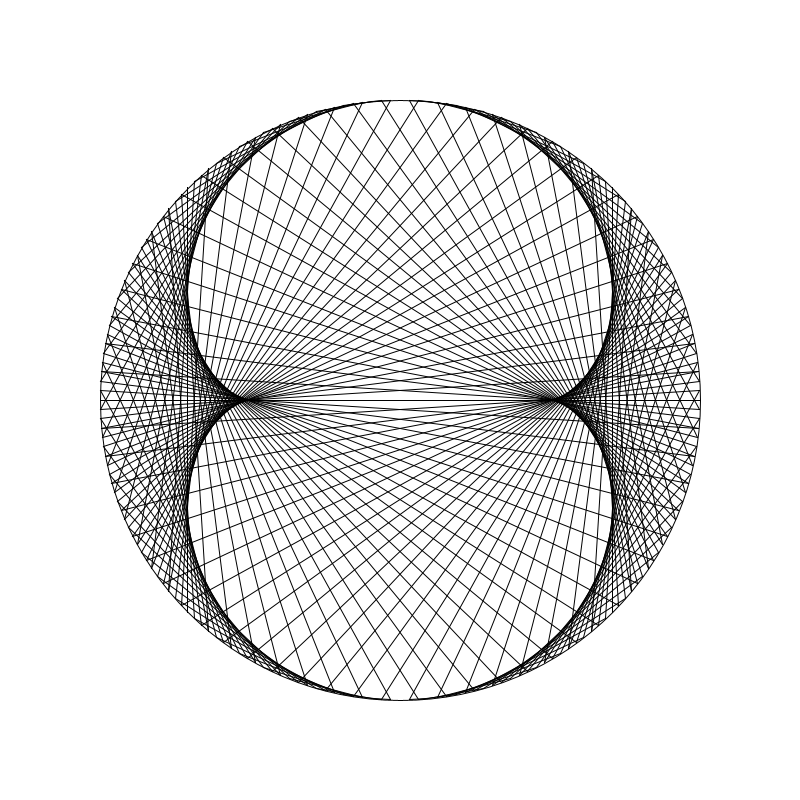}
         \caption{$\MMT(200,3)$}
         \label{fig:MMT-200-3}
     \end{subfigure}
     \begin{subfigure}[b]{0.2\textwidth}
         \centering
         \includegraphics[width=\textwidth]{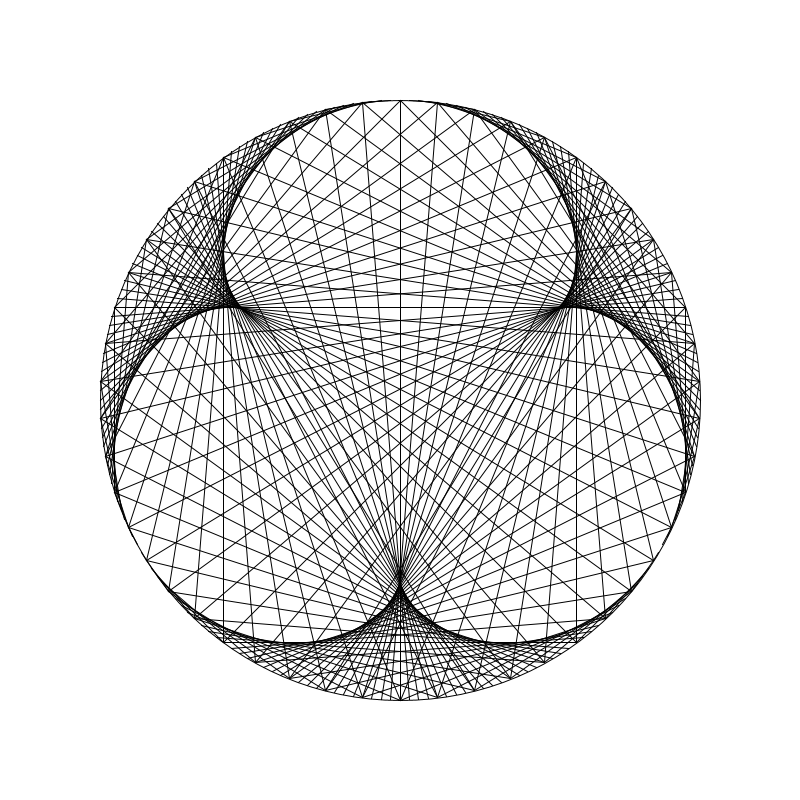}
         \caption{$\MMT(200,4)$}
         \label{fig:MMT-200-4}
     \end{subfigure}
     \\
     \begin{subfigure}[b]{0.2\textwidth}
         \centering
         \includegraphics[width=\textwidth]{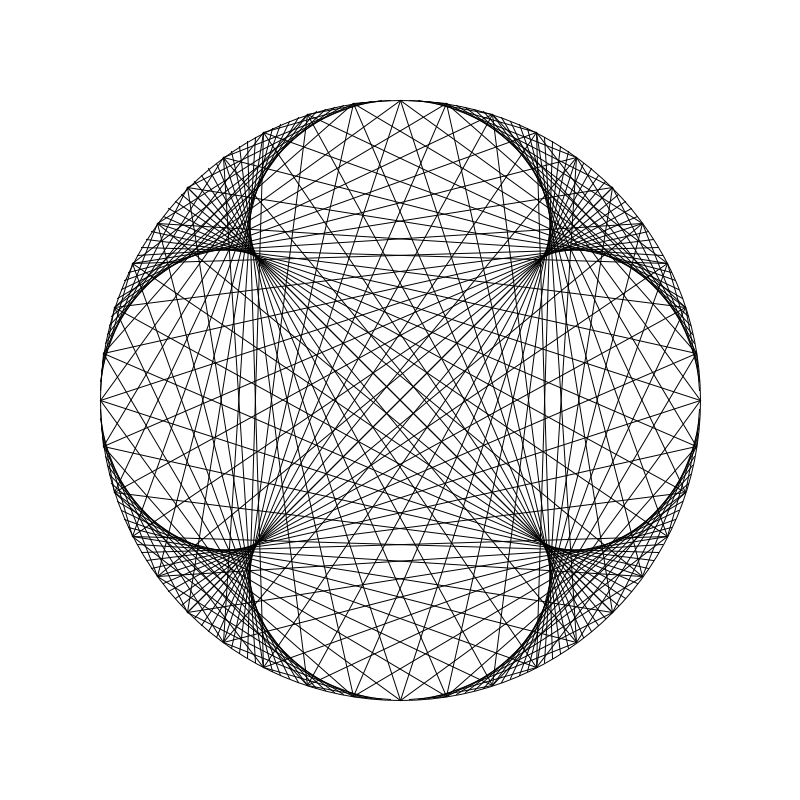}
         \caption{$\MMT(200,5)$}
         \label{fig:MMT-200-5}
     \end{subfigure}
     \begin{subfigure}[b]{0.2\textwidth}
         \centering
         \includegraphics[width=\textwidth]{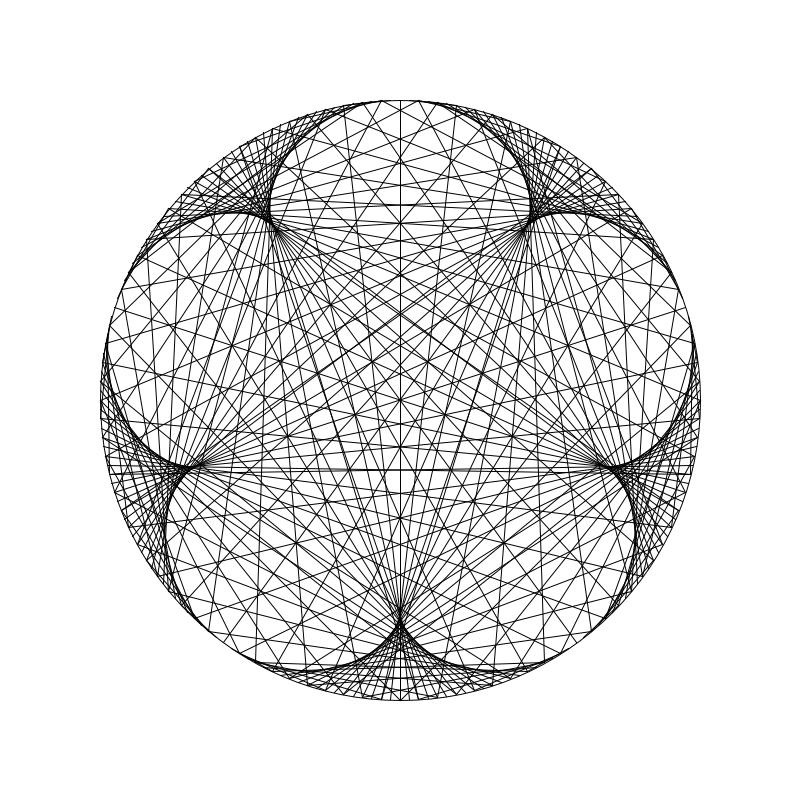}
         \caption{$\MMT(200,6)$}
         \label{fig:MMT-200-6}
     \end{subfigure}
     \begin{subfigure}[b]{0.2\textwidth}
         \centering
         \includegraphics[width=\textwidth]{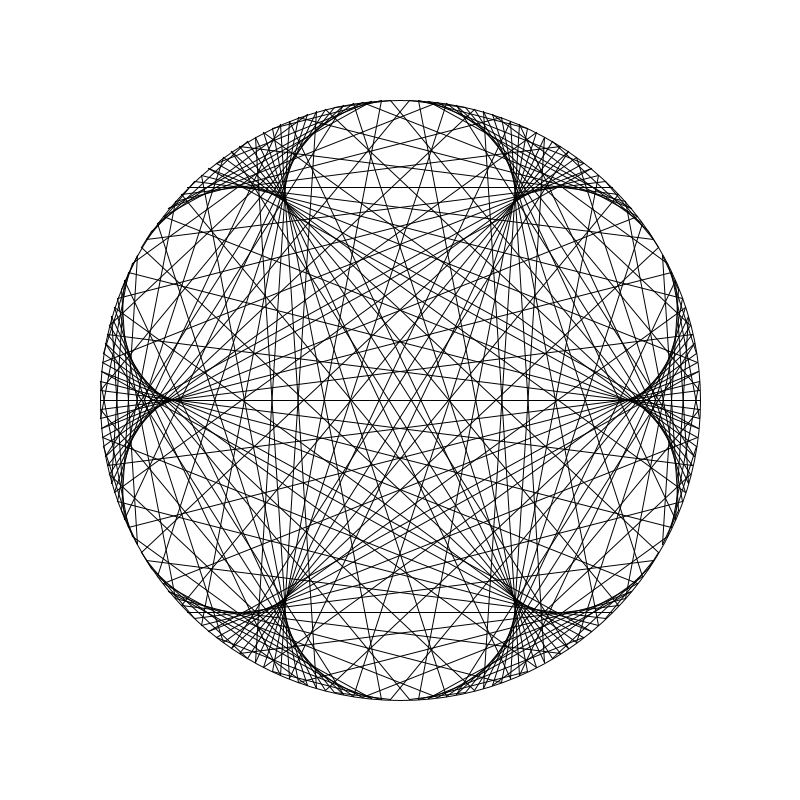}
         \caption{$\MMT(200,7)$}
         \label{fig:MMT-200-7}
     \end{subfigure}
        \caption{Modular stitch graphs with modulus $200$ and increasing multipliers.}
        \label{fig:fixed-m-increasing-a}
\end{figure}

\begin{fact}\label{fact:baby-envelopes}
     The epicycloid with $a -1$ petals is an envelope for the family of lines in $\MMT(m,a)$.
\end{fact}

This fact will follow from a more general statement given later in the paper (Proposition \ref{prop:envelope-full-statement}). Assuming this for now, we can ask a more specific version of the Big Question: \emph{For a large enough modulus $m$, is the epicycloid with $a-1$ petals always a good visual representation for the design of $\MMT(m,a)$?}

Experimenting with small multipliers (approximately less than $20$), we might be led to believe that the answer is ``yes''. But this is a hasty conclusion. Modular stitch graphs with higher multipliers introduce the variety of designs showcased in Figure~\ref{fig:examples-of-MMTs} which are much stranger than the basic curves in Figure \ref{fig:fixed-m-increasing-a}. We begin to suspect that the numerical relationship between $m$ and $a$ plays a key role in determining the design. 

\begin{question}\label{ques:families-of-tables}
    The three families of tables $\MMT(2a, a)$, $\MMT(2a -2, a)$, and $\MMT(2a + 2, a)$ each have a representative in Figure \ref{fig:examples-of-MMTs}. Can you explain why these designs appear?  
\end{question}

\begin{rem}
    Although it is often useful to think about modular stitch graphs as combinatorial objects, it will also be helpful for our purposes to define them geometrically. Parameterize the unit circle in $\C$ by $e^{2\pi it}$ for $t \in [0,1]$. Then $\MMT(m,a)$ is the set of chords with initial point $e^{2\pi i t_k}$ and terminal point $e^{2\pi iat_k}$ with $t_k = \frac{k}{m}$ for all $k \in \{0, \dots, m-1\}$.     
\end{rem}

\subsection{Envelopes}\label{sec:epicycloids-and-hypocycloids} 

An \emph{epicycloid} is a curve given by rolling one circle on the outside of another fixed circle and tracking a point on the boundary of the outer circle. To construct the epicycloid with $a -1$ petals, take the moving circle to have radius $\frac{1}{a+1}$ and the fixed circle to have radius $\frac{a -1}{a+1}$. See Figure \ref{fig:epicycloid-example}.

\begin{figure}[H]
    \centering
    \includegraphics[scale=0.4]{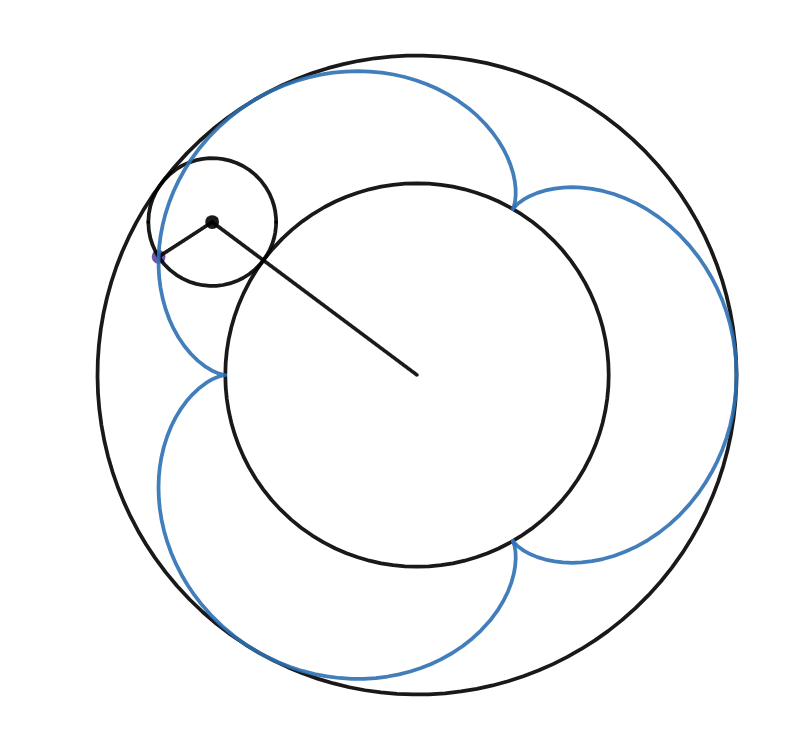}
    \includegraphics[scale=0.14]{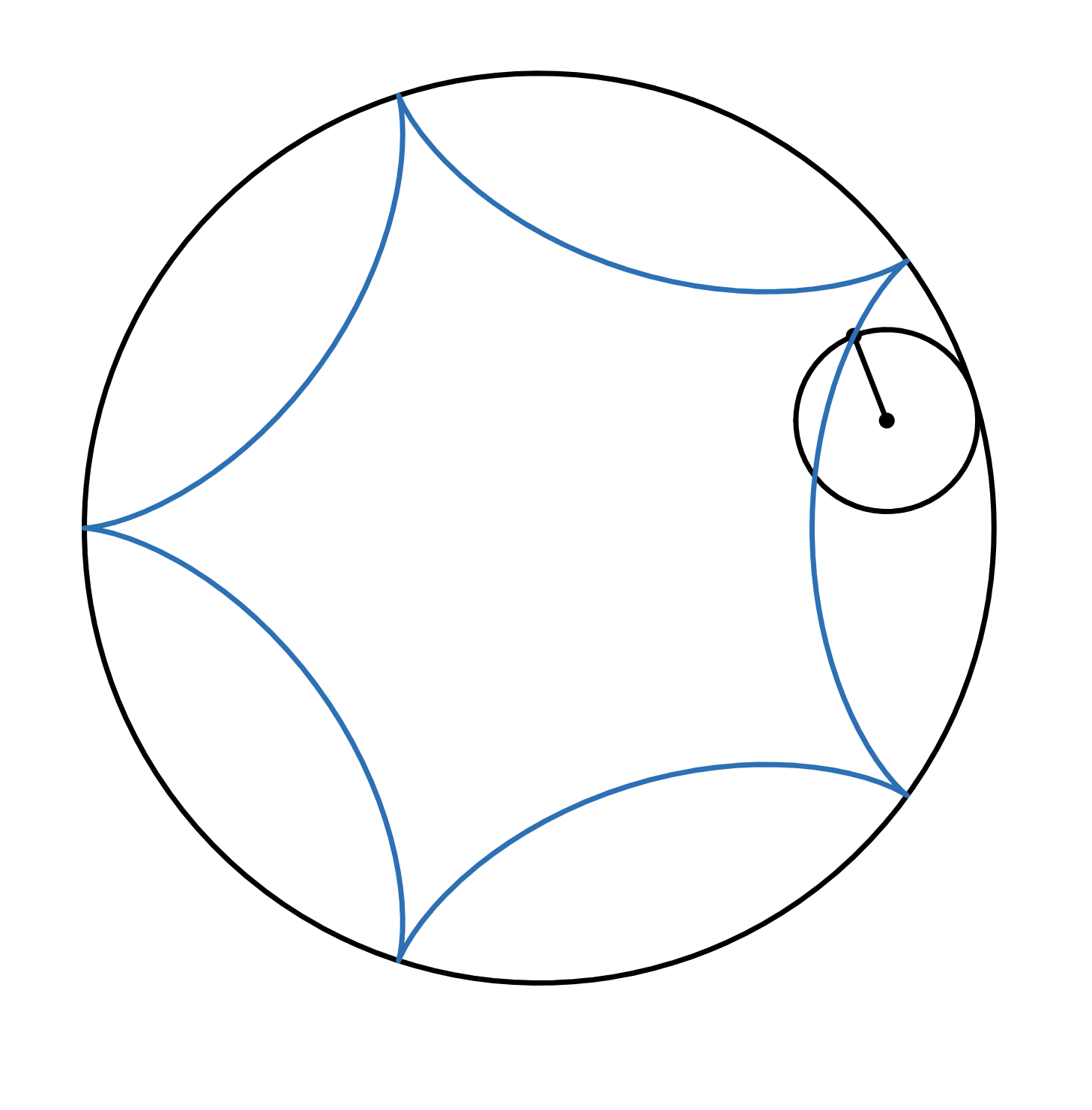}
    \caption{Constructing an epicycloid and a hypocycloid by rolling circles. (left) This epicycloid is formed by a circle of radius $\frac{1}{5}$ rolling outside a circle of radius $\frac{3}{5}$. (right) This hypocycloid is formed by a circle of radius $\frac{4}{3}$ rolling inside a circle of radius $\frac{5}{3}$.}
    \label{fig:epicycloid-example}
\end{figure}

A \emph{hypocycloid} is similar to an epicycloid but the rolling circle is on the inside of the fixed circle. We can realize both these curves with the following parametric equations.

\begin{align}\label{eq:epicycloid} 
     \begin{split}
        x(t) &= \frac{\alpha\cos(\beta t)  + \beta\cos(\alpha t)}{\alpha + \beta} \\
        y(t) &= \frac{\alpha\sin(\beta t)  + \beta\sin(\alpha t)}{\alpha + \beta}
    \end{split}  
\end{align}

If $\alpha$ and $\beta$ are both positive real numbers, then equations (\ref{eq:epicycloid}) describe an epicycloid. By a change of variables, we can assume without loss of generality that $0 < \beta \leq \alpha$. Then (\ref{eq:epicycloid}) corresponds with a circle of radius $\frac{\beta}{\alpha + \beta}$ rolling around a circle of radius $\frac{\alpha - \beta}{\alpha + \beta}$. If $\alpha > 0$ but $\beta < 0$, then equations (\ref{eq:epicycloid}) describe a hypocycloid with fixed circle radius $\frac{\alpha - \beta}{|\alpha + \beta|}$ and rolling circle radius $\frac{|\beta|}{|\alpha + \beta|}$.

In this way, we can realize epicycloids and hypocycloids using the same set of parametric equations. For our purposes, we will only be concerned with cases where $\alpha$ and $\beta$ are integers because we are interested in periodic curves.

Epicycloids and hypocycloids are the relevant envelopes for modular stitch graphs. In a quantifiable sense, understanding the design of a modular stitch graph amounts to finding the most ``natural'' epicycloid or hypocycloid which fits as an envelope. To accomplish this task, we construct a new object.

\section{Dancing Planets}
The second century Ptolemaic model for the solar system involved planets traversing ``epicycles'' over the course of their circular orbit around the Earth. Although we no longer view the universe as heliocentric, and we know that orbits are not circular, the beauty of this model endures. Even today, epicycloids from the ratios of planetary orbits make a common appearance in science-inspired art (see \cite{Henderson}). This history lends us a convenient metaphor to introduce our next mathematical object.\footnote{Other sources have used different metaphors to describe this same situation. Notably, in \cite{Simoson2000}, the author describes two runners on a circular track holding a bungee chord.}

Imagine two planets-- A and B-- on the same circular orbit where planet B is moving twice as fast as planet A. Held between them is an infinitely stretchy tether always pulled taught (see Figure \ref{fig:planet-dance}). We will imagine that A and B never collide but simply pass through each other. As the planets orbit, we watch from above and take a picture 100 times at regular intervals during planet A's orbit. The result? A picture which looks very similar to $\MMT(100, 2)$.

\begin{figure}
    \centering
    \includegraphics[width=0.5\linewidth]{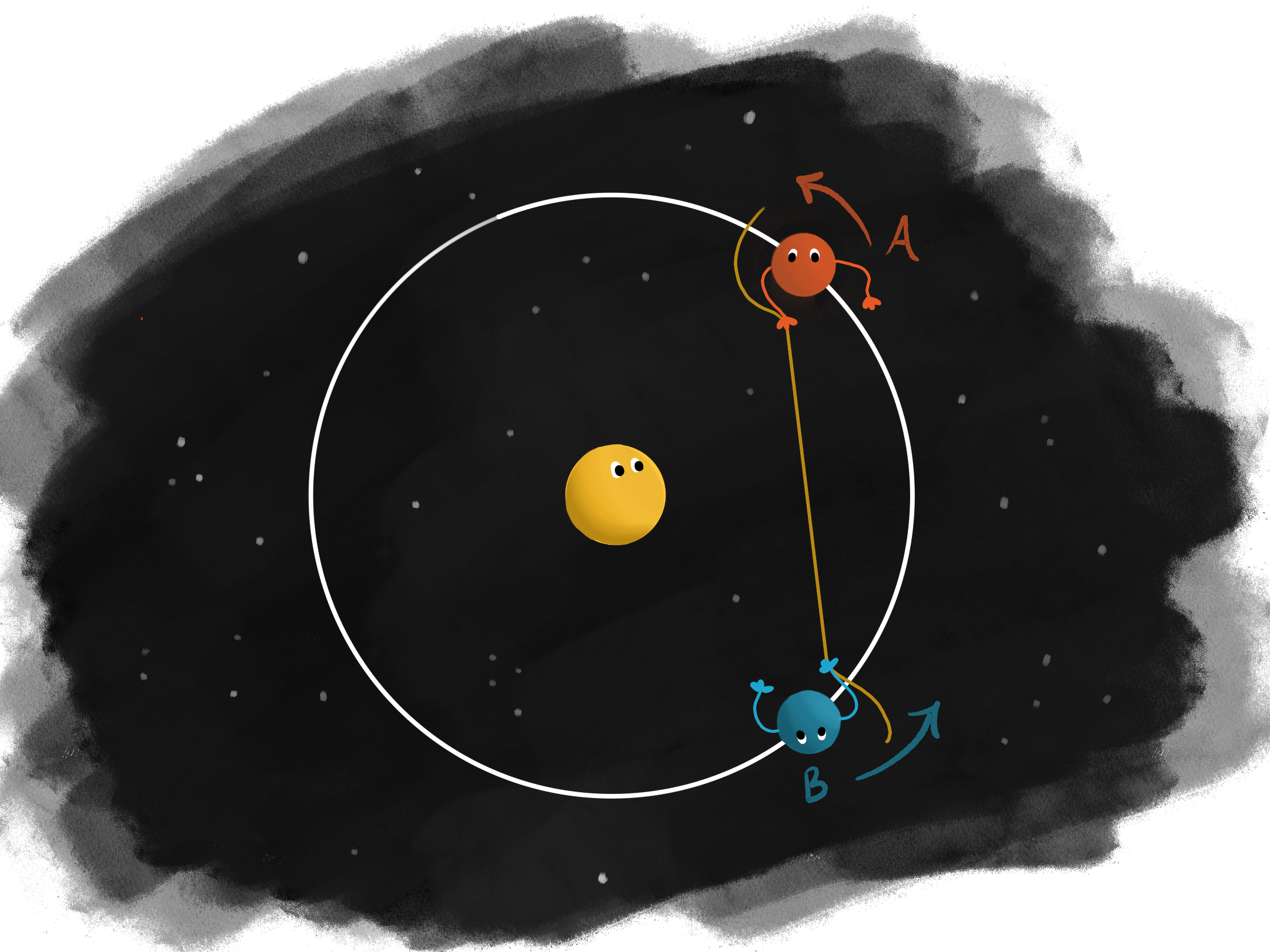}
    \caption{Two planets on the same circular orbit with a stretchy tether between.}
    \label{fig:planet-dance}
\end{figure}

\begin{defn} 
    We denote a \emph{planet dance} by $\moon{\alpha}{\beta}$ where $\alpha, \beta \in \Z$ are integers. These integers define a set of directed chords of the circle with initial point at $e^{2\pi i \alpha t}$ and terminal point at $e^{2\pi i \beta t}$ for all $t \in [0, 1]$.
\end{defn}

When the planet dance $\moon{\alpha}{\beta}$ is $\moon{0}{0}$, $\moon{0}{1}$, $\moon{1}{0}$, or $\gcd(\alpha, \beta) = 1$, we say that $\moon{\alpha}{\beta}$ is in \emph{reduced form}.

\begin{defn}\label{def:postive-negative-dances}
    If $\alpha$ and $\beta$ have the same sign, then $\moon{\alpha}{\beta}$ is a \emph{positive planet dance}. If $\alpha$ and $\beta$ have opposite signs, then $\moon{\alpha}{\beta}$ is a \emph{negative planet dance}. We will typically assume that $\alpha \geq 0$, so that the sign of the planet dance is determined by the sign of $\beta$.
\end{defn}

For a positive planet dance $\moon{\alpha}{\beta}$, the chords envelope a recognizable epicycloid. Figure~\ref{fig:moonDanceExample} shows both $\moon{3}{2}$ and the corresponding epicycloid. This curve is described by the equations (\ref{eq:epicycloid}) with $\alpha = 3$ and $\beta = 2$.

Upon first inspection, negative planet dances do not share this correspondence. However, the pattern is still there, though a bit hidden. To see the correspondence with an hypocycloid, we need to extend the chords beyond the circle. Figure \ref{fig:negative-planet-dance-ex} shows $\moon{5}{-3}$ and its corresponding hypocycloid.

\begin{figure}
     \centering
     \begin{subfigure}[b]{0.3\textwidth}
         \centering
         \includegraphics[width=\textwidth]{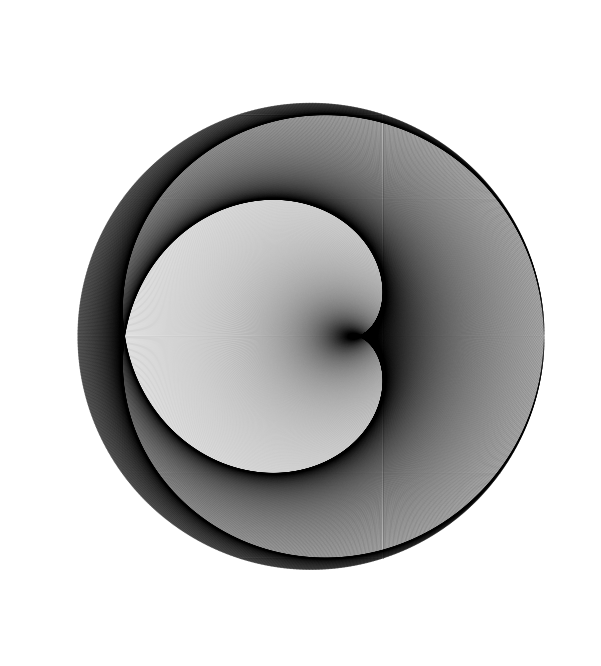}
         \caption{$\moon{3}{2}$}
         \label{fig:moonDance(3,2)}
     \end{subfigure}
     \begin{subfigure}[b]{0.3\textwidth}
         \centering
         \includegraphics[width=\textwidth]{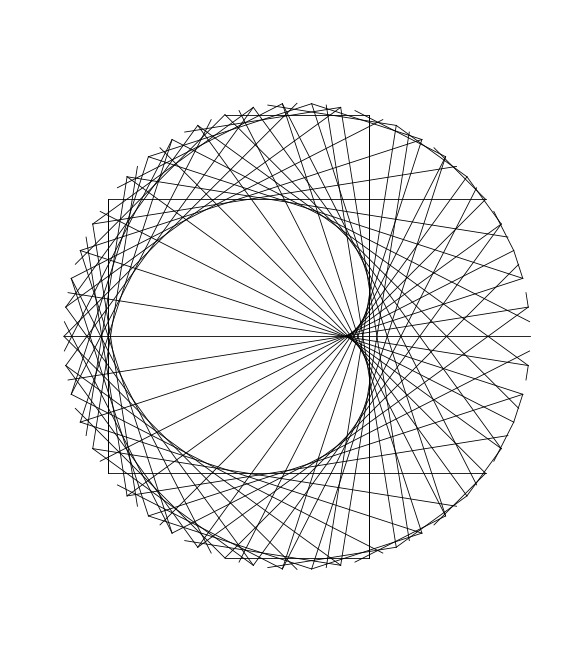}
         \caption{100-sample of $\moon{3}{2}$}
         \label{fig:moonDanceSample(3,2,100)}
     \end{subfigure}
      \begin{subfigure}[b]{0.38\textwidth}
         \centering
         \includegraphics[width=\textwidth]{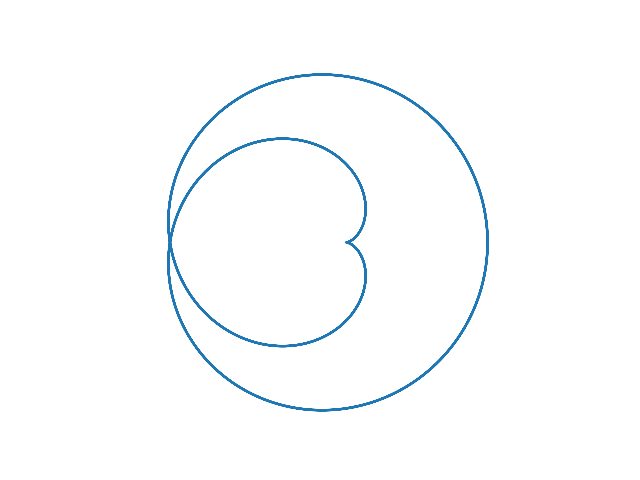}
         \caption{The epicycloid given by $\alpha = 3$ and $\beta = 2$.}
        \label{fig:epicycloid(3,2)}
     \end{subfigure}
        \caption{A planet dance, its 100-sample, and the corresponding epicycloid.}
    \label{fig:moonDanceExample}
\end{figure}

\begin{figure}
    \centering
    \includegraphics[scale=0.48]{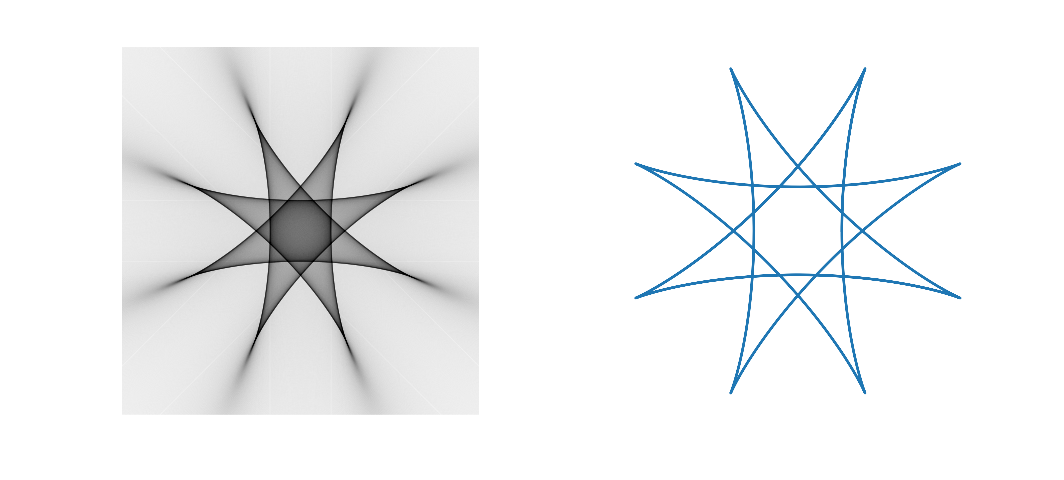}
    \caption{$\moon{5}{-3}$ and the corresponding hypocycloid}
    \label{fig:negative-planet-dance-ex}
\end{figure}

\begin{prop}[Envelope of a planet dance]\label{prop:envelope-full-statement}
    The curve (epicycloid or hypocycloid) given by equations (\ref{eq:epicycloid}) for integers $\alpha$ and $\beta$ will be an envelope for the family of lines given by $\moon{\alpha}{\beta}$.
\end{prop}

Proposition \ref{prop:envelope-full-statement} generalizes Fact \ref{fact:baby-envelopes} as will be explained in the remainder of this section. Since we want to focus on extending this result, we will not prove it here but refer to ``The Trochoid as a Tack in a Bungee Cord'' by Simoson \cite{Simoson2000} or Bouthillier's master's thesis \cite{Bouthillier} for a nice proof.

Notice that $\moon{\alpha}{\beta}$ is a \emph{continuous} family of lines. Hence, the envelope given by Proposition \ref{prop:envelope-full-statement} is unique. This stands in contrast to the many envelopes for the discrete object $\MMT(m,a)$.

\begin{defn}[$m$-sampling of a planet dance]
    Let $\moon{\alpha}{\beta}$ be a planet dance and $m$ be a positive integer. An \emph{$m$-sampling} of $\moon{\alpha}{\beta}$ is the finite set of chords for $t = 0, \frac{1}{m}, \frac{2}{m}, \dots, \frac{m-1}{m}$. We denote this set $\sample{m}{\alpha, \beta}$.
\end{defn}

A sampling of a planet dance looks rather like a modular stitch graph (see Figure \ref{fig:moonDanceExample} for an example). But does this visual correlation point to an actual mathematical correspondence?

\begin{question}\label{ques: correspondence}
    For every $\MMT(m,a)$ are there $\alpha$ and $\beta$ such that $\MMT(m,a) = \sample{m}{\alpha, \beta}$? And, for each $\sample{m}{\alpha, \beta}$, can we find $a$ such that $\sample{m}{\alpha, \beta} = \MMT(m, a)$?
\end{question}

The first part of Question \ref{ques: correspondence} has an immediate affirmative answer. The modular stitch graph $\MMT(m, a)$ is exactly the same set of chords as $\sample{m}{1, a}$. When the initial end of the chord has traveled distance $p$, the terminal end will have traveled a distance of $a p$. Thus, each chord in $\sample{m}{1, a}$ can be constructed by multiplying the initial point $p$ by $a$. We will call planet dances of the form $\moon{1}{\beta}$ \emph{integral dances}.

\begin{lem}[Fundamental Correspondence]\label{lem:fundamental-correspondence}
    The modular stitch graph $\MMT(m,a)$ is an $m$-sampling of the integral planet dance $\moon{1}{a}$.
\end{lem}

Together, Lemma \ref{lem:fundamental-correspondence} and Proposition \ref{prop:envelope-full-statement} imply Fact \ref{fact:baby-envelopes} from the previous section. A modular stitch graph is a discrete sampling of a planet dance, and the envelope for a planet dance is an envelope for its discrete sampling. However, a given modular stitch graph may be a discrete sampling of multiple planet dances, just as it might have multiple envelopes. 

The answer to the second part of Question \ref{ques: correspondence} is not so simple. If we have a planet dance $\moon{\alpha}{\beta}$ with $\alpha > 1$ and $\gcd(\alpha, \beta) = 1$, the multiplier of the corresponding modular stitch graph should be $\frac{\beta}{\alpha}$. But in our definition of modular stitch graphs, we do not allow non-integral multipliers. Of course, we could change the definition, but we will see that there is a way forward without this restructuring. 

\begin{rem}
    Other sources have taken the route of defining modular stitch graphs using non-integer multipliers. See \cite{Lengler} for an example.
\end{rem}
 
Why do we hope that we can use the current definition of modular stitch graphs to represent samplings of non-integral planet dances? Empirical evidence. Figure \ref{fig:MMT-100-34} depicts $\MMT(100, 34)$ and we notice that it looks very similar to $\moon{3}{2}$ in Figure \ref{fig:moonDance(3,2)}. In fact, any modular stitch graph of the form $\MMT(3a -2, a)$ will have this same epicycloid curve as an envelope\footnote{This example is one member in an infinite family that initiated this research project. I discuss this story and the family of graphs in Section~\ref{sec:table-of-tables}.}. What is the deeper reason that these two designs look the same? 

\section{Introducing Topology}\label{sec:a-new-perspective}

\subsection{Planet dances as paths on a torus}\label{sec:moon-systems-as-paths}

A directed chord on the circle is uniquely determined by the position of the two endpoints---planet A and planet B. Hence, the space of all such possible chords is a circle times a circle... a torus! We denote the circle by $\Sp^1$ and the torus by $\T^2$. 

\begin{figure}[H]
    \centering
    \includegraphics[width=0.3\linewidth]{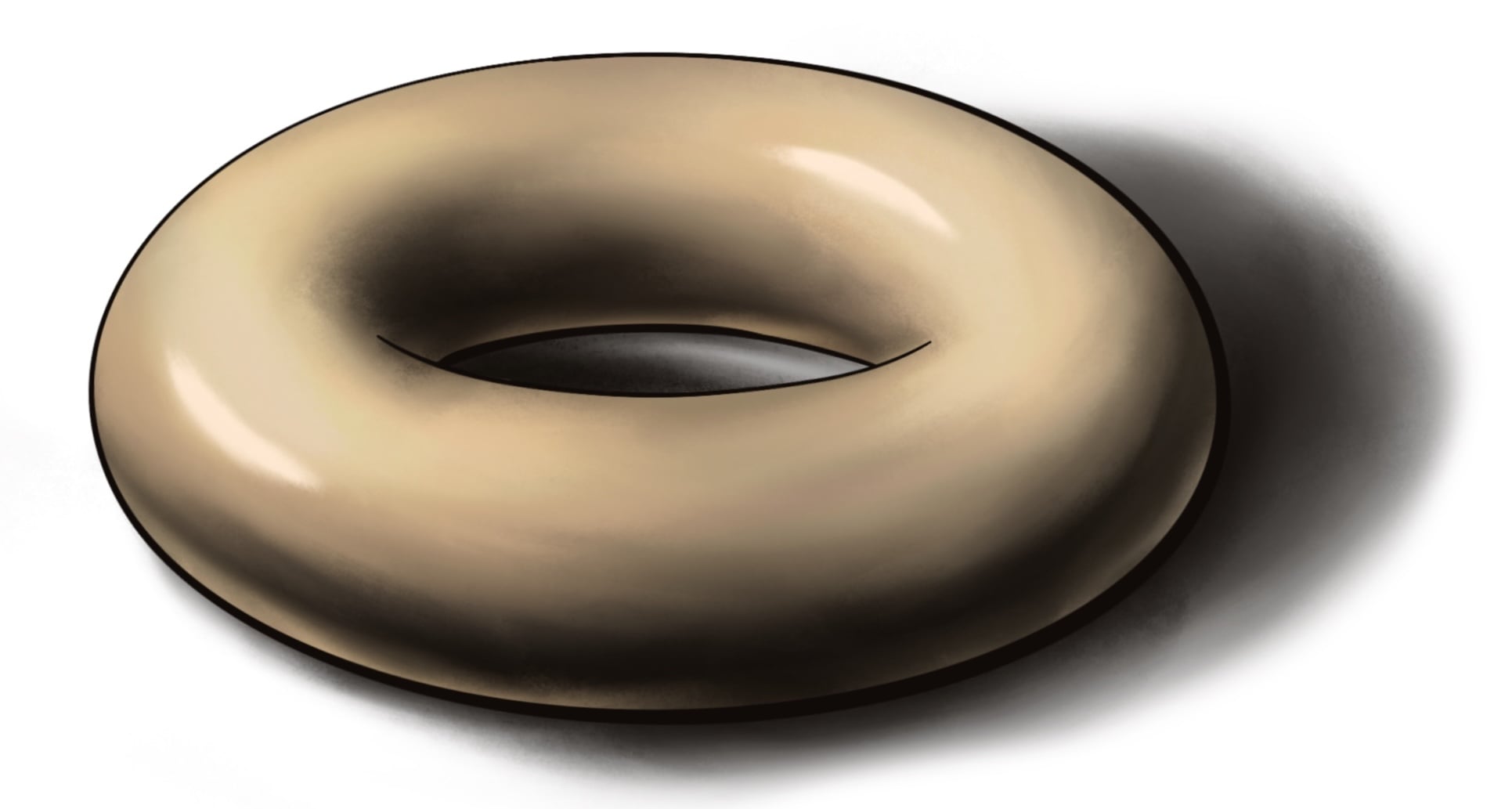}
    \caption{A standard embedding of a torus in $\R^3$.}
    \label{fig:torus}
\end{figure}

Instead of imagining the familiar doughnut shape as in Figure \ref{fig:torus}, we will be working with the flat torus.\footnote{Thanks to Elliot Kienzle for the beautiful torus drawing.} Parameterize the circle with $e^{2\pi i t}$ for $t \in [0,1]$. Then the product of two circles $\Sp^1 \times \Sp^1$ is parameterized by the unit square $[0,1] \times [0,1]$. We can then  represent $\T^2$ as the unit square in $\R^2$ with opposite sides identified or, equivalently, $\T^2 \cong \R^2 / \Z^2$.

A point on the torus gives a single chord on the circle where the coordinates determine the endpoints of the chord. A path on the torus gives a continuous family of chords. For a planet dance $\moon{\alpha}{\beta}$, we know planet A and planet B are moving at constant speeds $\alpha$ and $\beta$ respectively. So all chords in $\moon{\alpha}{\beta}$ form a linear closed path on the flat torus. These loops are commonly used to represent torus knots. Figure~\ref{fig:moon-systems-on-torus} shows two examples of these linear loops on the torus.

\begin{figure}
    \centering
    \includegraphics[scale=0.45]{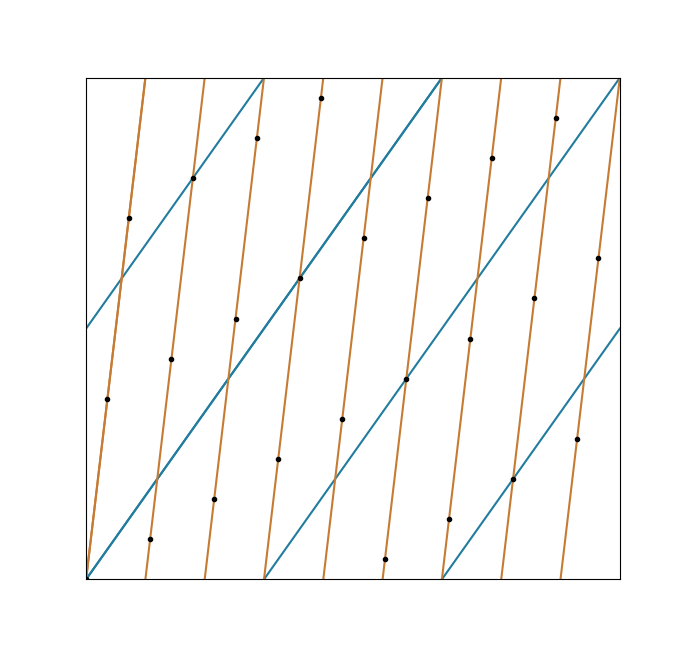}
    \caption{An example of linear loops on the torus. Identify the top and bottom edges together, and the left and right edges together with a translation. The blue lines depict $y = \frac{3}{2}x$ which corresponds with $\moon{2}{3}$ and the orange lines depict $y = 9x$ which corresponds to $\moon{1}{9}$. The black dots indicate that $\moon{1}{9}$ is sampled at a rate of $25$, corresponding with $\MMT(25,9)$.}
    \label{fig:moon-systems-on-torus}
\end{figure}

The line in $\R^2/\Z^2$ which corresponds to $\moon{\alpha}{\beta}$ is given by
\[x = \alpha t, \hspace{10pt} y = \beta t, \hspace{10pt} \text{ or, if } \alpha \neq 0, \hspace{10pt} y = \tfrac{\beta}{\alpha}x.\]
We will identify the planet dance $\moon{\alpha}{\beta}$ with this line on the torus, and depending on the context, refer to this line as $\moon{\alpha}{\beta}$.

\begin{rem}[On direction]
     Swapping the values of $\alpha$ and $\beta$ in a planet dance will not make a difference in the picture, but it does change the direction of each chord and change the slope of the corresponding line in $\R^2/\Z^2$. Moving forward, we will maintain the direction of chords as this inherently matters for modular stitch graphs. In a modular stitch graph, we think of the initial point of the chord being multiplied by $a$ to get the terminal point of the chord. That is, we would like $\MMT(m,a)$ to correspond with $\sample{m}{1, a}$ and not $\sample{m}{a, 1}$.
\end{rem}

If a planet dance is a linear loop on the torus, then a sampling of the planet dance is a set of equally spaced points along the loop. As the sampling becomes more frequent, the set of points more closely approximates the loop and the discrete sampling picture becomes a better representation of the continuous-time planet dance. 

Not only do planet dances and sampled planet dances have a geometric interpretation, but they also have an algebraic one. We can write $\moon{\alpha}{\beta}$ as a collection of points in $\R^2/\Z^2$.
\[\moon{\alpha}{\beta} = \left\{\left(\alpha t, \beta t\right) \: | \: t \in \R\right\}\]
This set forms a rank-one continuous subgroup of $\R^2/\Z^2$ with component-wise addition.
\[(\alpha t, \beta t) + (\alpha s, \beta s) = \left(\alpha(t + s), \: \beta(t + s)\right) \in \moon{\alpha}{\beta}\]
Likewise, $\sample{m}{\alpha, \beta}$ is a rank-one discrete subgroup of $\R^2/\Z^2$. We will use this group structure in the subsequent sections.

\subsection{Aliasing in planet dances}\label{sec:planet-dance-aliasing}
We have observed that sometimes samplings of planet dances can appear in disguise. Like how $\MMT(100, 34) = \sample{m}{1, 34}$ looks much more like $\moon{3}{2}$ than $\moon{1}{34}$. It's possible the reader has seen this kind of phenomena before.

\emph{Aliasing} is a pervasive topic in signal processing. The idea is best demonstrated with sine waves. Suppose that a computer is storing discrete-time samples of a musical note, represented as a sine wave. If the sampling rate is too low, then when the computer plays back the note from its discrete samples, the listener hears it as a lower note, i.e. with a lower frequency. This happens because there is a more ``natural'' wave that fits the samples. The two waves---the original sine wave and the lower frequency alias---intersect exactly at the sample points. An example is shown in Figure ~\ref{fig:sine-waves}. Common examples of aliasing are Moir\'e patterns and the ``wagon wheel'' effect in cinema. In signal processing, aliasing is often seen as an unfortunate reality, and research into ``anti-aliasing filters'' is common. However, this phenomenon is why we see such a large variety of designs among modular stitch graphs. This will be the focus for the rest of the paper.

\begin{figure}
    \centering
    \includegraphics[scale=1.3]{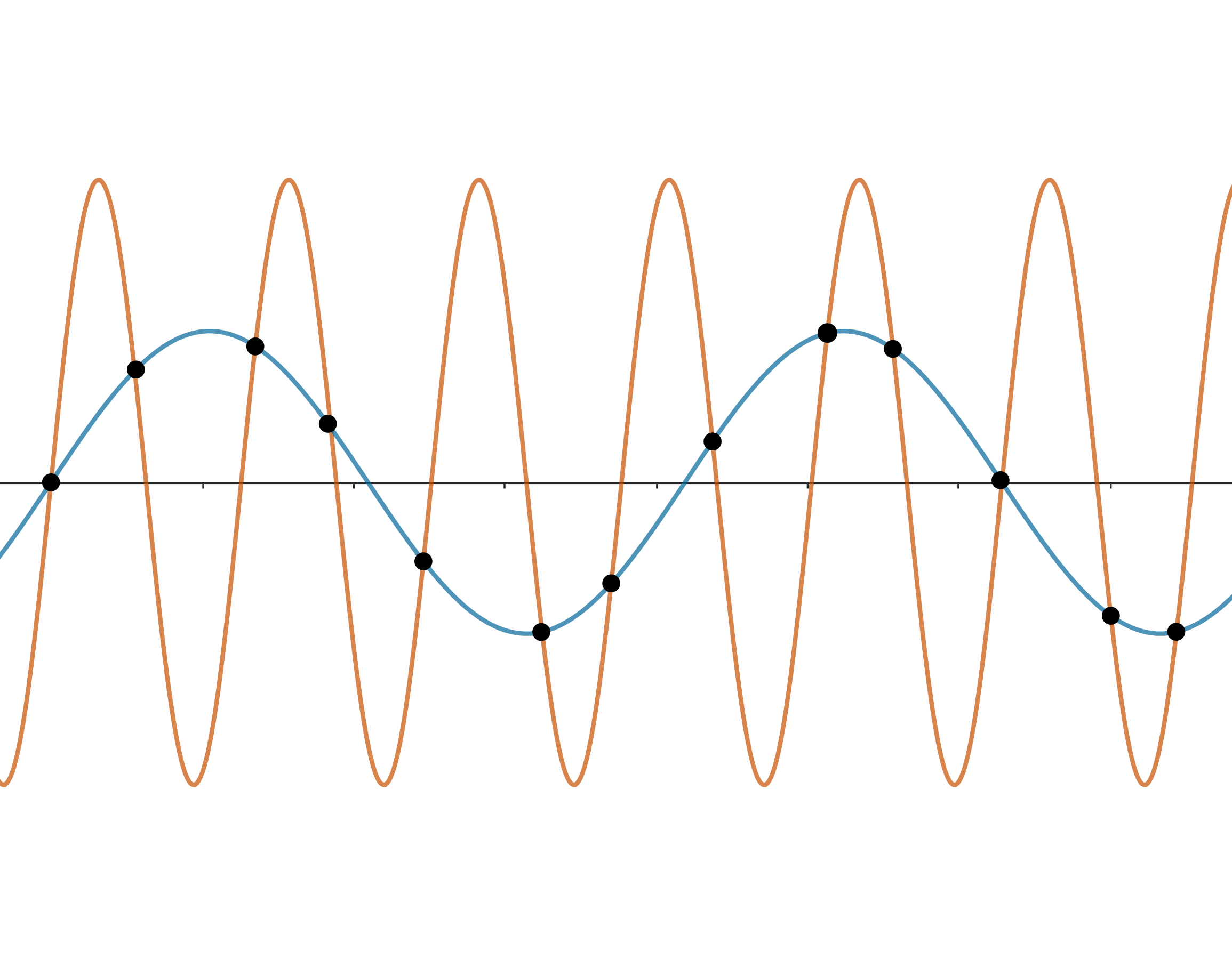}
    \caption{When a higher frequency sine wave is under-sampled, it can appear as a lower frequency sine wave where the sampling points are exactly the points of intersection.}
    \label{fig:sine-waves}
\end{figure}

Just as one sine wave can alias another when we take a discrete sampling, one planet dance can alias another. The key to this aliasing is the points of intersection. 

Suppose that the lines corresponding to planet dances $\mathcal{P}_1$ and $\mathcal{P}_2$ intersect $m$ times in $\R^2/\Z^2$. These intersection points occur at regular intervals because both lines have constant rational slopes. Then the $m$-sampling of $\mathcal{P}_1$ and the $m$-sampling of $\mathcal{P}_2$ are the same set of chords and will look identical. Since we know how to achieve $m$-samplings of \textit{integral} planet dances with modular stitch graphs, this gives us a way to sample \textit{any} planet dance with a modular stitch graph by finding an integral alias. 

\begin{lem}[Planet dance intersection]\label{lem:torus-knot-intersection}
    Let lines $\ell_1$ and $\ell_2$ in $\R^2/\Z^2$ be given by
    \[\ell_1(t) = (\alpha t, \beta t) \hspace{10pt} \text{and} \hspace{10pt} \ell_2(t) = (\gamma t, \delta t)\]
    such that $\gcd(\alpha, \beta) = \gcd(\gamma, \delta) = 1$. Then $\ell_1$ and $\ell_2$ will intersect $|\alpha\delta - \beta \gamma|$ times and will do so at regular intervals.
\end{lem}

This is a standard result often applied to intersection numbers of torus knots in $\T^2$. The proof follows by doing a change of basis to send the vectors $(\alpha, \beta)$ and $(\gamma, \delta)$ to $(1,0)$ and $(0,1)$ respectively. Then one can apply Pick's Theorem to relate the number of integer lattice points inside the relevant parallelogram to the determinant of the change of basis matrix. By restating Lemma \ref{lem:torus-knot-intersection}, we can relate non-integral planet dances to modular stitch graphs.

\begin{cor}\label{cor:aliasing-planet-dances}
    Suppose that $\moon{\alpha}{\beta}$ and $\moon{\gamma}{\delta}$ are two planet dances in reduced form and let $m = |\alpha\delta - \beta \gamma|$. Then $\sample{m}{\alpha, \beta} = \sample{m}{\gamma, \delta}$. 
\end{cor}

\begin{prop}\label{prop:aliasing-planet-dances}
    Given a planet dance $\moon{\alpha}{\beta}$ in reduced form and $a \in \Z$, let $m = |\alpha a - \beta|$. Then the modular stitch graph $\MMT(m, a)$ will be an $m$-sampling of $\moon{\alpha}{\beta}$.
\end{prop}

Although we arrived at Proposition \ref{prop:aliasing-planet-dances} geometrically, there is also an algebraic way to view this result. Consider that for $\alpha, \beta, a$, and $m$ as above, we have the relation $\beta = \alpha a \pm m$. Then we note that $\sample{m}{\alpha, \alpha a \pm m} = \sample{m}{\alpha, \alpha a}$. This follows from the definition of a sampling. We can then rephrase Proposition \ref{prop:aliasing-planet-dances} as $\sample{m}{1, a} = \sample{m}{\alpha, \alpha a}$. 

Note that this is only true because $\gcd(\alpha, m) = 1$. To see this, write $\sample{m}{1, a}$ and $\sample{m}{\alpha, \alpha a}$ as subsets of $\T^2$.
\[\sample{m}{1, a} = \left\{ \left(\frac{k}{m}, \frac{ak}{m}\right) \: \Big| \:0 \leq k < m\right\} \hspace{30pt} \sample{m}{\alpha, \alpha a} = \left\{ \left(\frac{\alpha k}{m}, \frac{a(\alpha k)}{m}\right) \: \Big| \:0 \leq k < m\right\}\]
In order for these to be the same set of points, we need that the map $k \mapsto \alpha k$ on $\Z/m\Z$ is an isomorphism of the group. This happens if and only if $\alpha$ is invertible in $\Z/m\Z$. There are then $|(\Z/m\Z)^{\times}|$ possible choices for $\alpha$ given $m$ and $a$. 

This brings to light a way to describe the most ``visually natural'' envelope for a set of chords $\sample{m}{1, a}$. We use the metric on $\T^2$ inherited from $\R^2$ to determine the sample point $(\frac{\alpha}{m}, \frac{a\alpha}{m})$ closest to the origin. If $\gcd(\alpha, m) = 1$, then the line in $\R^2/\Z^2$ connecting $(0,0)$ to $(\frac{\alpha}{m}, \frac{a\alpha}{m})$ will pass through all the sample points. This determines a planet dance alias for $\sample{m}{1,a}$. If $\gcd(\alpha, m) = d > 1$, then we have some interesting behavior which will be explained in Section \ref{sec:table-of-tables}.

To address an inverse of Proposition \ref{prop:aliasing-planet-dances}, a sampling of a planet dance $\sample{m}{\alpha, \beta}$ can be realized as a modular stitch graph exactly when $m = |\alpha a - \beta|$
for some positive integer $a$. This has a solution if and only if $m \equiv \pm \beta \mod \alpha$. If we start with a general sampled planet dance $\sample{m}{\alpha, \beta}$, we may not be able to represent it as a modular stitch graph with the exact sample rate $m$. We can instead choose the nearest $m'$ to $m$ such that $m' \equiv \pm \beta \mod \alpha$ and then we can realize $\sample{m'}{\alpha, \beta}$ as an modular stitch graph with multiplier $m'$. The size of $\alpha$ will determine how far we may need to roam to find a solution.

\subsection{Returning to our questions}

Looking back at Question \ref{ques: correspondence}, we now have an answer. Given any planet dance $\moon{\alpha}{\beta}$, we can find a corresponding family of modular stitch graphs which are all samplings of $\moon{\alpha}{\beta}$. But given a specific sampling $\sample{m}{\alpha, \beta}$, we may not be able to realize it as a modular stitch graph. This also reveals why certain modular stitch graphs exhibit designs that are very different from the integral planet dance given by the Fundamental Correspondence Lemma \ref{lem:fundamental-correspondence}. When $\MMT(m,a)$ is viewed as an $m$-sampling of $\moon{1}{a}$, if $m$ is an ``under-sampling'' then $\MMT(m,a)$ will alias as a sampling of another planet dance. But this correspondence is only the tip of the iceberg.

\section{Overlaying Designs}\label{sec:table-of-tables}

We now return to the patterns in modular stitch graphs that we noticed in Section \ref{sec:mod-mult-tables}. Assuming that $m$ is even, we observed that graphs of the form $\MMT(m, \frac{m}{2})$ or $\MMT(m, \frac{m}{2} + 1)$ or $\MMT(m, \frac{m}{2} - 1)$ each display a distinct design which holds for all such $m$ (recall Question \ref{ques:families-of-tables}). A natural generalization is to replace `2' with any natural number $b$. That is, supposing that $m$ is divisible by $b$, we consider three classes of graphs:\footnote{One might want to describe these families as $\MMT(ba, a)$ and $\MMT(ba \pm b, a)$ to avoid the assumption that $b | m$. However, in order to keep the narrative of discovery coherent, we will keep the above perspective for now.}
\[\MMT\left(m, \tfrac{m}{b}\right), \hspace{20pt} \MMT\left(m, \tfrac{m}{b} + 1\right), \hspace{10pt} \text{and} \hspace{10pt} \MMT\left(m, \tfrac{m}{b} - 1\right).\]
The designs for each of these classes are generalizations of those when $b = 2$. Examples of these graphs are shown in Figure~\ref{fig:m-div-b-examples}.

\begin{figure}
    \centering
    \begin{subfigure}[b]{0.2\textwidth}
         \centering
         \includegraphics[width=\textwidth]{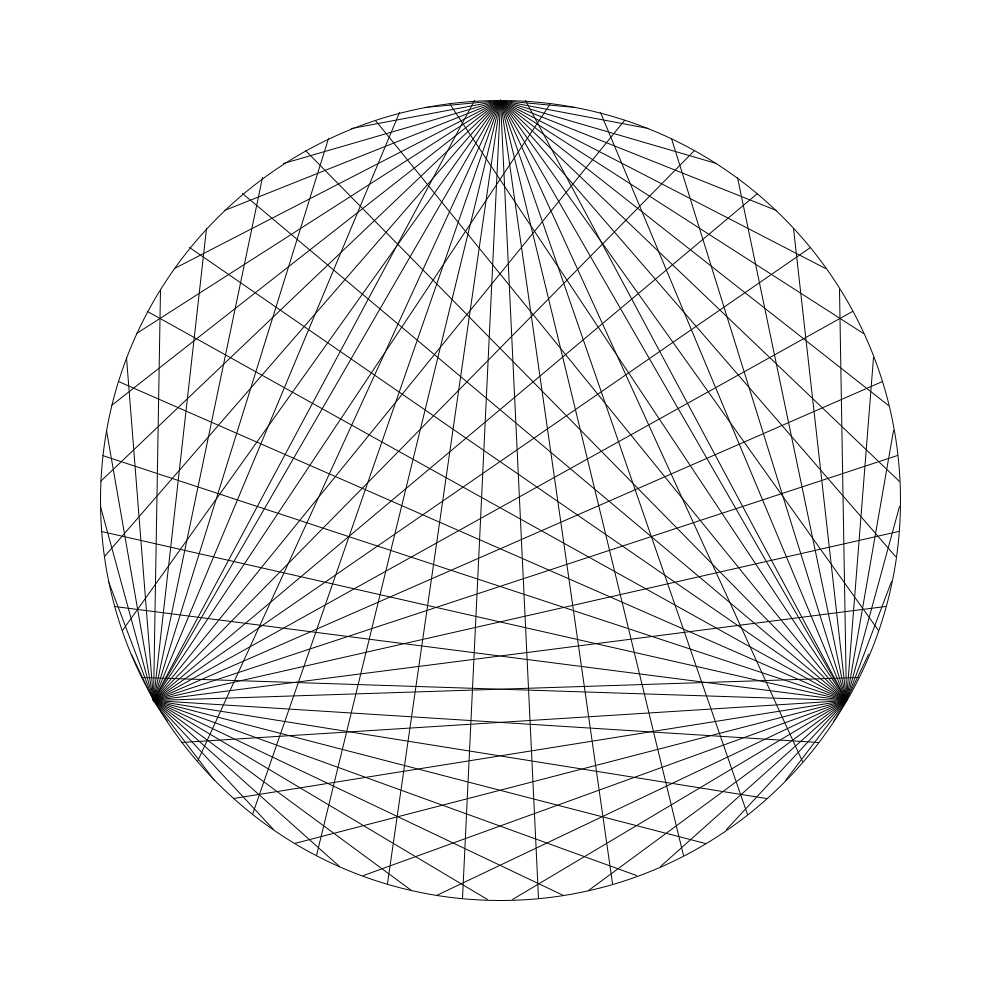}
         \caption{$\MMT(99,33)$}
         \label{fig:MMT-99-33}
     \end{subfigure}
     \begin{subfigure}[b]{0.2\textwidth}
         \centering
         \includegraphics[width=\textwidth]{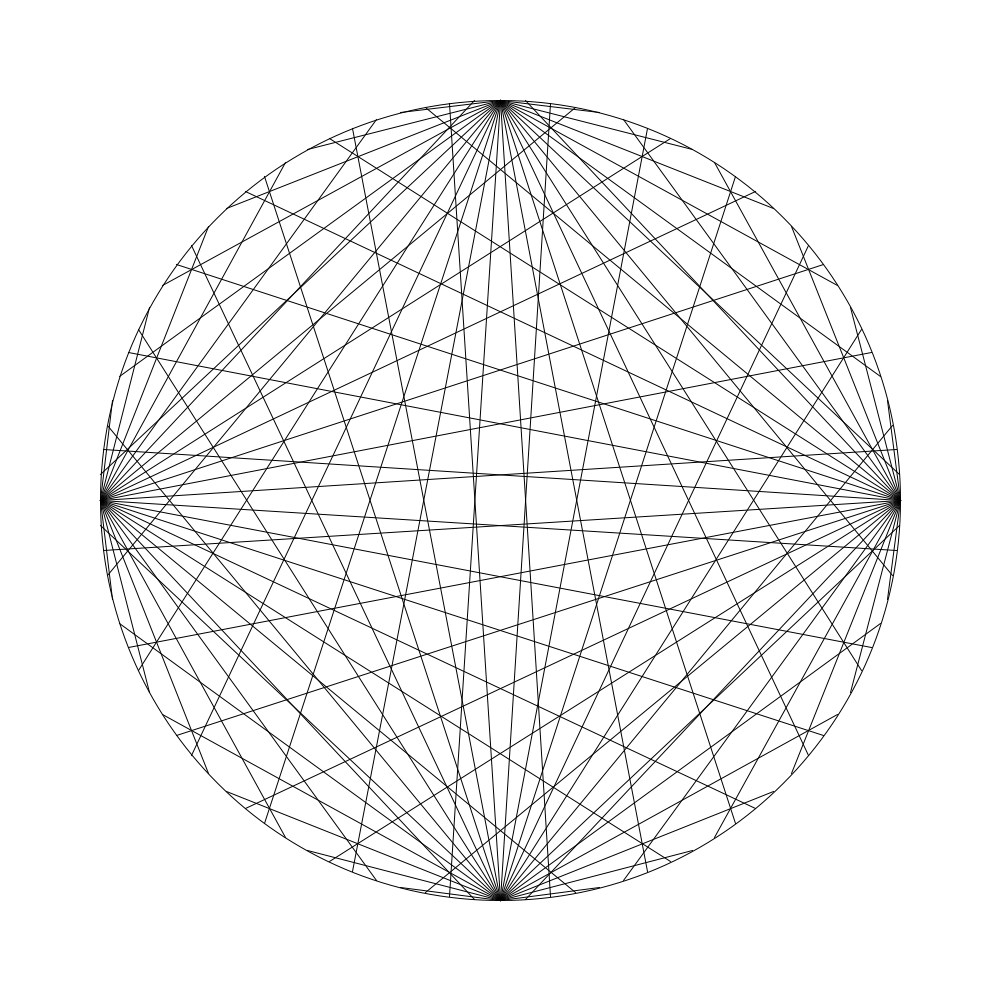}
         \caption{$\MMT(100,25)$}
         \label{fig:MMT-100-25}
     \end{subfigure}
     \begin{subfigure}[b]{0.2\textwidth}
         \centering
         \includegraphics[width=\textwidth]{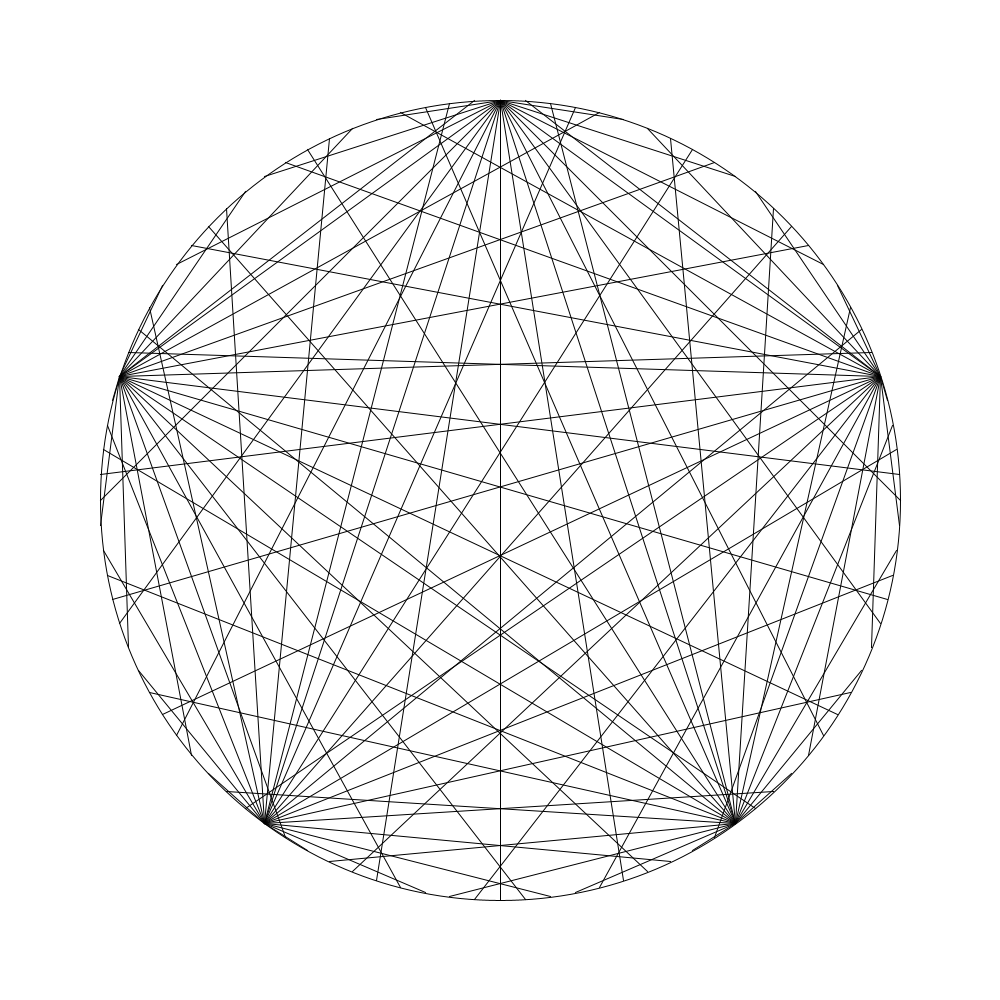}
         \caption{$\MMT(100,20)$}
         \label{fig:MMT-100-20}
     \end{subfigure} \\
     \begin{subfigure}[b]{0.2\textwidth}
         \centering
         \includegraphics[width=\textwidth]{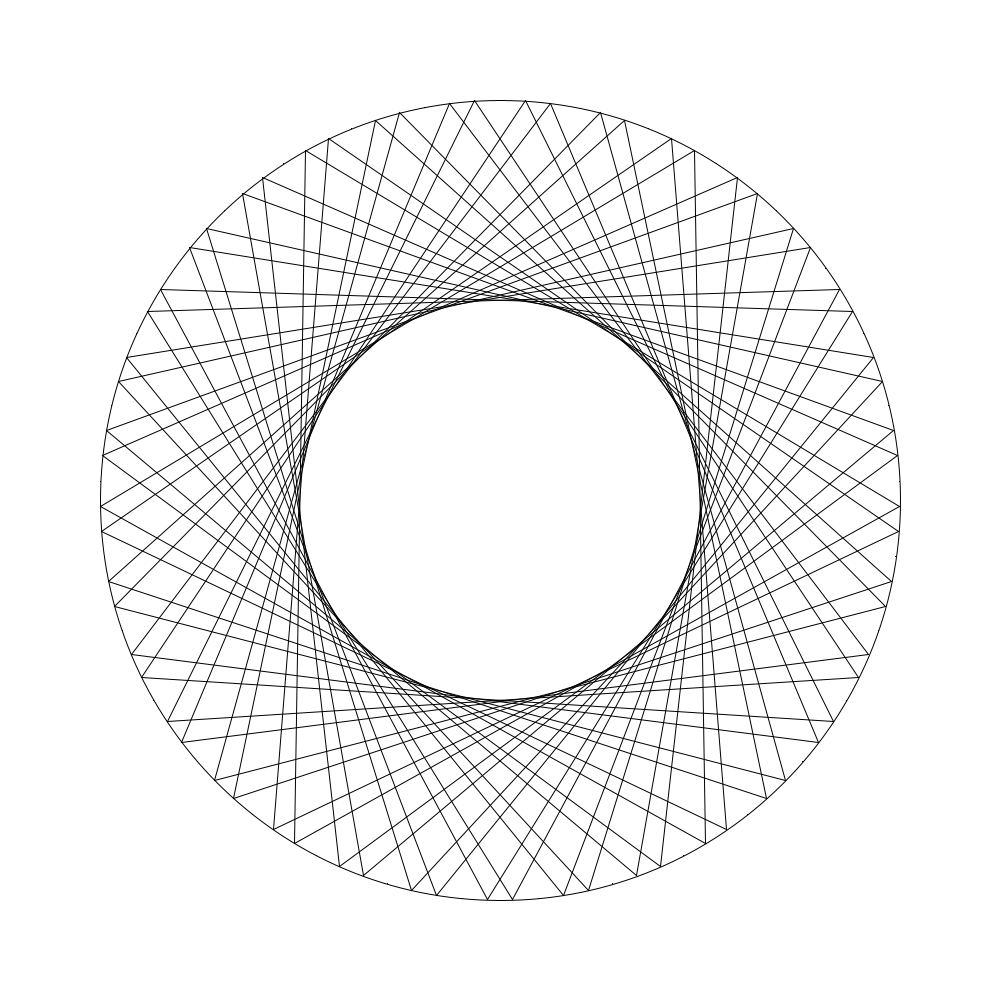}
         \caption{$\MMT(99,34)$}
         \label{fig:MMT-99-34}
     \end{subfigure}
     \begin{subfigure}[b]{0.2\textwidth}
         \centering
         \includegraphics[width=\textwidth]{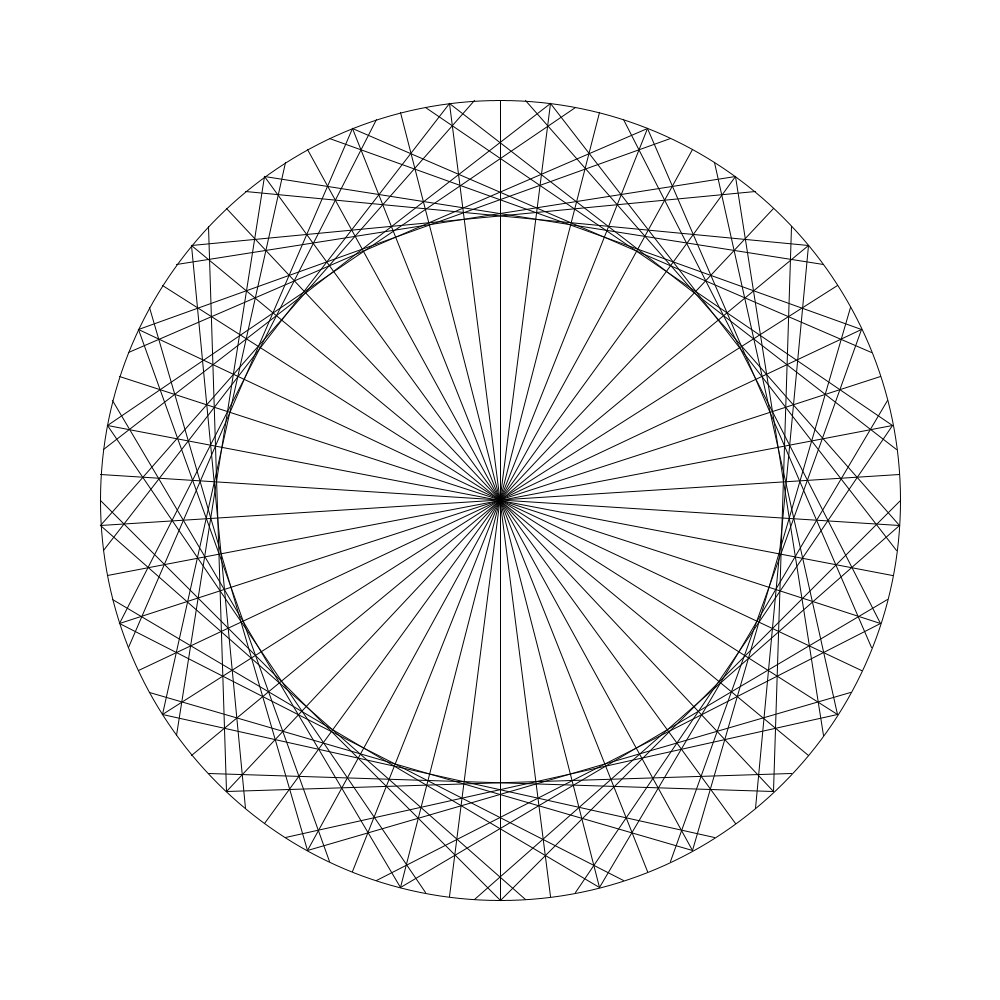}
         \caption{$\MMT(100,26)$}
         \label{fig:MMT-100-26}
     \end{subfigure}
     \begin{subfigure}[b]{0.2\textwidth}
         \centering
         \includegraphics[width=\textwidth]{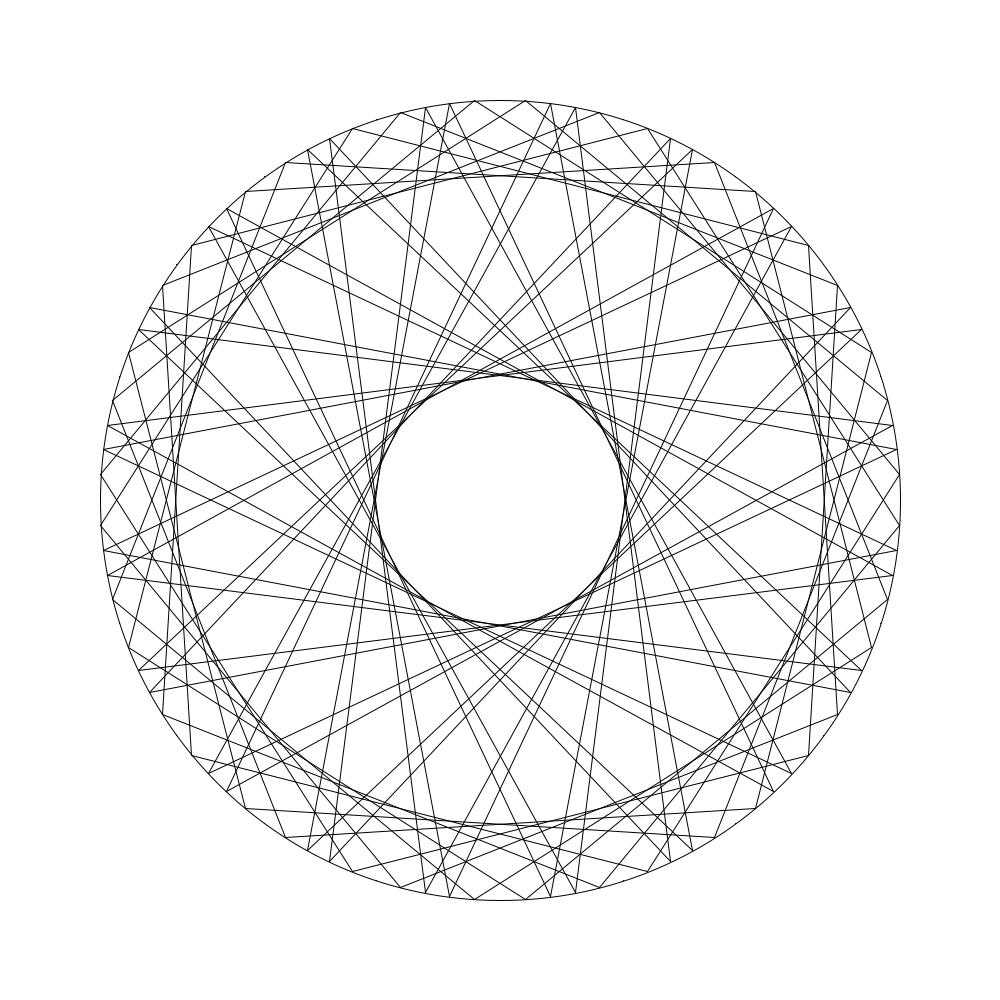}
         \caption{$\MMT(100,21)$}
         \label{fig:MMT-100-21}
     \end{subfigure} \\
    \begin{subfigure}[b]{0.2\textwidth}
         \centering
         \includegraphics[width=\textwidth]{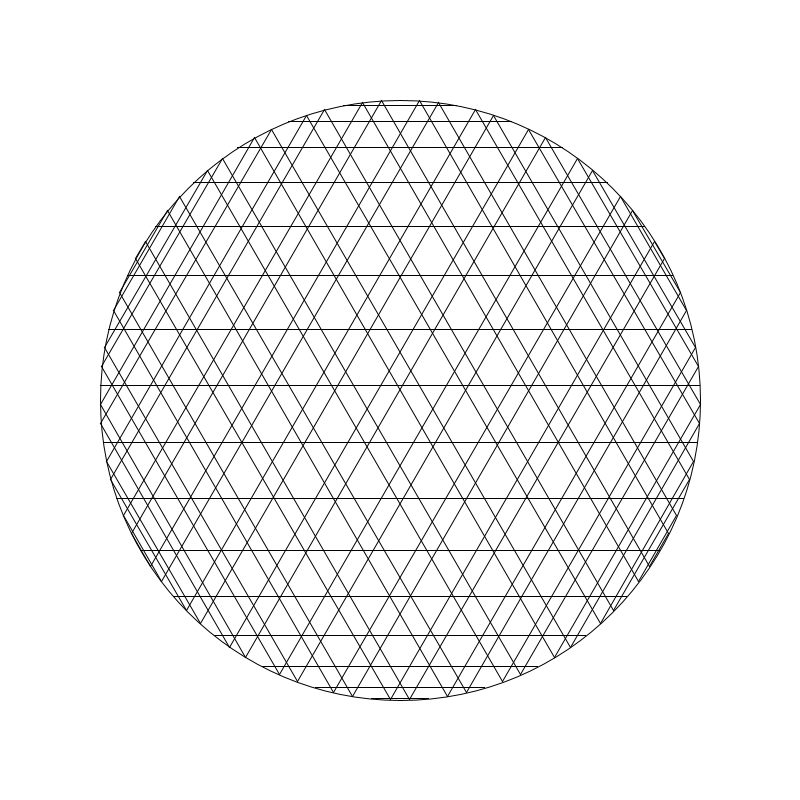}
         \caption{$\MMT(99,32)$}
         \label{fig:MMT-99-32}
     \end{subfigure}
     \begin{subfigure}[b]{0.2\textwidth}
         \centering
         \includegraphics[width=\textwidth]{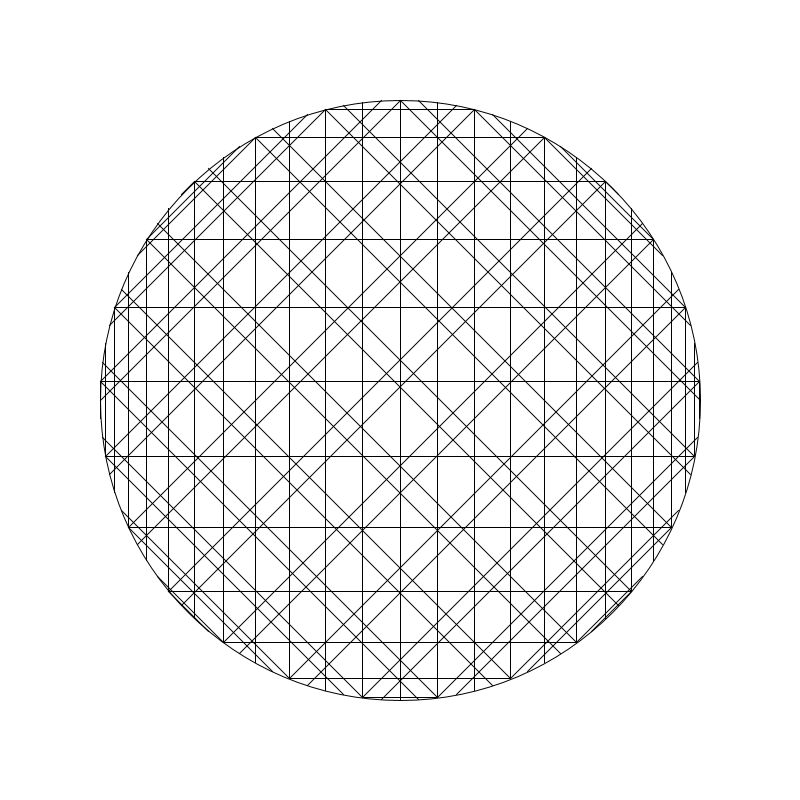}
         \caption{$\MMT(100,24)$}
         \label{fig:MMT-100-24}
     \end{subfigure}
     \begin{subfigure}[b]{0.2\textwidth}
         \centering
         \includegraphics[width=\textwidth]{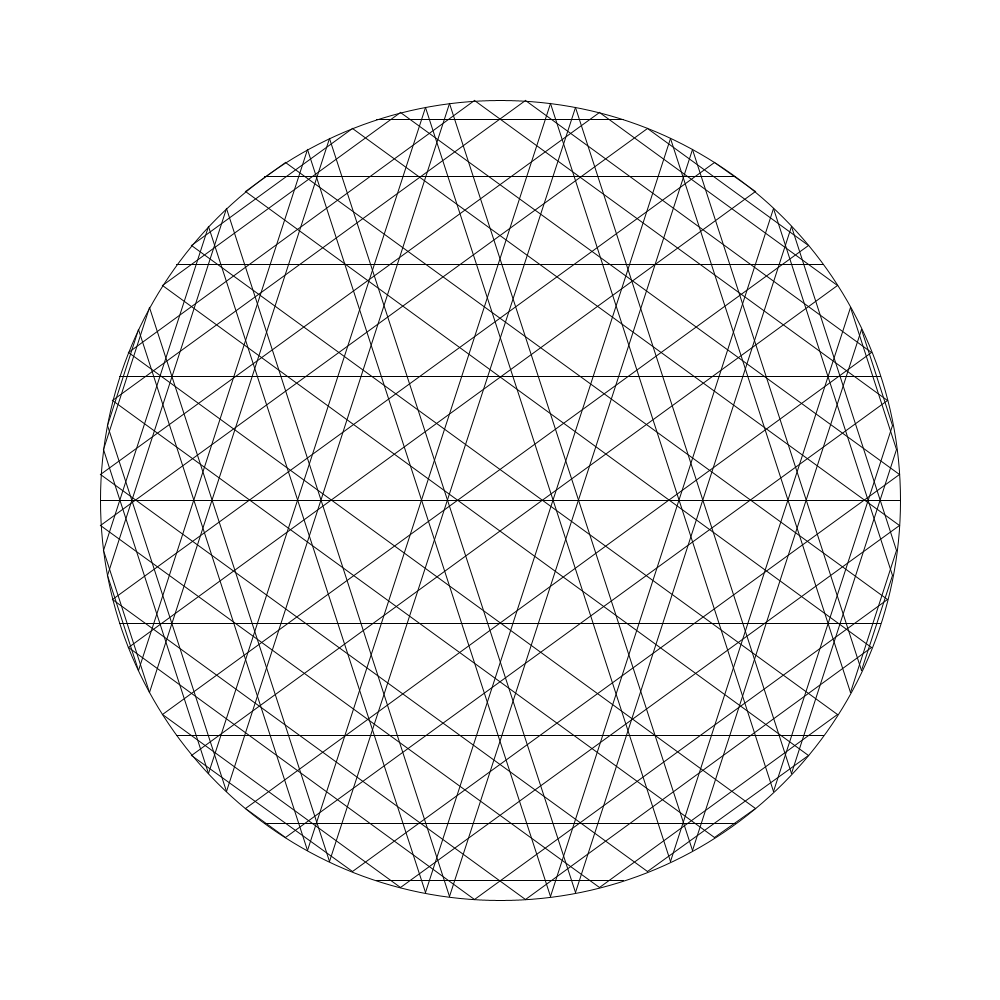}
         \caption{$\MMT(100, 19)$}
         \label{fig:MMT-100-19}
     \end{subfigure}
    \caption{The top row shows examples of graphs of the form $\MMT(m, \frac{m}{b})$ with $b = 3, 4, 5$ from left to right. The second row shows examples of graphs of the form $\MMT(m, \frac{m}{b} + 1)$  with the same $b$ values. The last row shows graphs of the form $\MMT(m, \frac{m}{b} - 1)$ again with the same values for $b$.}
    \label{fig:m-div-b-examples}
\end{figure}

As I was considering these families of modular stitch graphs, I wondered if we could define a similar kind of graph without the assumption that $m$ is divisible by $b$. I decided to look at $\MMT(m,a)$ where $b \nmid m$ and $a = \lceil \frac{m}{b} \rceil$ or $a = \lfloor \frac{m}{b} \rfloor$, since I still wanted $a$ to be an integer. I was amazed to find many clear patterns. Figure~\ref{fig:table_of_tables} shows graphs of the form $\MMT(m, \lceil \frac{m}{b} \rceil)$. The rows are indexed by the value for $b$ and the columns are indexed by the remainder of $m$ mod $b$. I encourage the reader to take a minute and find patterns within this ``grid of graphs''.

\begin{figure}
\setlength\tabcolsep{2pt}
\begin{tabularx}{\textwidth}{@{}c*{9}{C}@{}}
    & 1& & & & & & & & \\
    2 &
   \includegraphics[align=c, width=\linewidth, height=\linewidth, keepaspectratio]{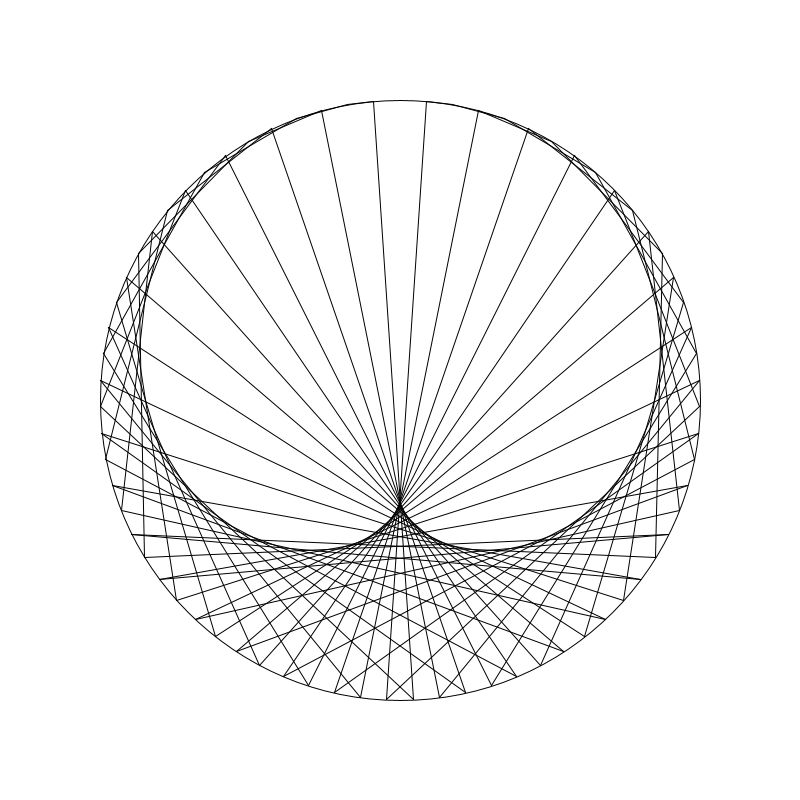} & 2 & & & & & & & \\
   3 &
   \includegraphics[align=c, width=\linewidth, height=\linewidth, keepaspectratio]{"images/LargeTable/34-100".png} & \includegraphics[align=c, width=\linewidth, height=\linewidth, keepaspectratio]{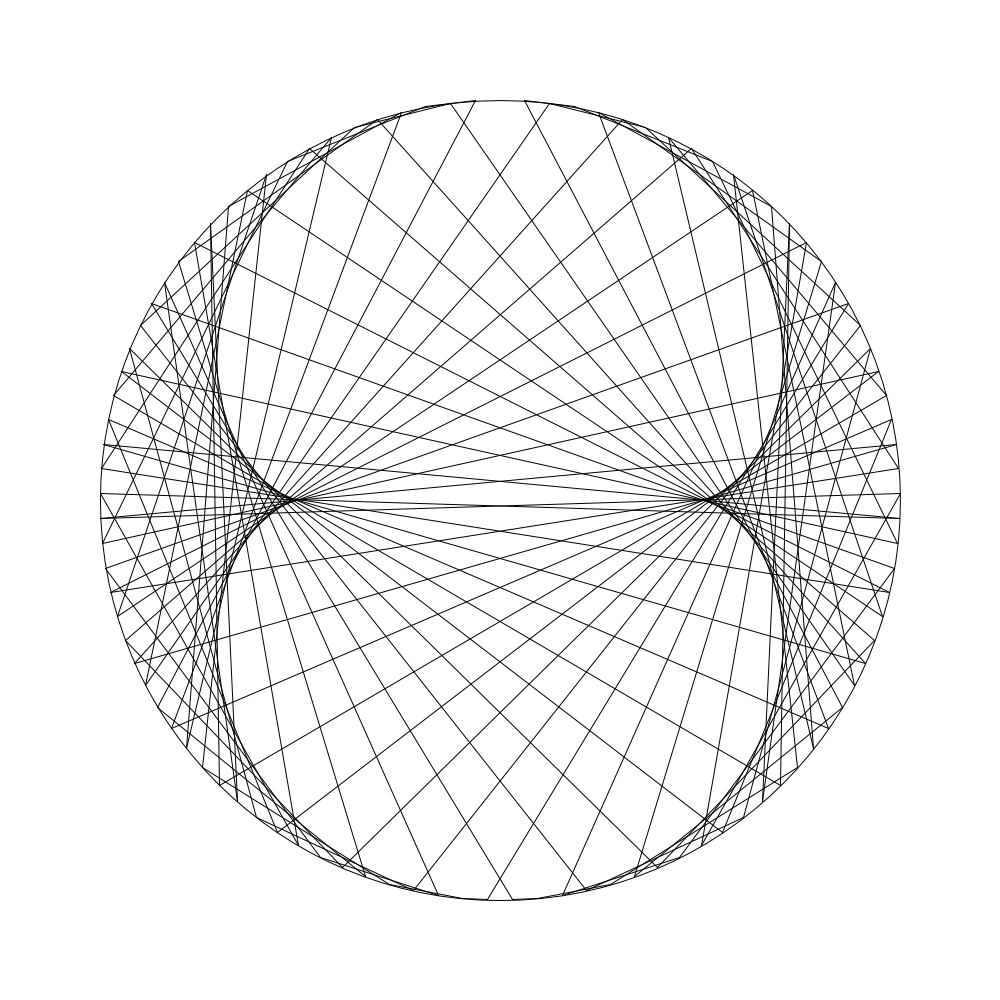}& 3 & & & & & & \\
   4 &
   \includegraphics[align=c, width=\linewidth, height=\linewidth, keepaspectratio]{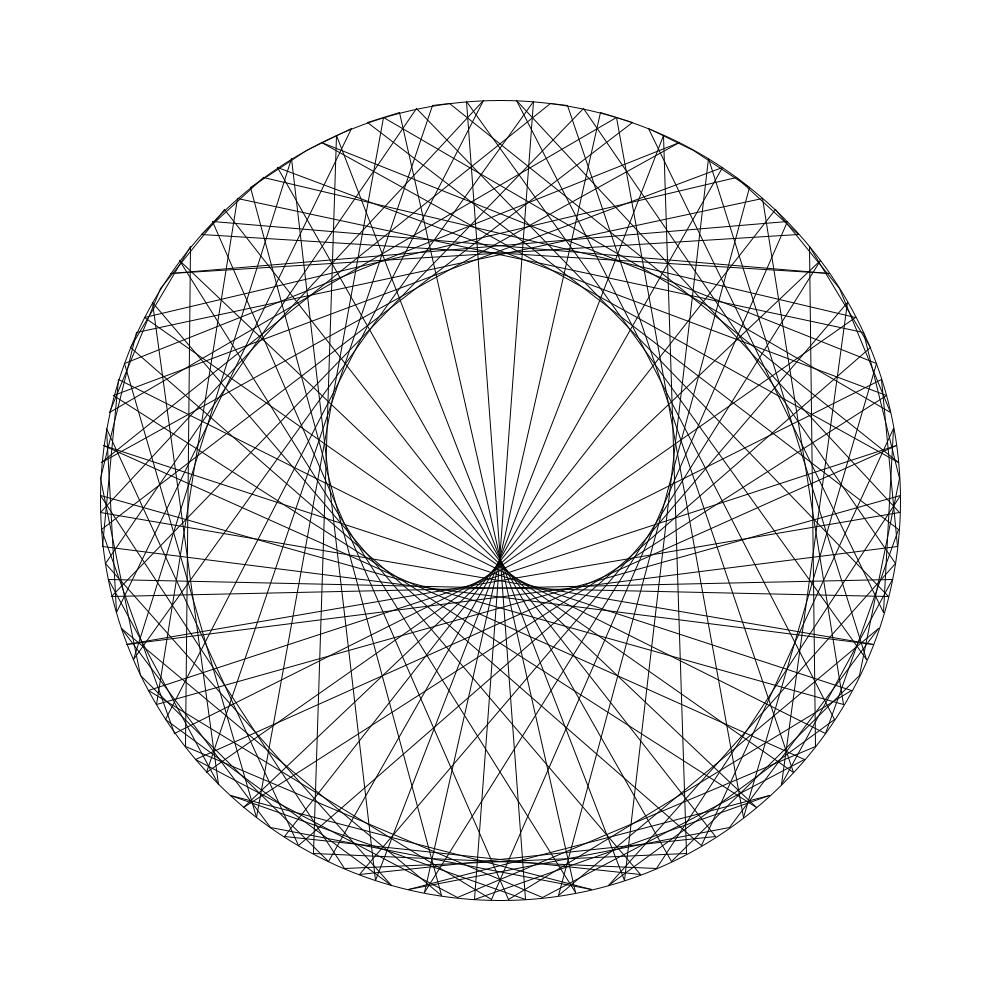} & \includegraphics[align=c, width=\linewidth, height=\linewidth, keepaspectratio]{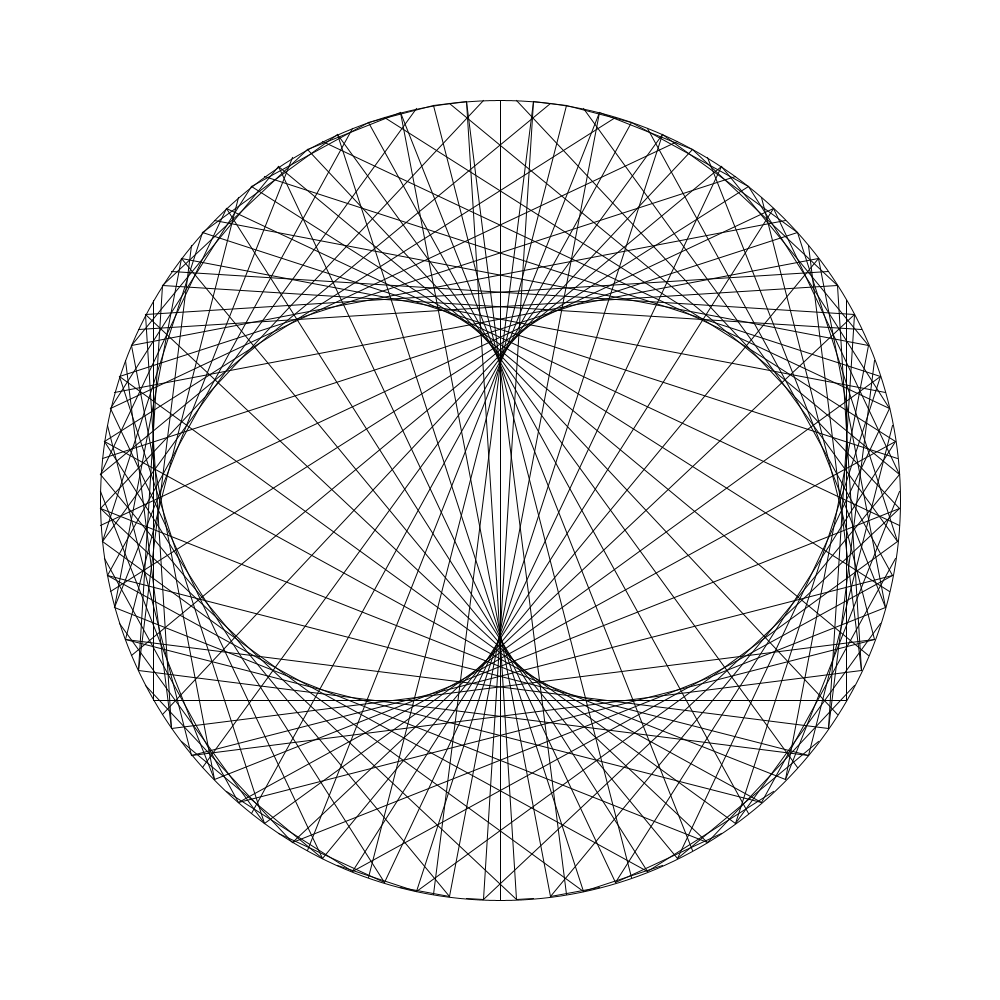} & \includegraphics[align=c, width=\linewidth, height=\linewidth, keepaspectratio]{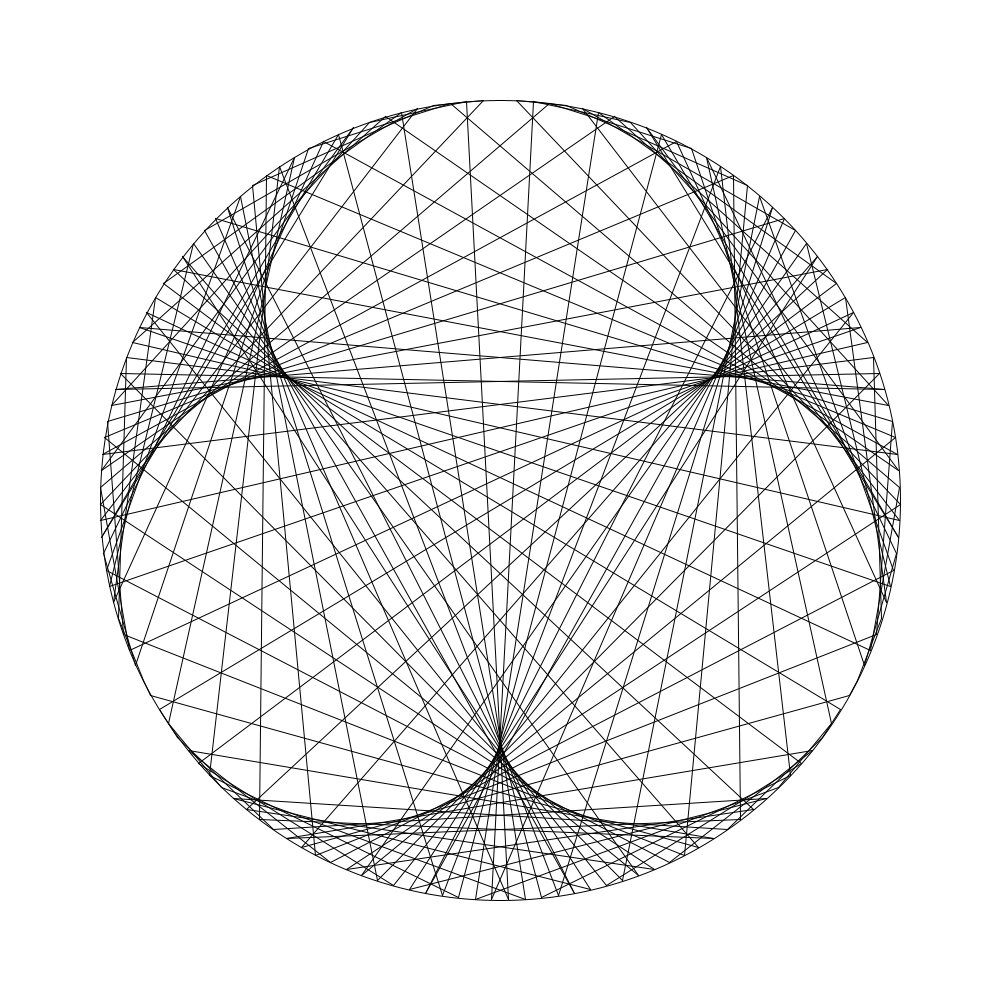} & 4 & & & & & \\
    5 &
   \includegraphics[align=c, width=\linewidth, height=\linewidth, keepaspectratio]{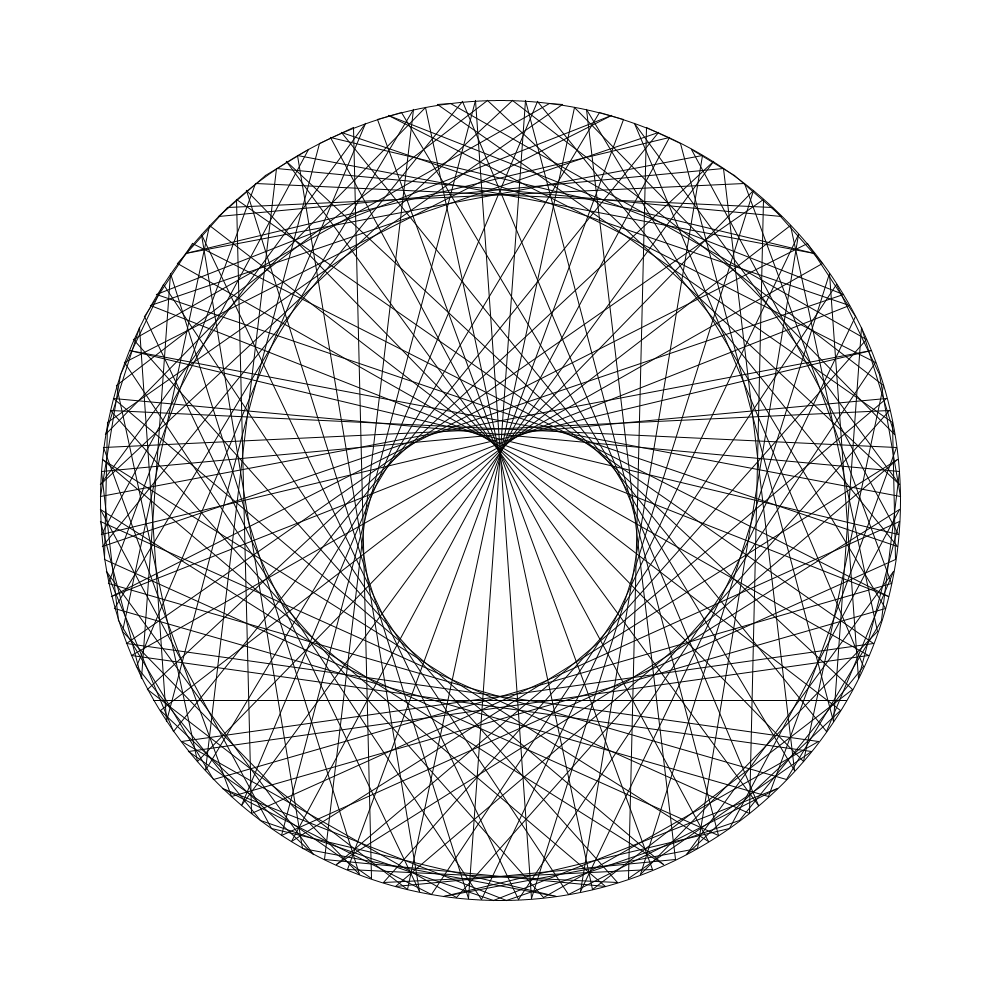} & \includegraphics[align=c, width=\linewidth, height=\linewidth, keepaspectratio]{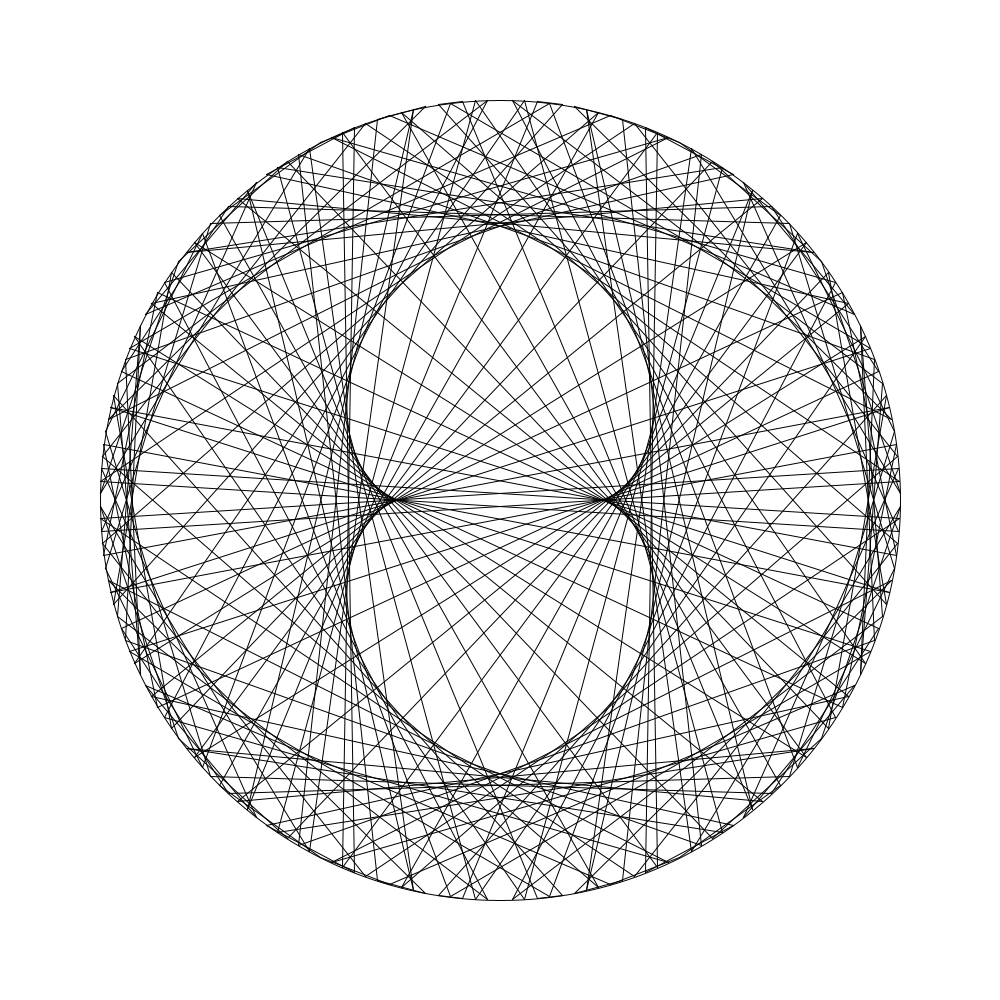}& \includegraphics[align=c, width=\linewidth, height=\linewidth, keepaspectratio]{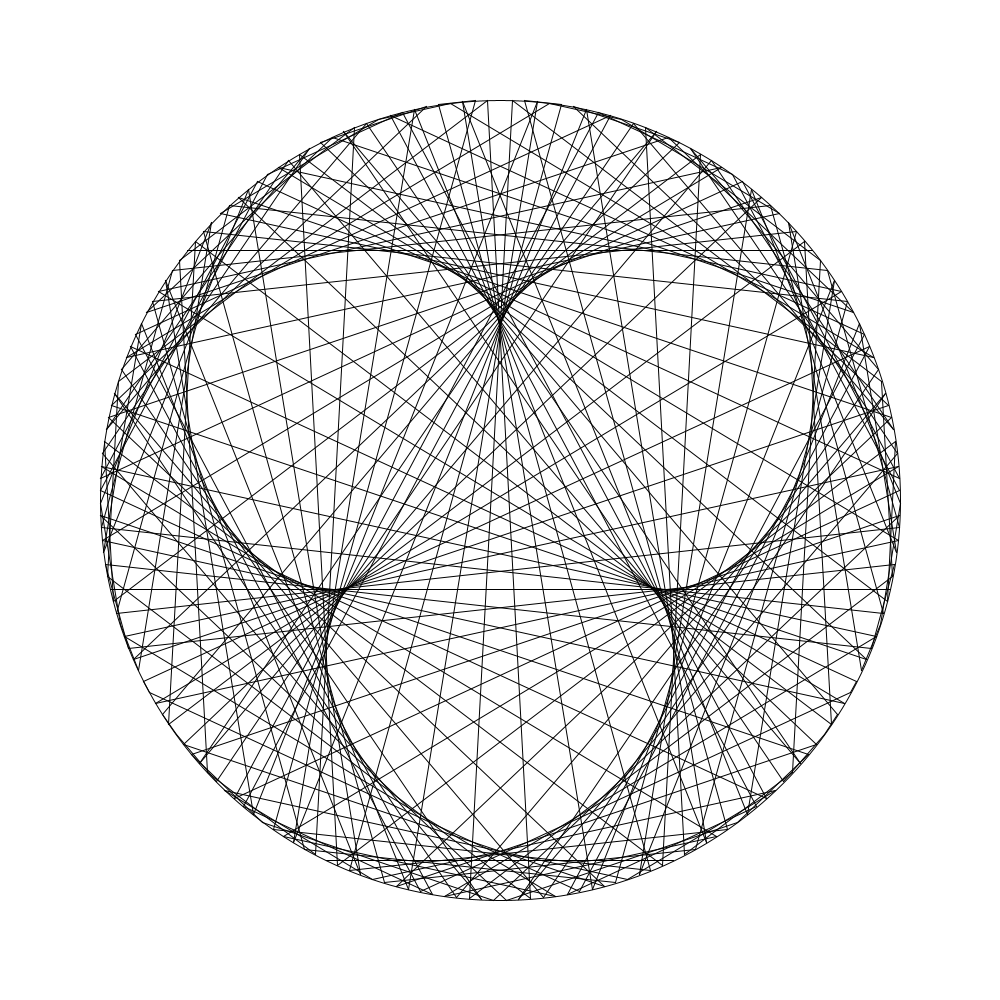}& 
   \includegraphics[align=c, width=\linewidth, height=\linewidth, keepaspectratio]{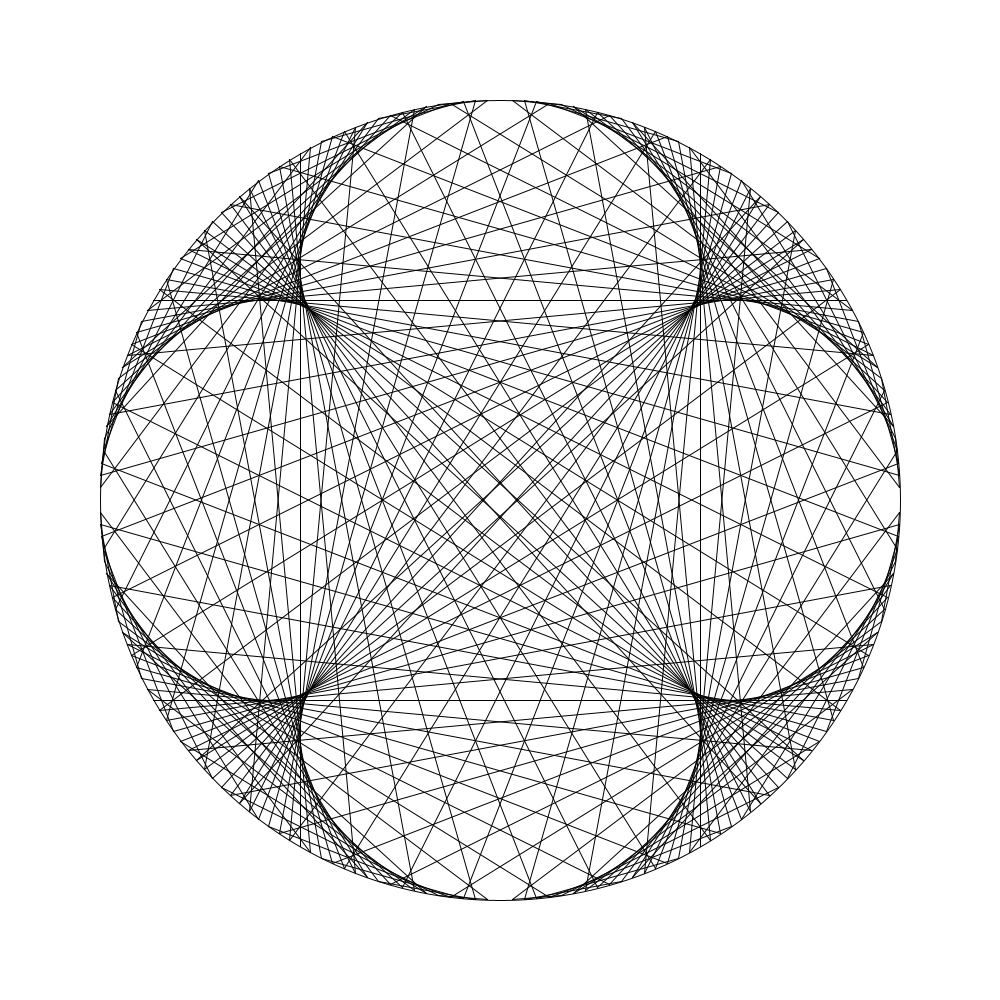}& 5 & & & & \\
   6 &
   \includegraphics[align=c, width=\linewidth, height=\linewidth, keepaspectratio]{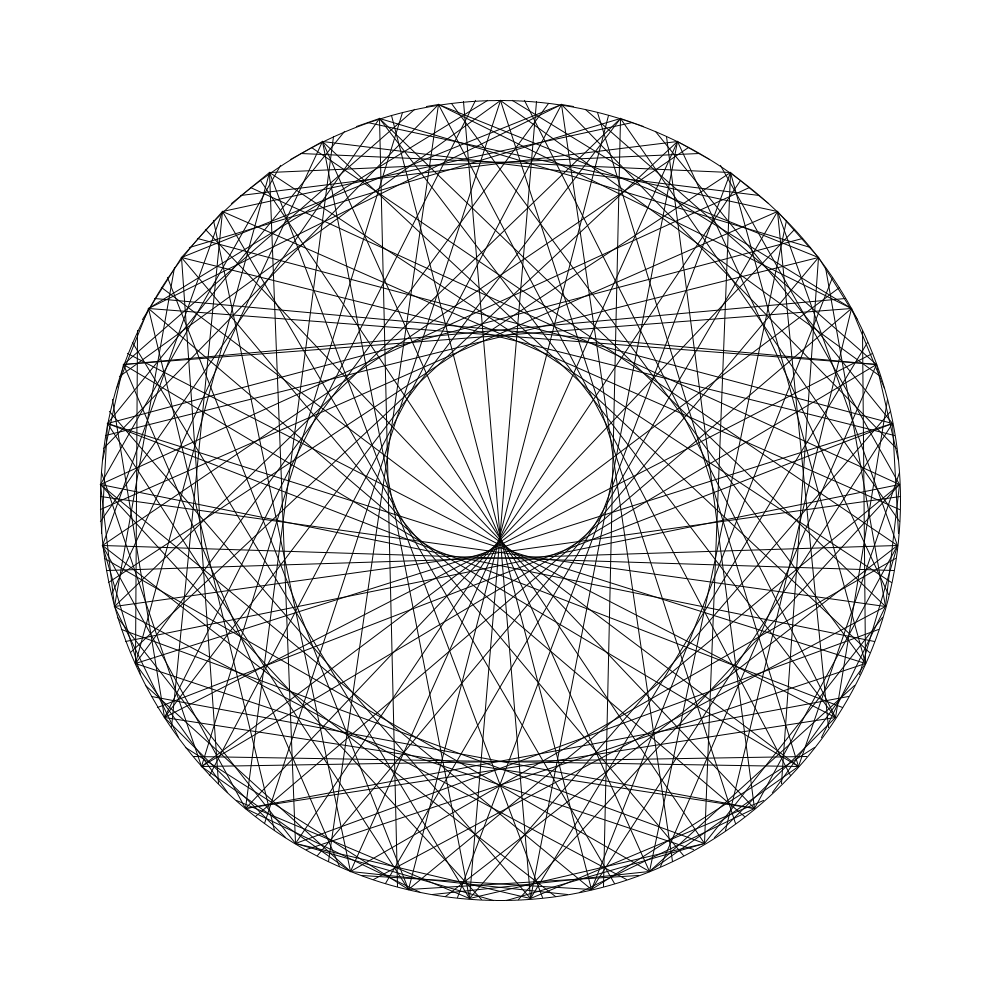} & \includegraphics[align=c, width=\linewidth, height=\linewidth, keepaspectratio]{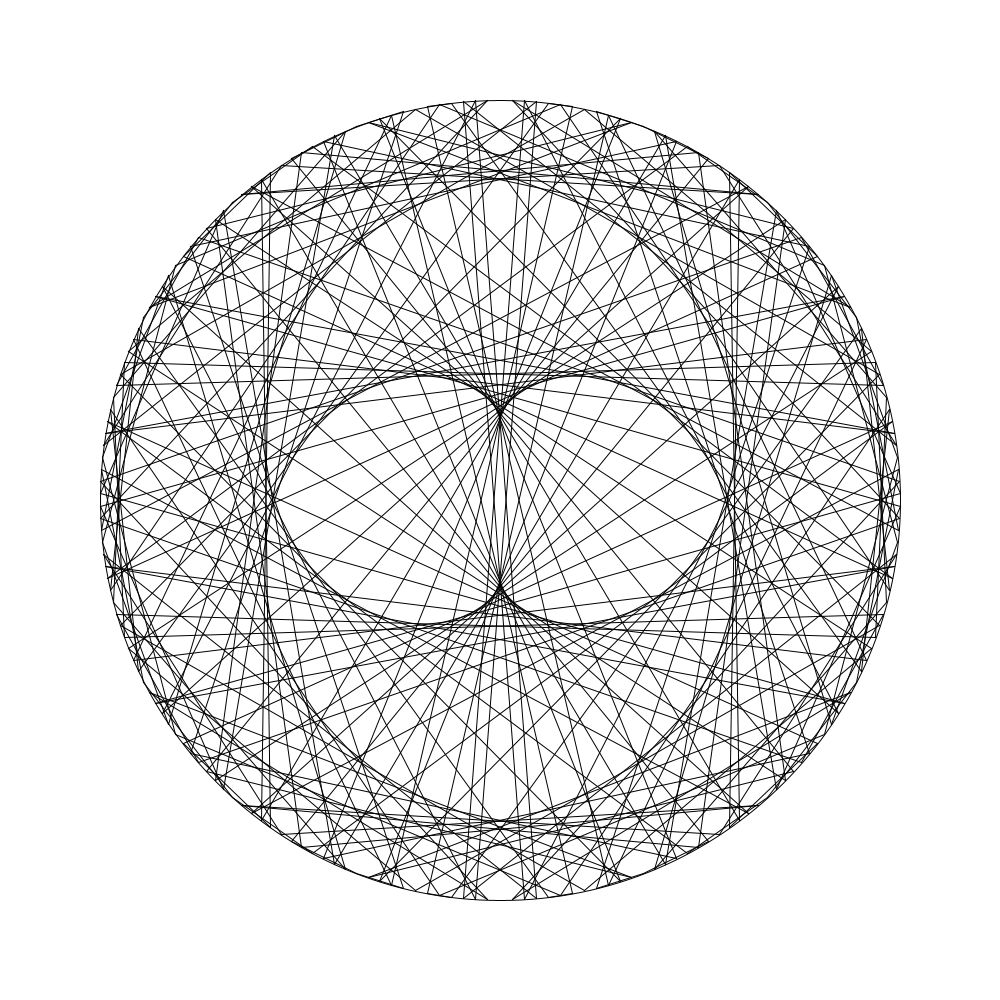}& \includegraphics[align=c, width=\linewidth, height=\linewidth, keepaspectratio]{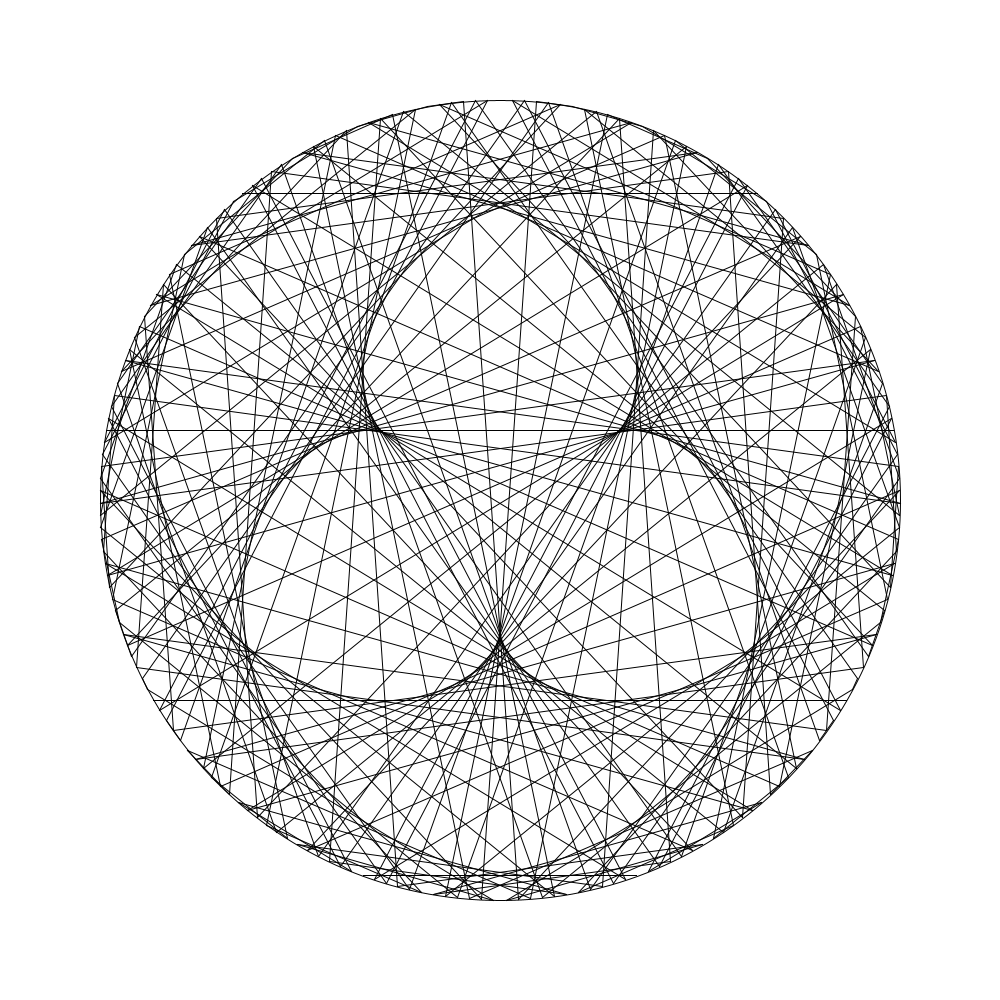}& 
   \includegraphics[align=c, width=\linewidth, height=\linewidth, keepaspectratio]{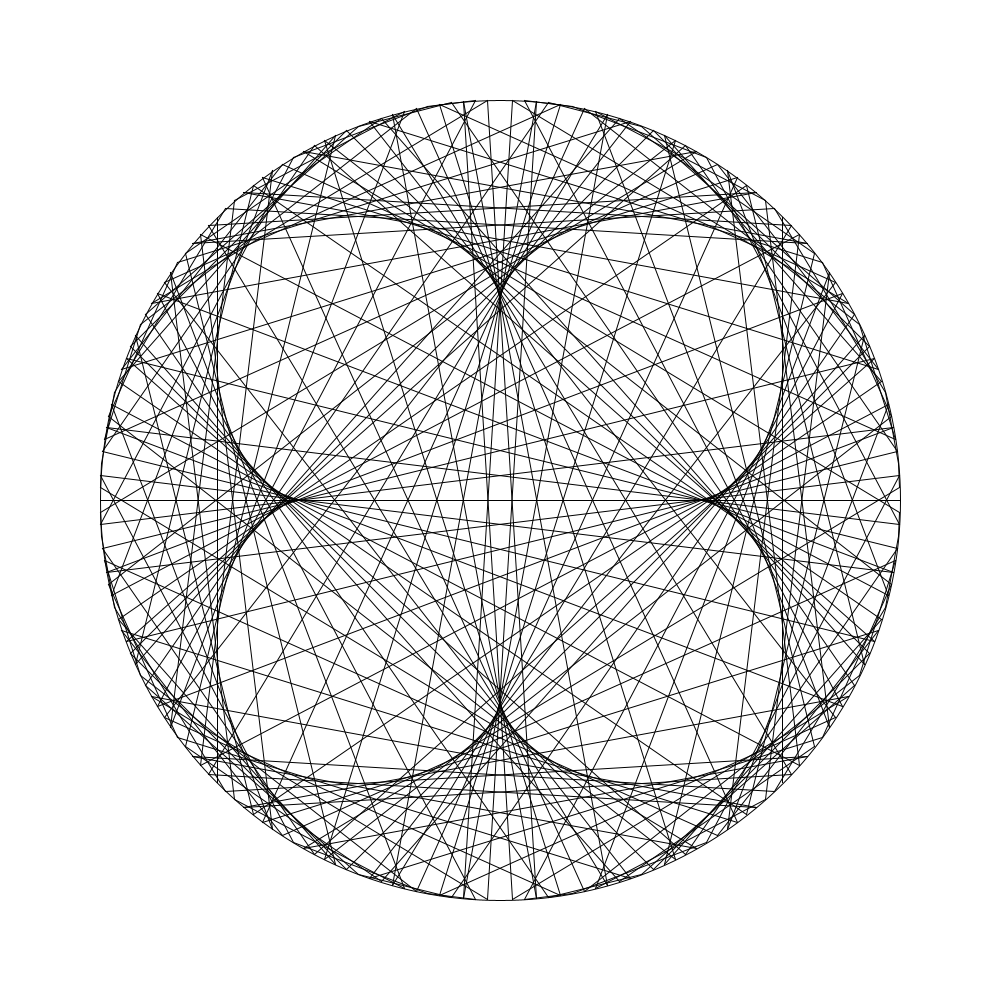}& 
   \includegraphics[align=c, width=\linewidth, height=\linewidth, keepaspectratio]{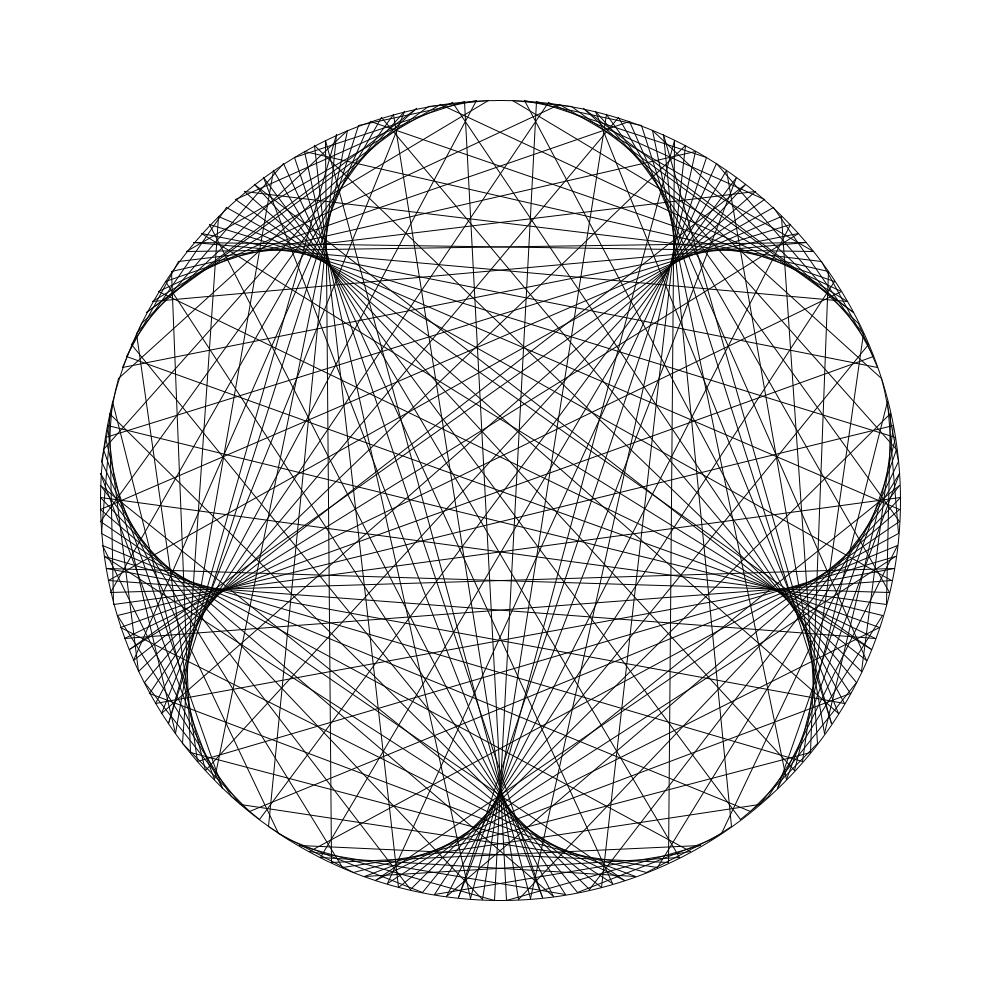}& 6 & & & \\
   7 &
   \includegraphics[align=c, width=\linewidth, height=\linewidth, keepaspectratio]{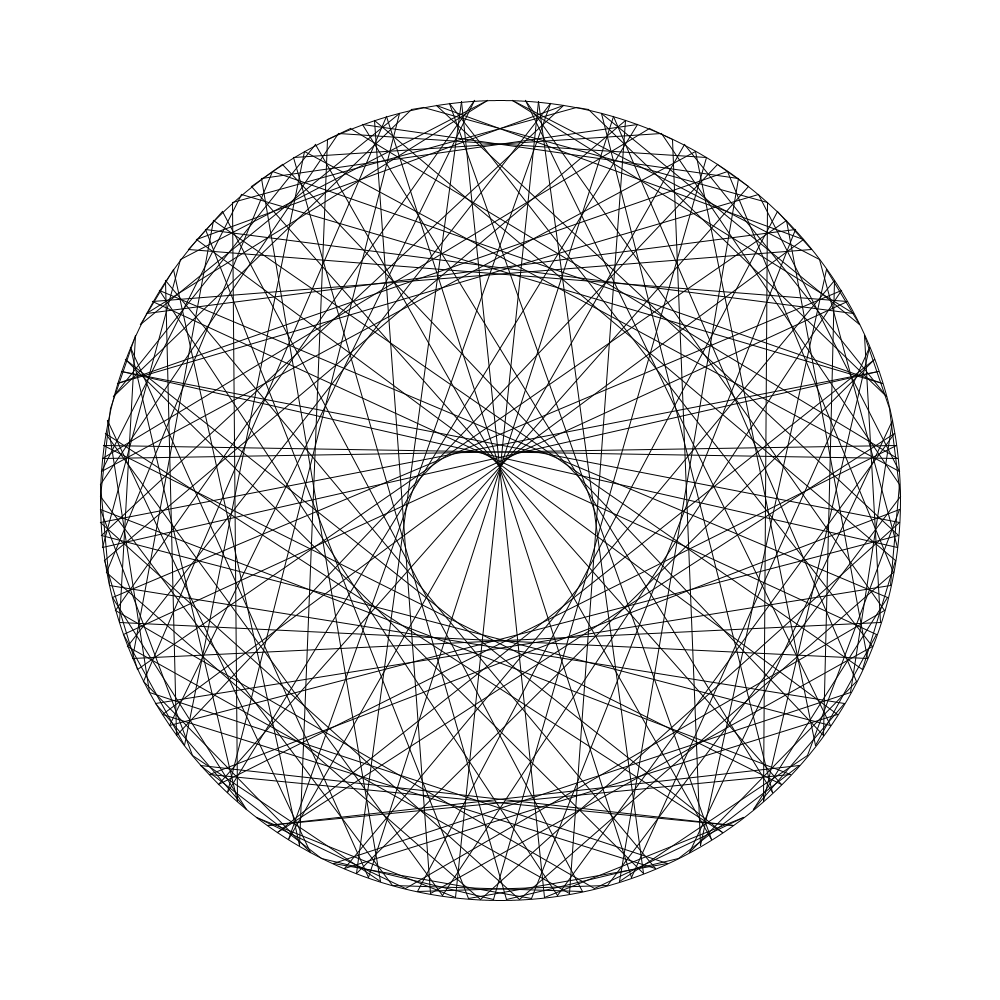} & \includegraphics[align=c, width=\linewidth, height=\linewidth, keepaspectratio]{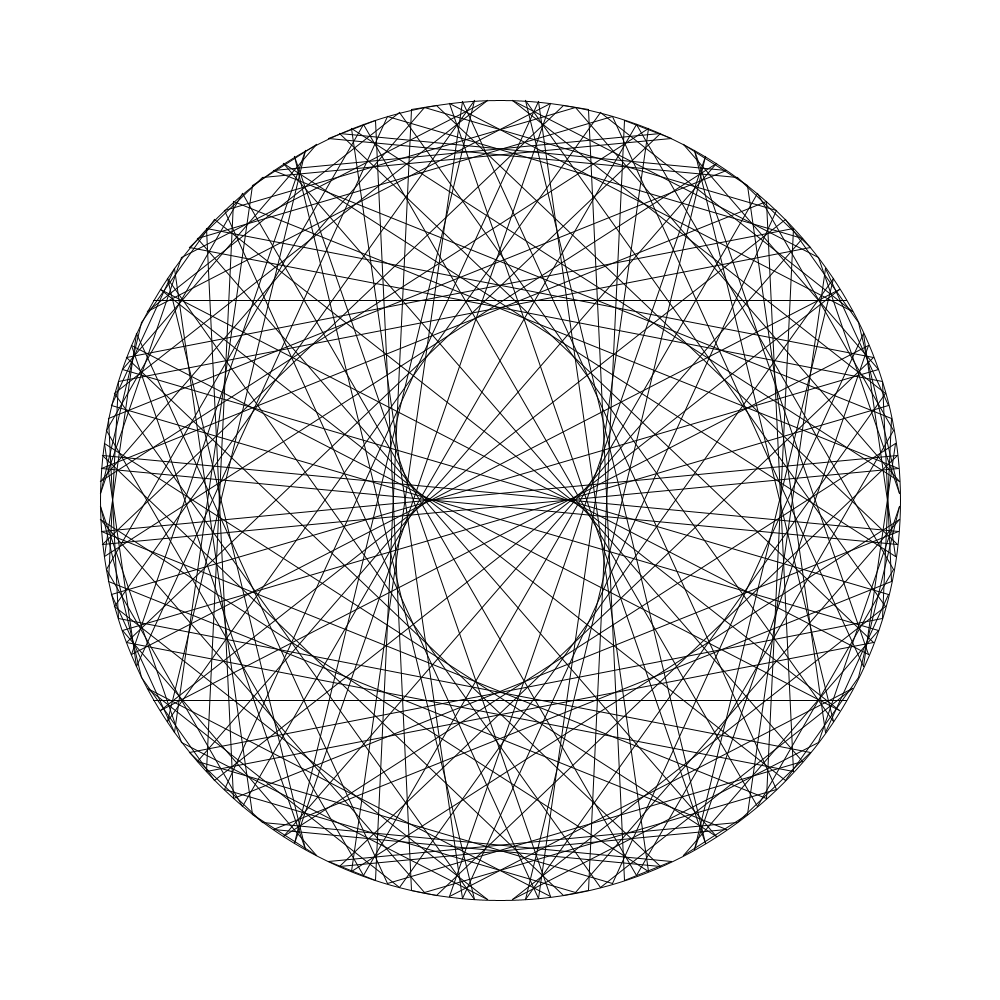}& \includegraphics[align=c, width=\linewidth, height=\linewidth, keepaspectratio]{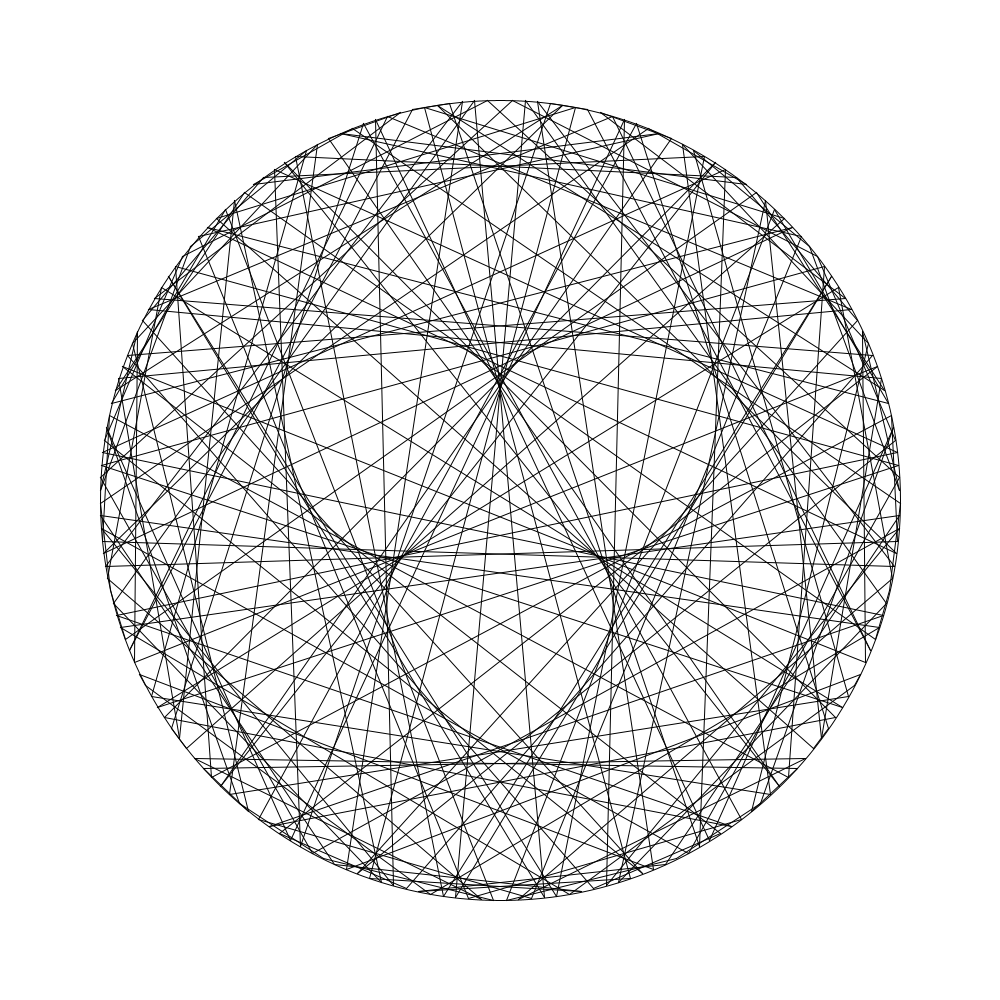}& 
   \includegraphics[align=c, width=\linewidth, height=\linewidth, keepaspectratio]{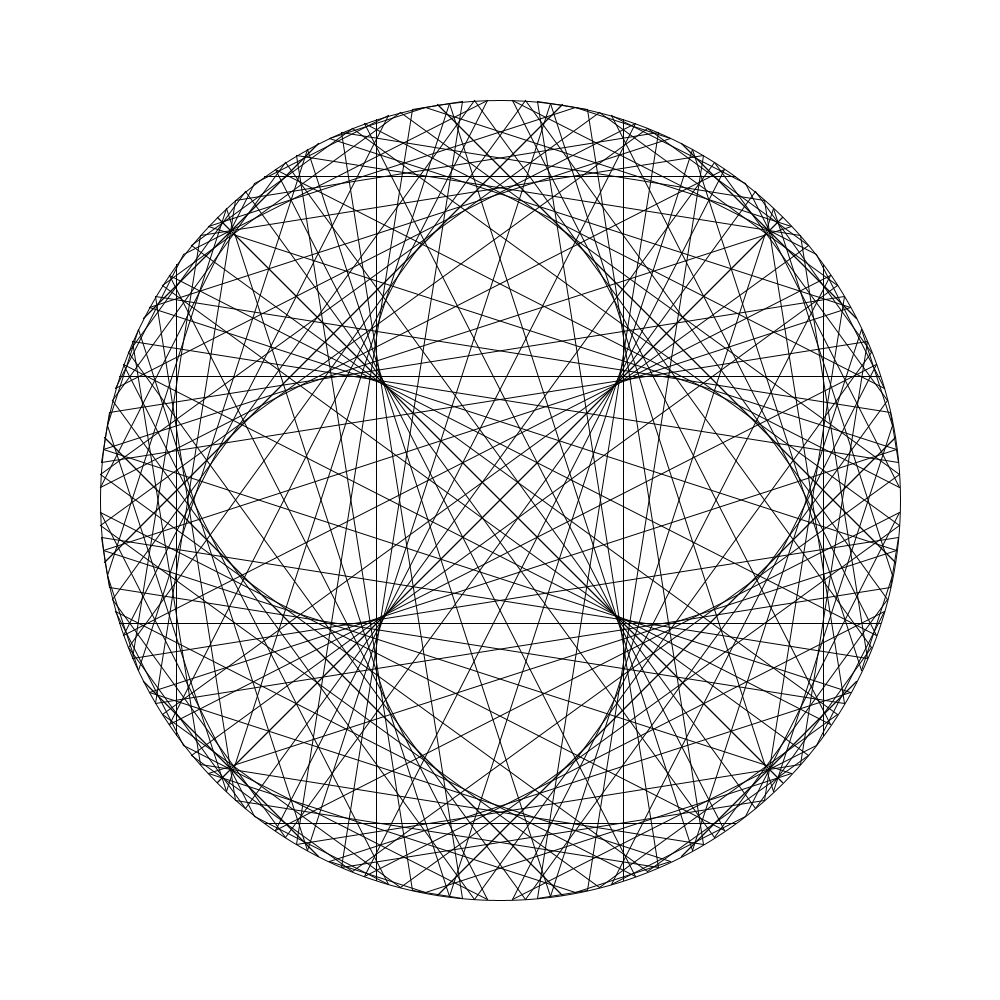}& 
   \includegraphics[align=c, width=\linewidth, height=\linewidth, keepaspectratio]{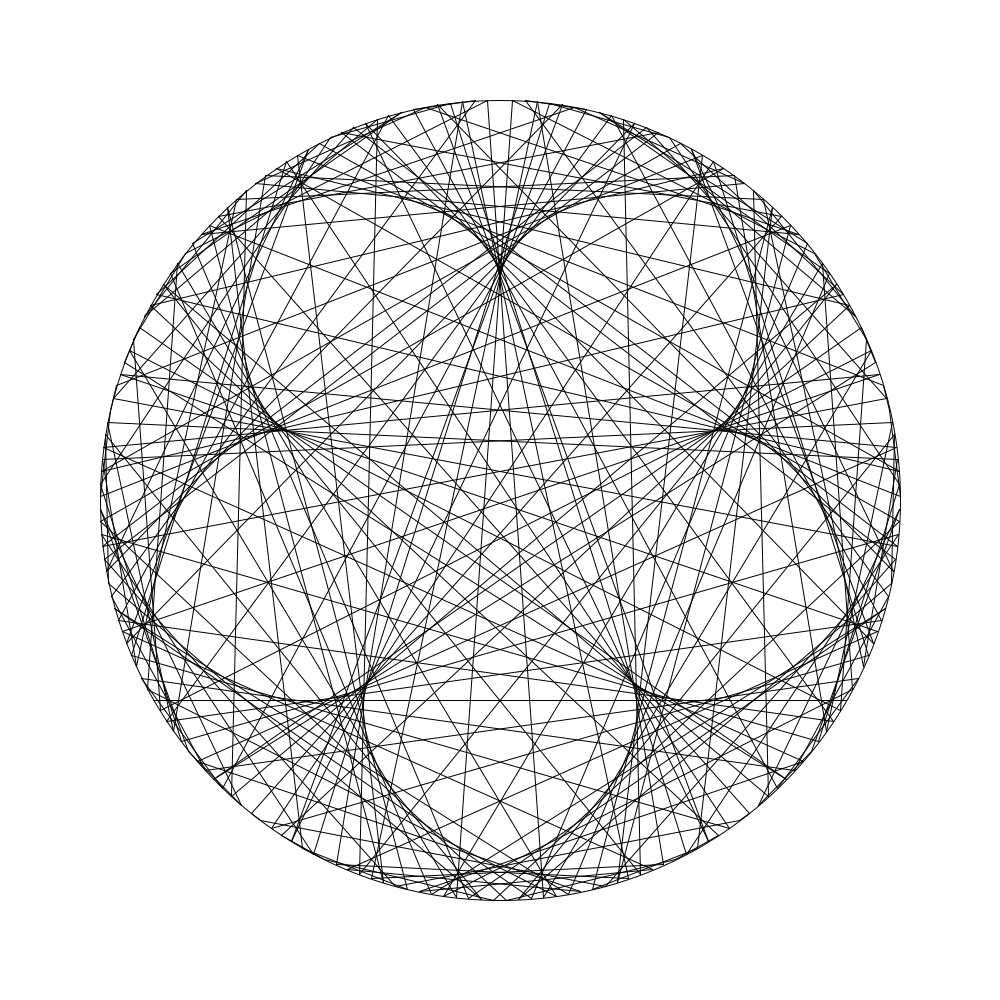}& 
   \includegraphics[align=c, width=\linewidth, height=\linewidth, keepaspectratio]{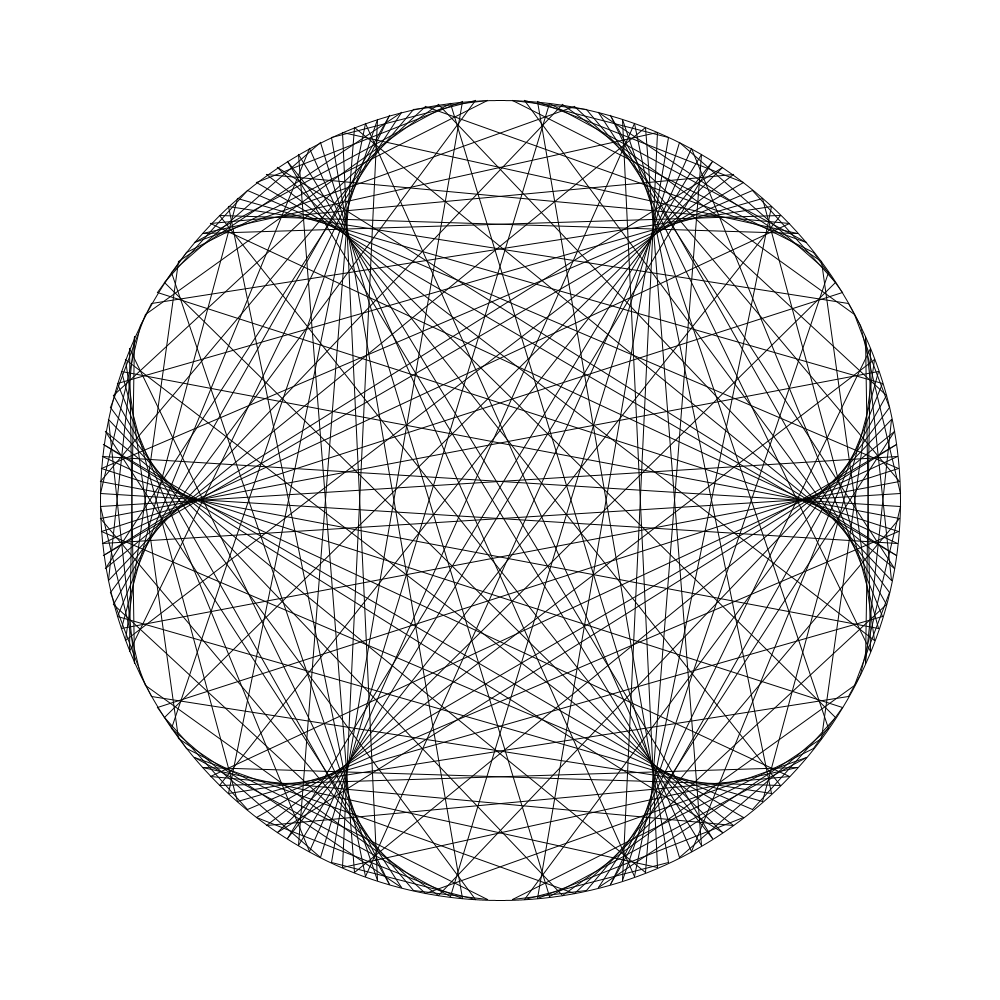}& 7 & & \\
    8 &
   \includegraphics[align=c, width=\linewidth, height=\linewidth, keepaspectratio]{"images/LargeTable/201-26".png} & \includegraphics[align=c, width=\linewidth, height=\linewidth, keepaspectratio]{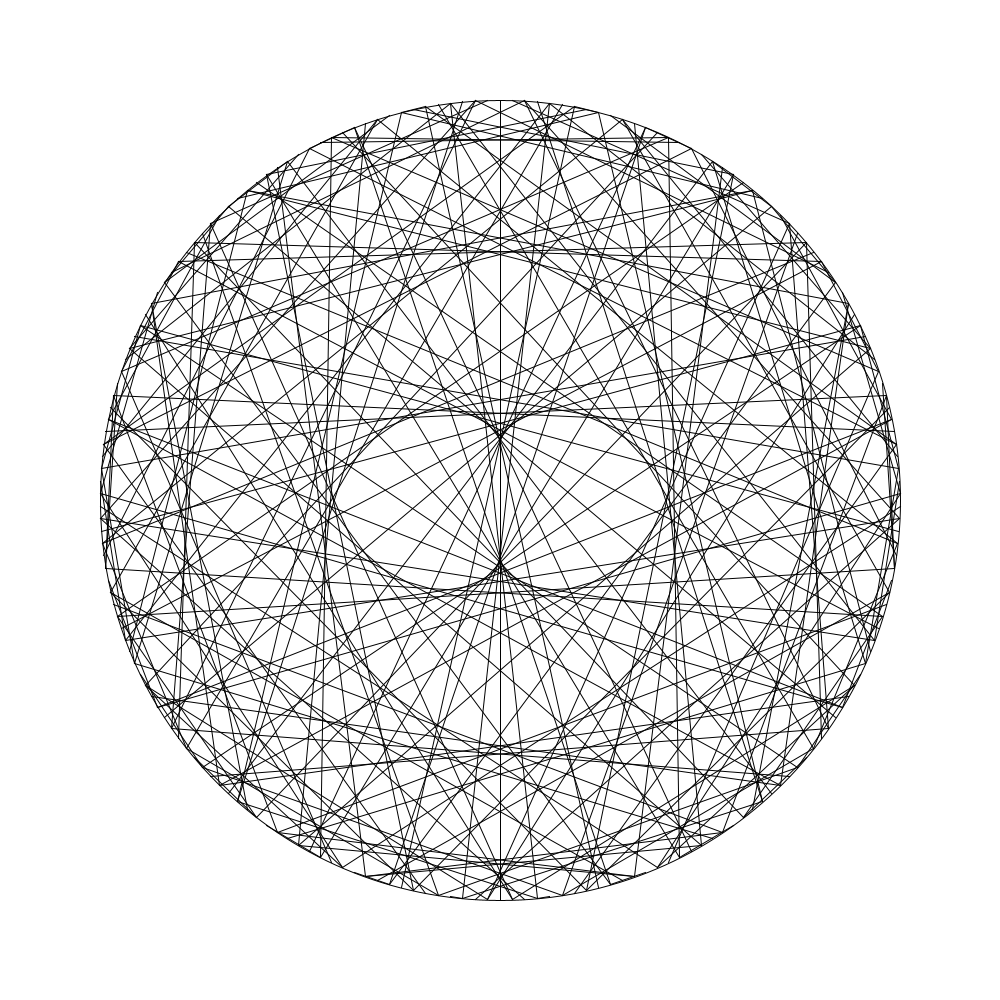}& \includegraphics[align=c, width=\linewidth, height=\linewidth, keepaspectratio]{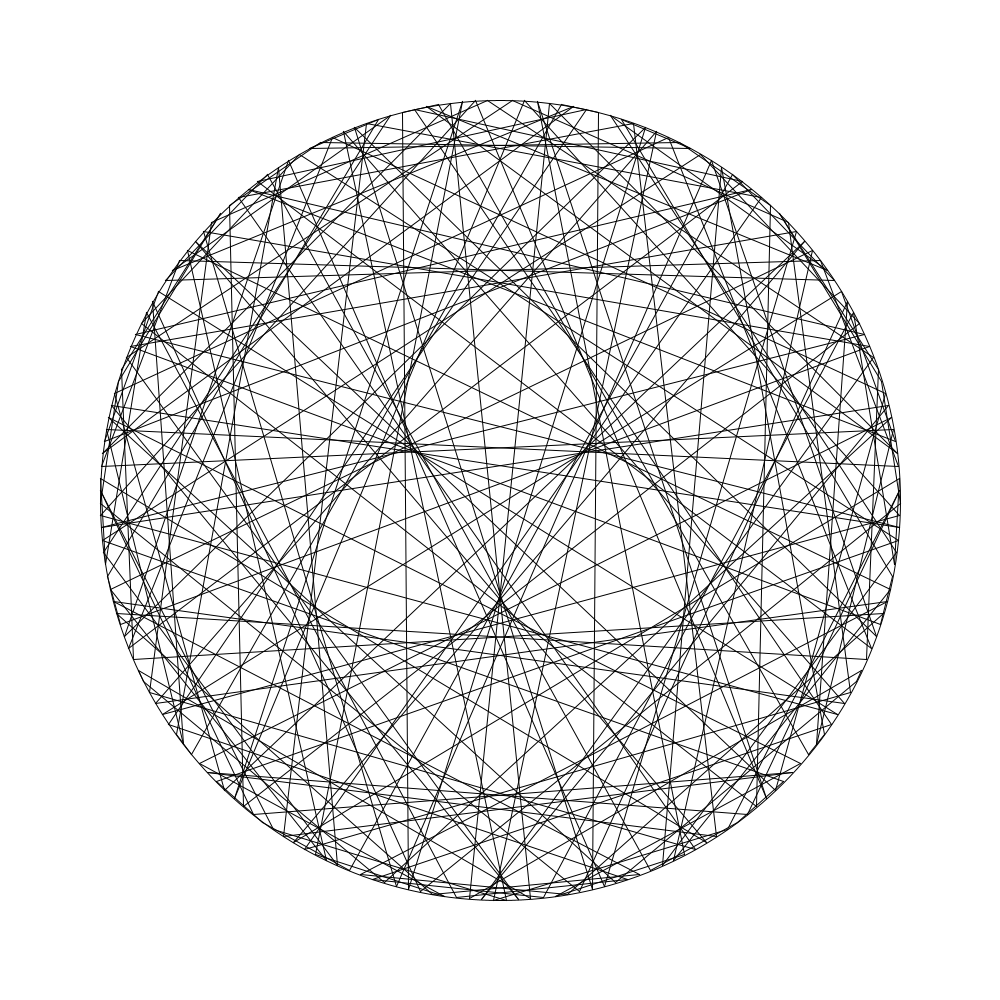}& 
   \includegraphics[align=c, width=\linewidth, height=\linewidth, keepaspectratio]{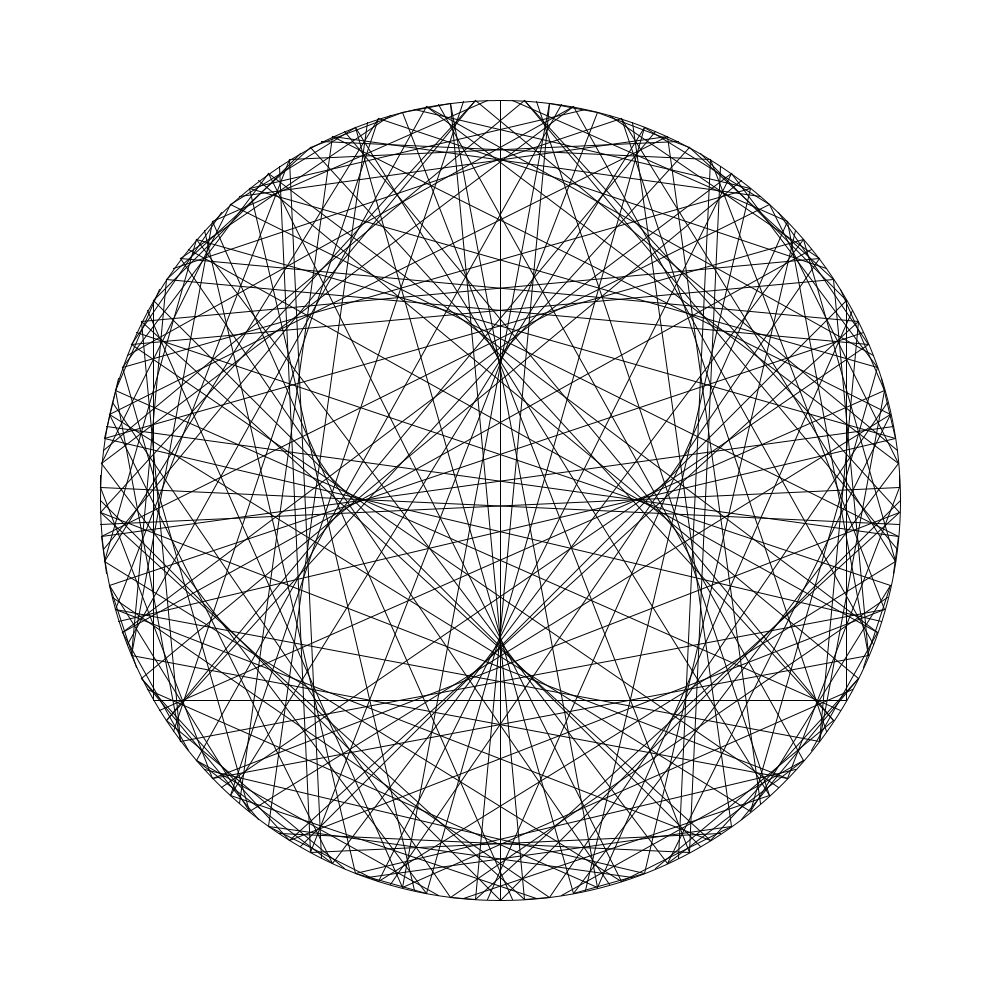}& 
   \includegraphics[align=c, width=\linewidth, height=\linewidth, keepaspectratio]{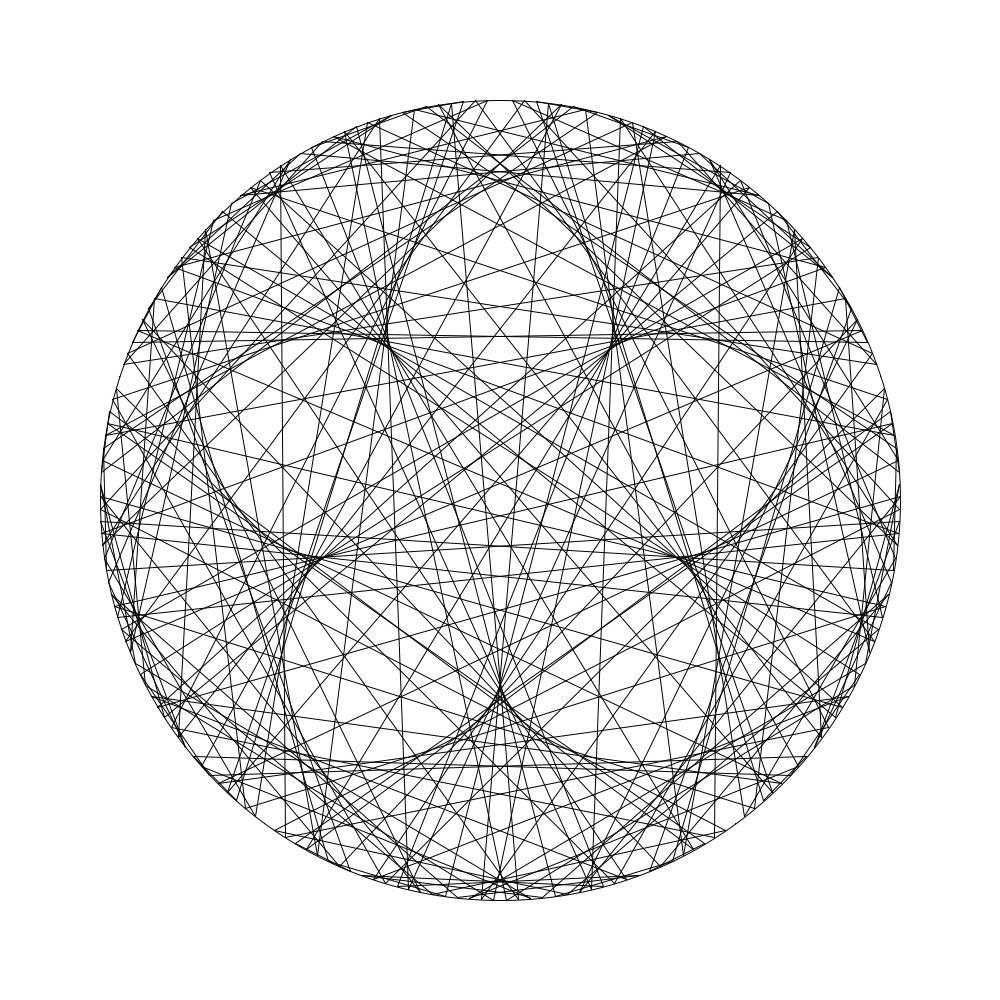}& 
   \includegraphics[align=c, width=\linewidth, height=\linewidth, keepaspectratio]{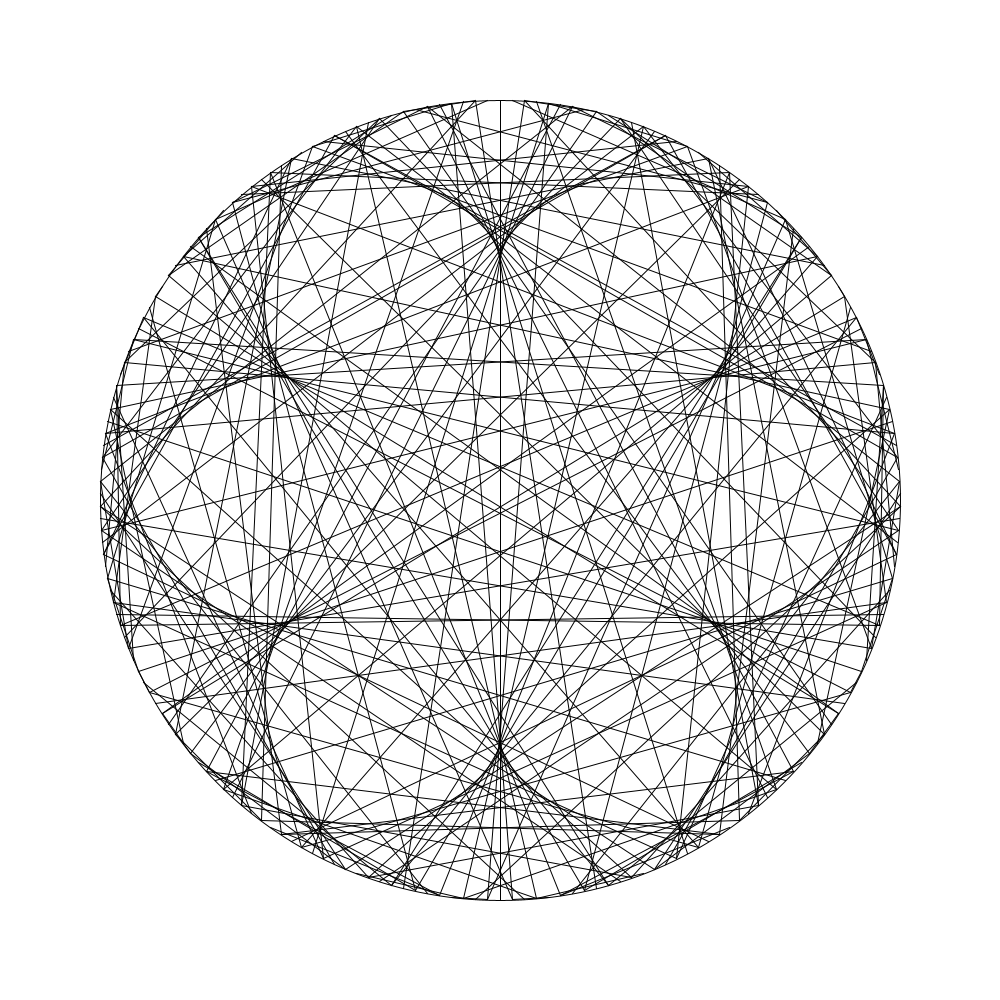}&
   \includegraphics[align=c, width=\linewidth, height=\linewidth, keepaspectratio]{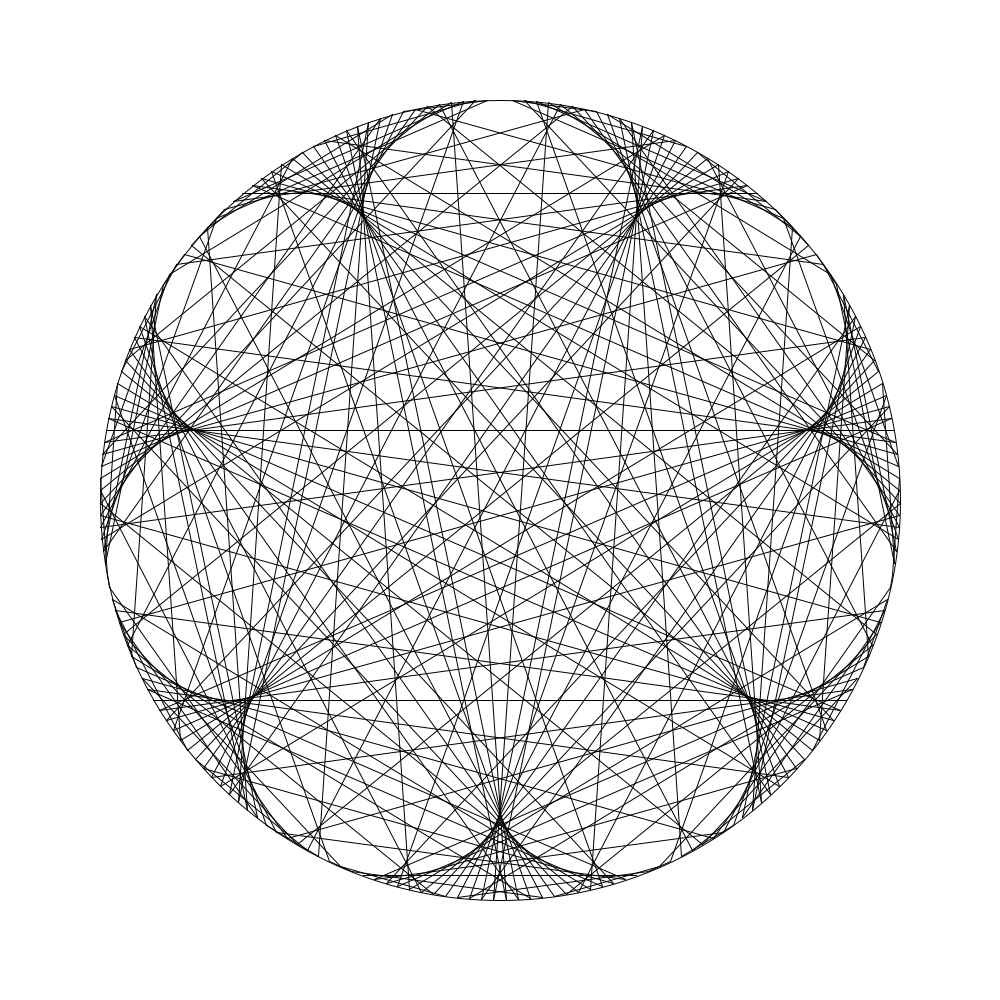}& 8 & \\
    9 &
   \includegraphics[align=c, width=\linewidth, height=\linewidth, keepaspectratio]{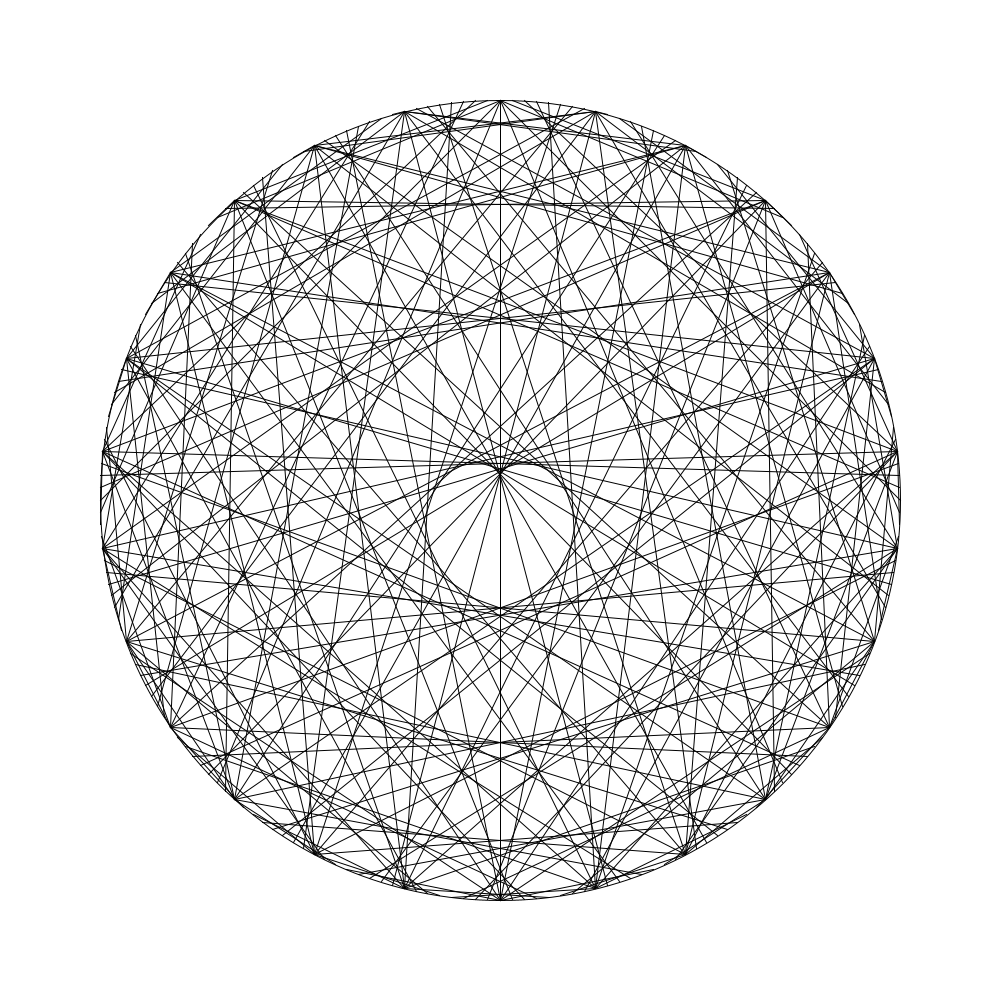} & \includegraphics[align=c, width=\linewidth, height=\linewidth, keepaspectratio]{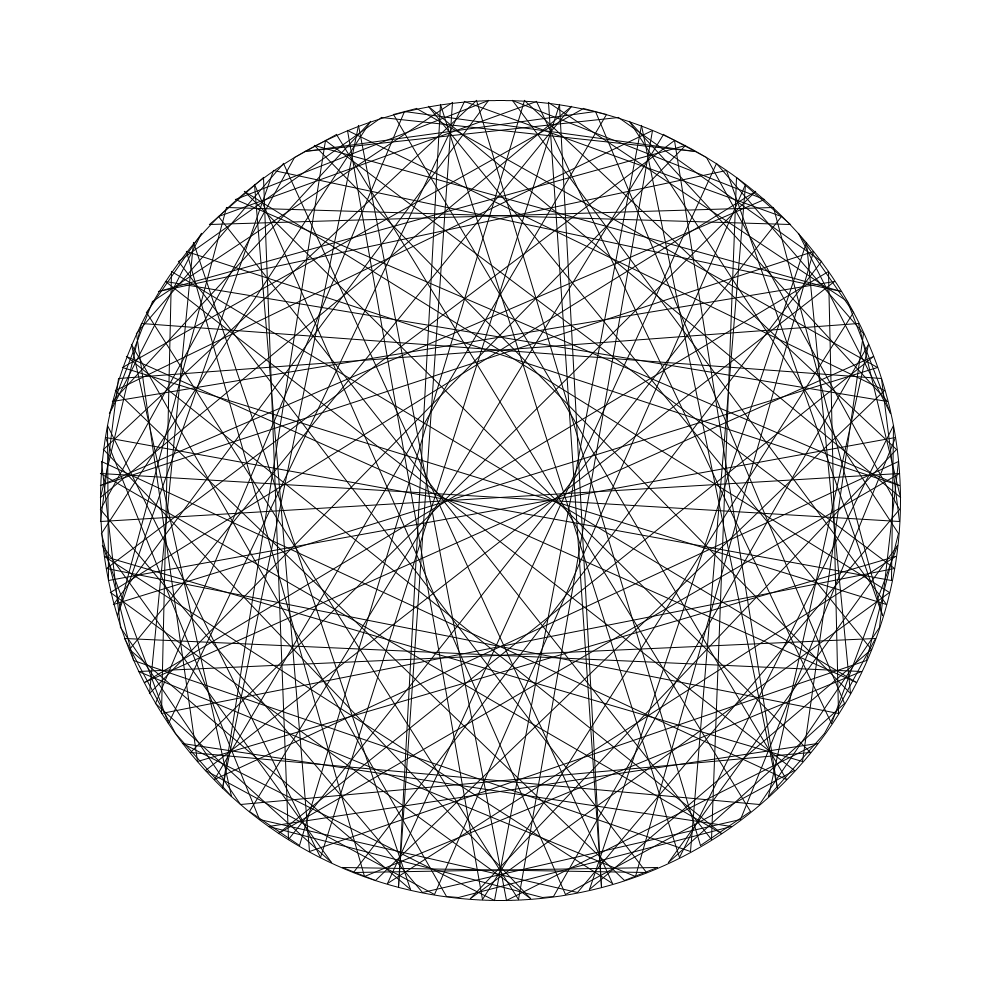}& \includegraphics[align=c, width=\linewidth, height=\linewidth, keepaspectratio]{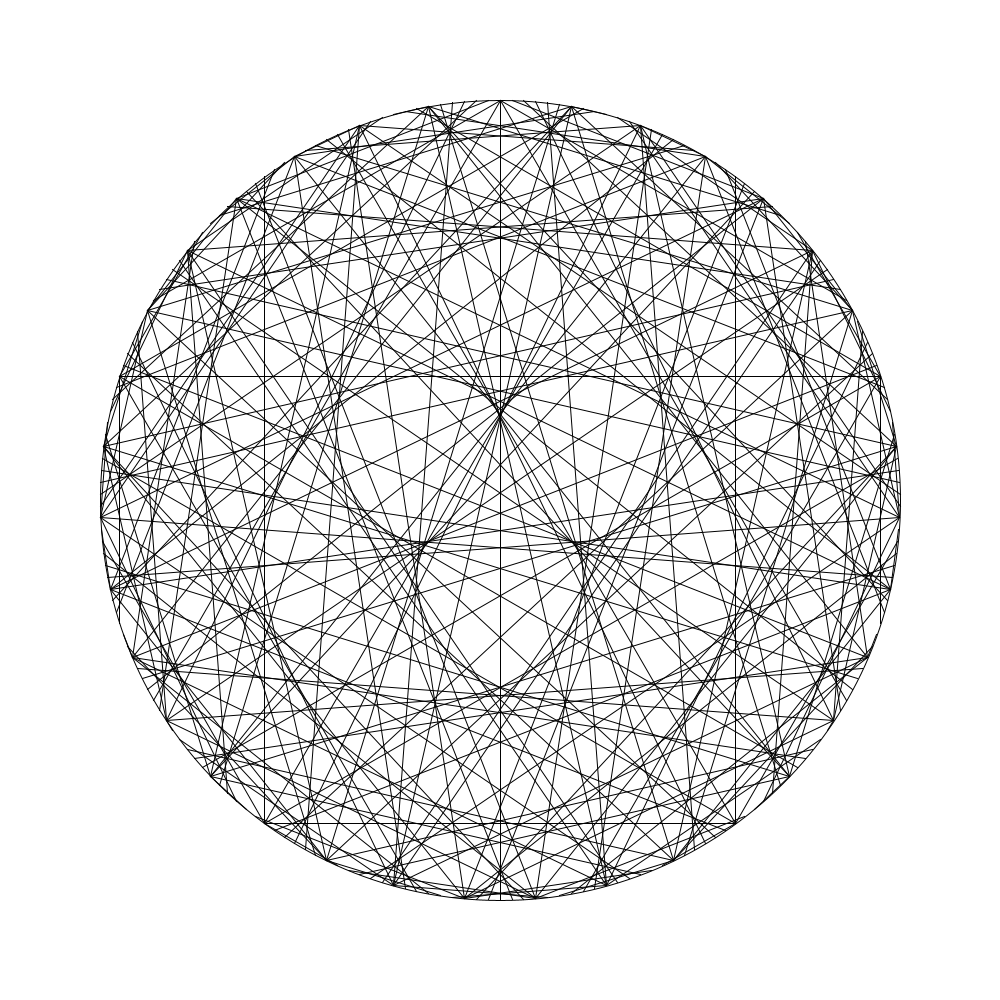}& 
   \includegraphics[align=c, width=\linewidth, height=\linewidth, keepaspectratio]{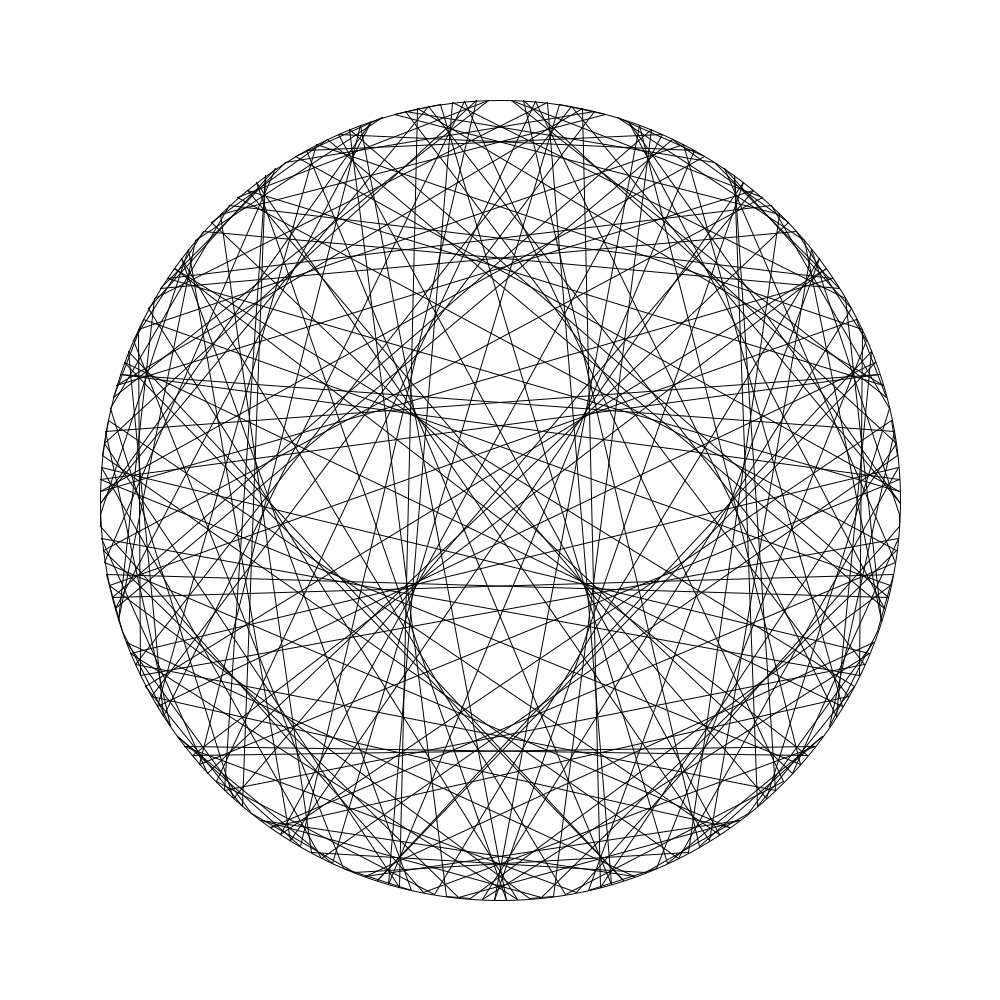}& 
   \includegraphics[align=c, width=\linewidth, height=\linewidth, keepaspectratio]{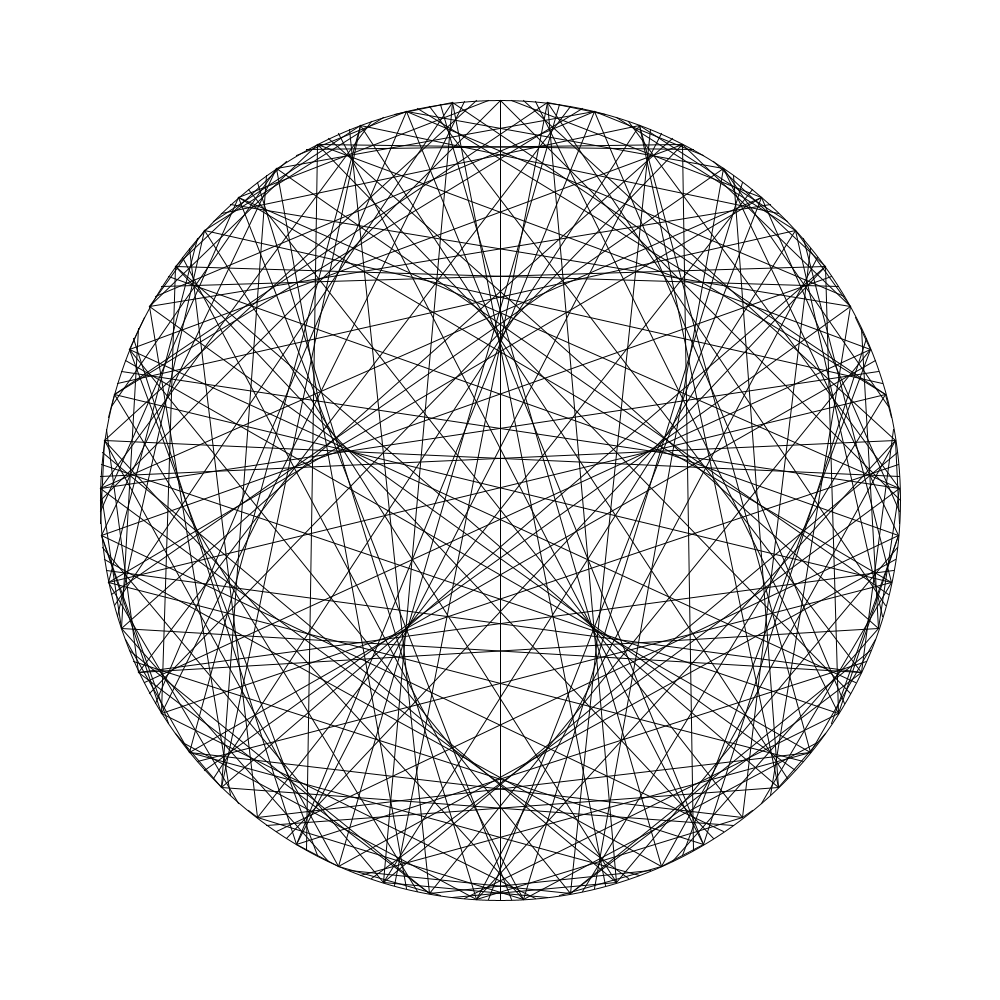}& 
   \includegraphics[align=c, width=\linewidth, height=\linewidth, keepaspectratio]{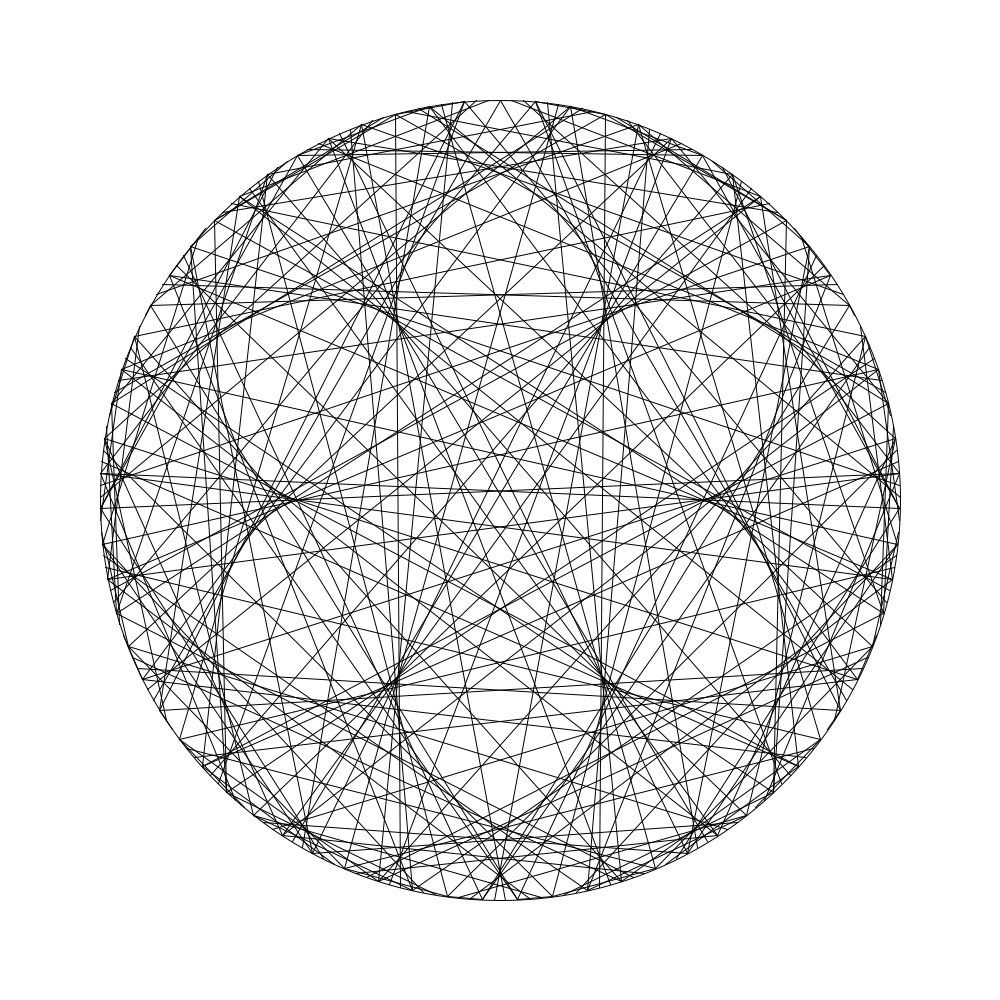}&
   \includegraphics[align=c, width=\linewidth, height=\linewidth, keepaspectratio]{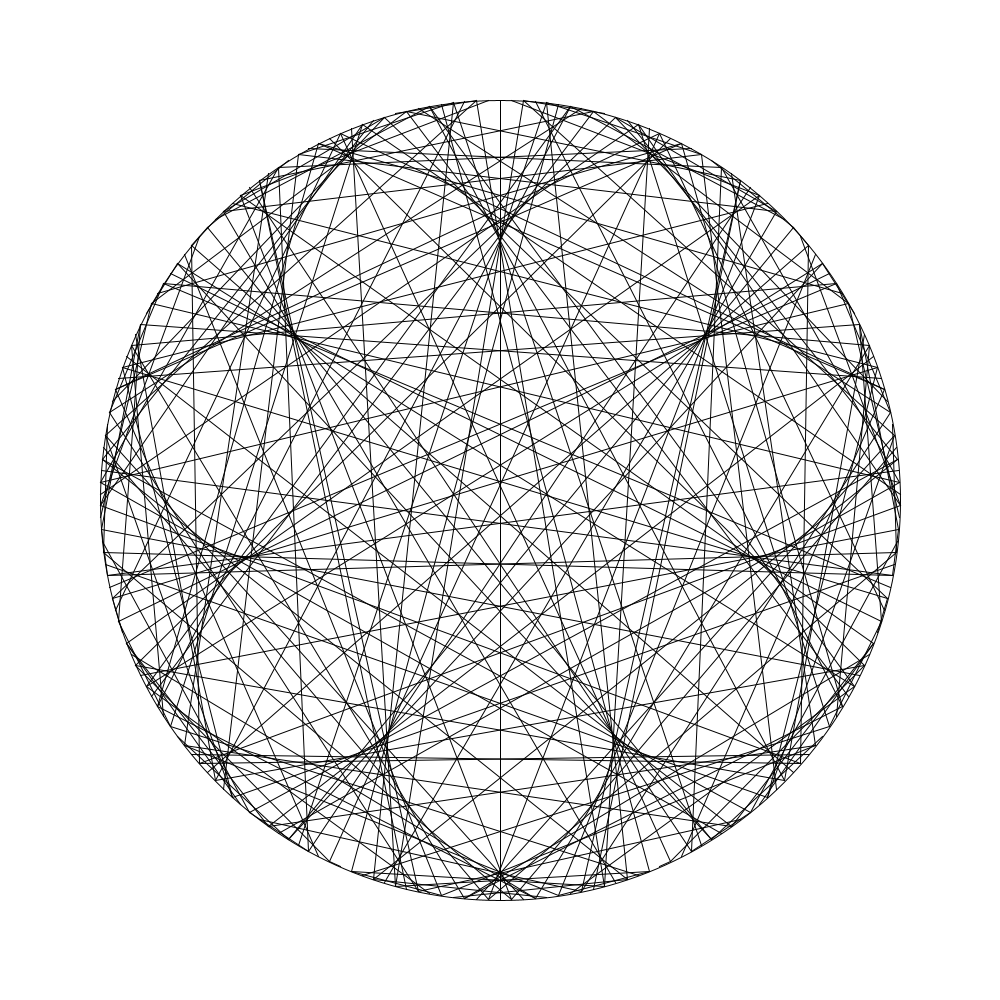}& 
    \includegraphics[align=c, width=\linewidth, height=\linewidth, keepaspectratio]{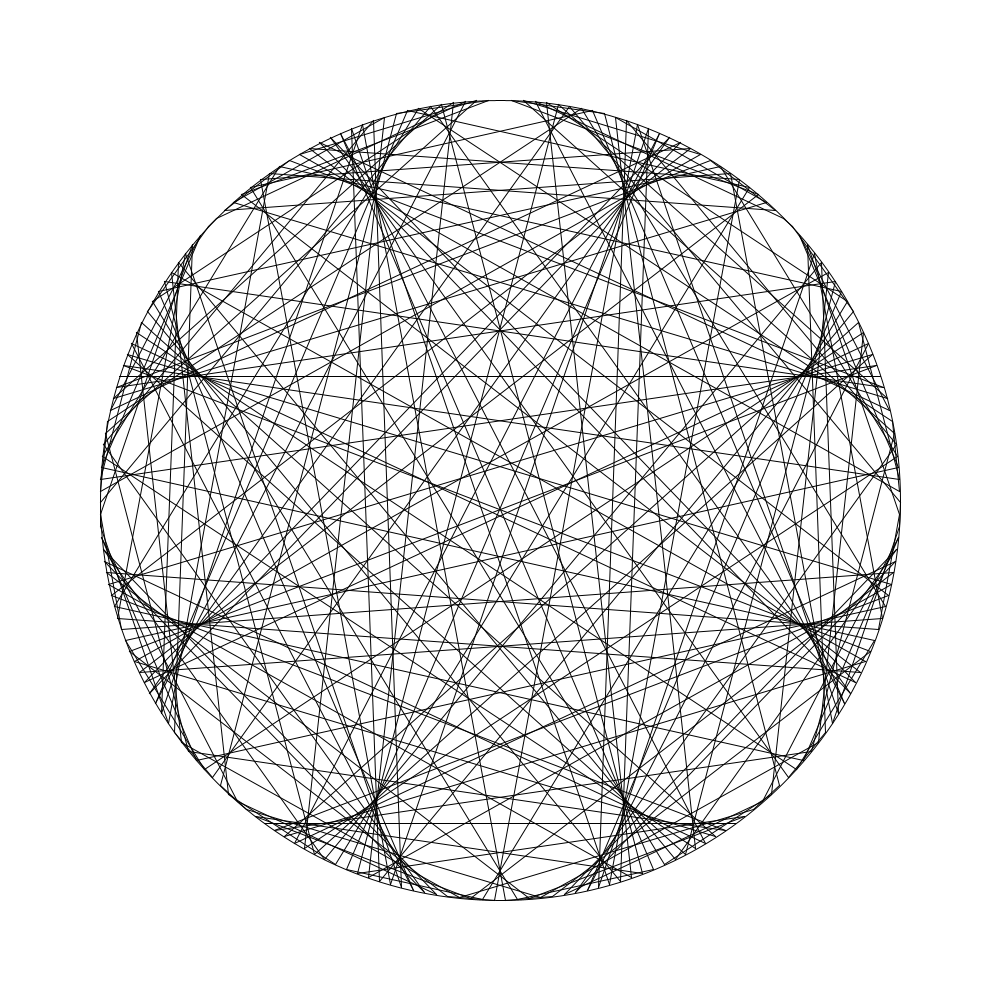}& \\
\end{tabularx}

\caption{A ``grid of graphs''! Each graph is of the form $\MMT(m, \lceil \frac{m}{b} \rceil)$ for some positive integers with $b < m$. The rows are indexed by the value for $b$ and the columns are indexed by $r$ where $m \equiv r \mod b$. We have taken $m$ to be approximately 200 to see the design clearly.}\label{fig:table_of_tables}
\end{figure}

\begin{figure}
     \centering
     \begin{subfigure}[c]{0.23\textwidth}
         \centering
         \includegraphics[width=\textwidth]{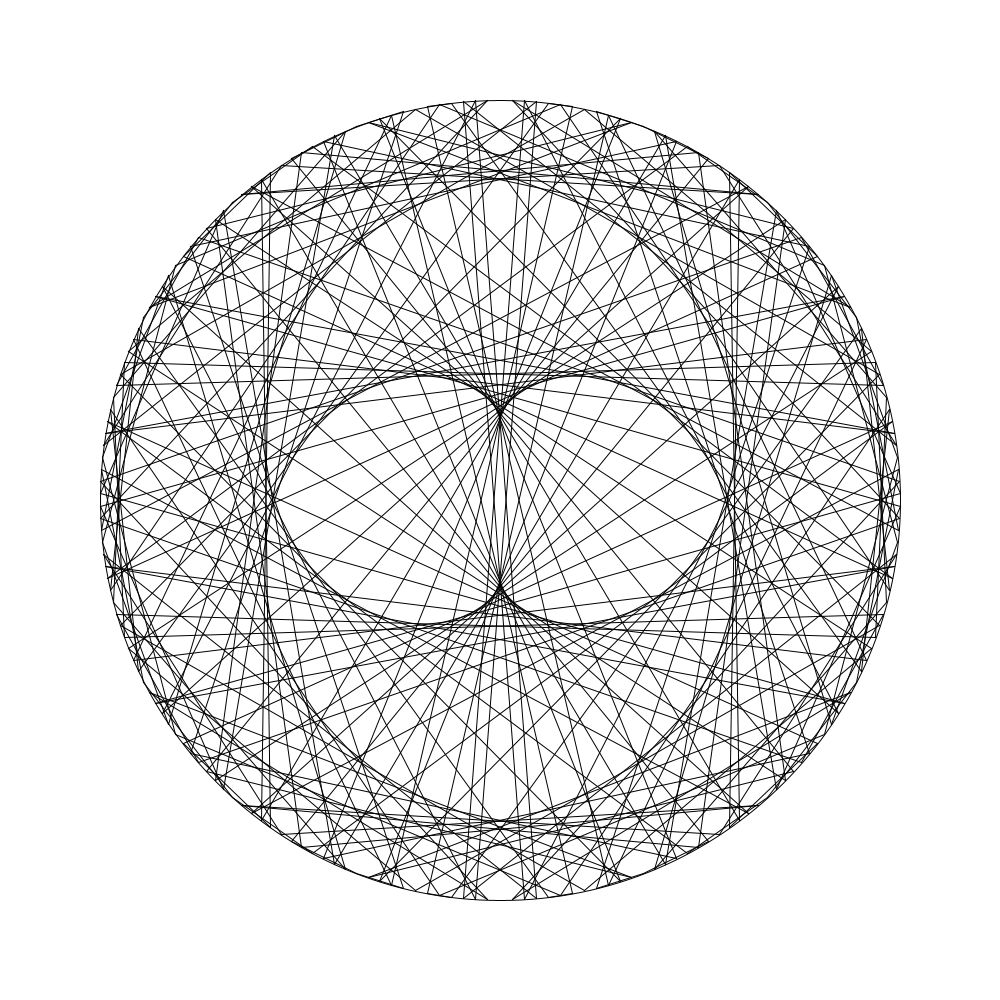}
         \caption{$\MMT(206,35)$}
         \label{fig:MMT-206-35}
     \end{subfigure}
     \hfill
     \begin{subfigure}[c]{0.23\textwidth}
         \centering
         \includegraphics[width=\textwidth]{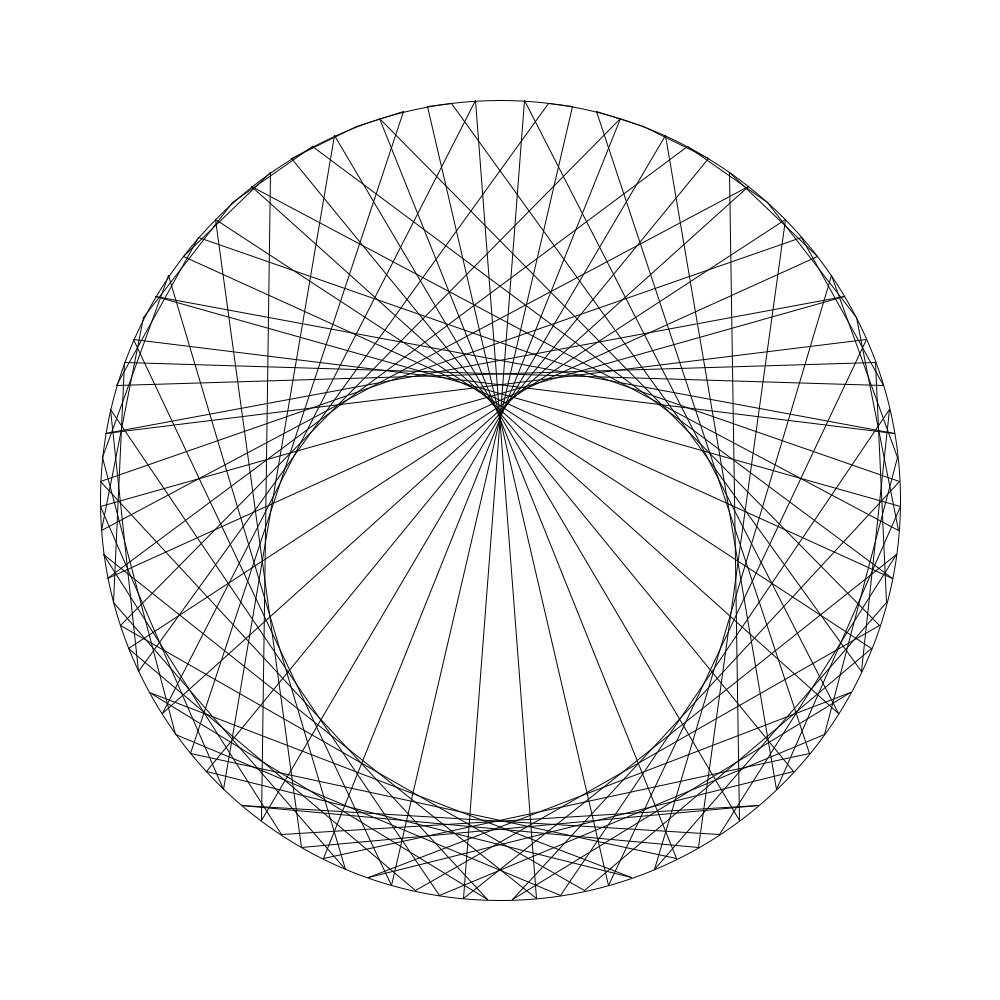}
         \caption{All chords with initial point even.}
         \label{fig:MMT-206-35-0mod2}
     \end{subfigure}
     \hfill
     \begin{subfigure}[c]{0.23\textwidth}
         \centering
         \includegraphics[width=\textwidth]{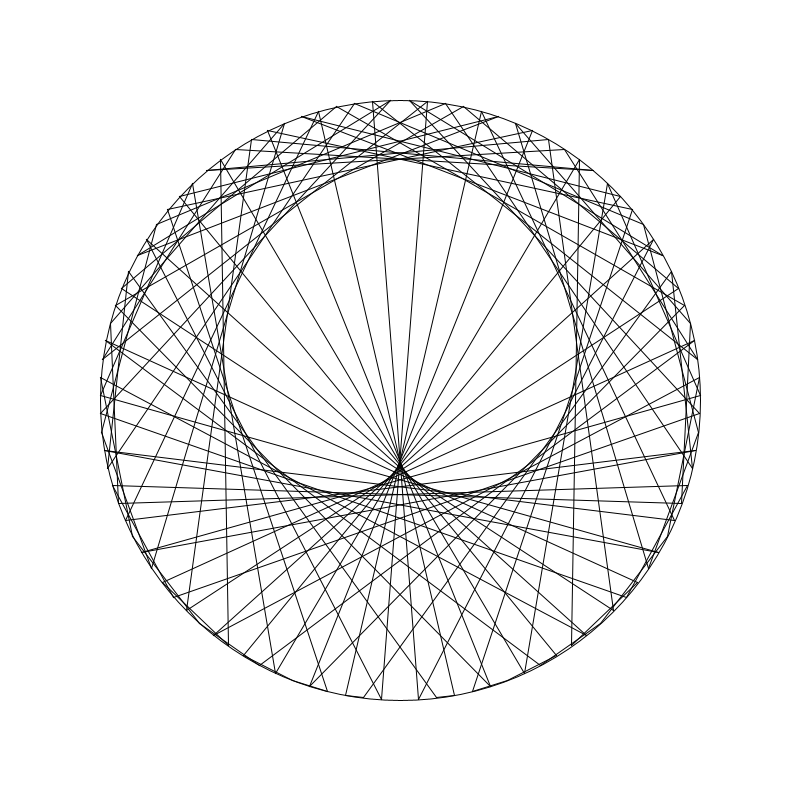}
         \caption{All chords with initial point odd.}
         \label{fig:MMT-206-35-1mod2}
     \end{subfigure}
    \hfill
    \begin{subfigure}[c]{0.26\textwidth}
         \centering
        \includegraphics[width=\textwidth]{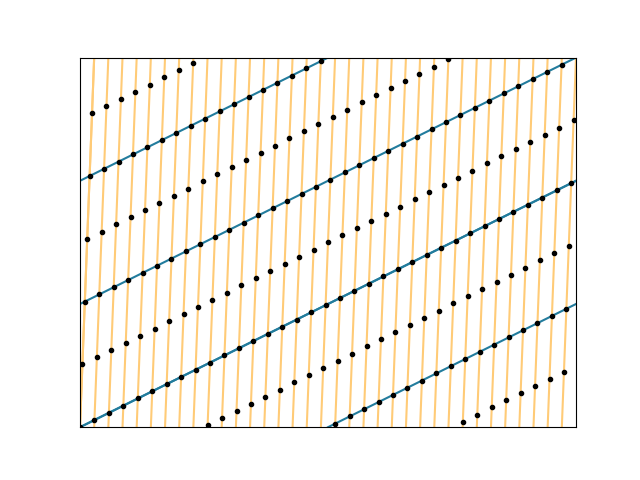}
         \caption{The corresponding planet dance and its discrete sampling.}
         \label{fig:knots-for-206-35}
     \end{subfigure}
    \caption{$\MMT(206, 35)$ ``reduces'' to two copies of $\MMT(103, 35)$, one rotated halfway around the circle. Figure~\ref{fig:knots-for-206-35} depicts the corresponding linear torus loops. The yellow lines represent $\moon{1}{35}$, with the black dots showing a sampling at rate $206$. The blue lines represent $\moon{3}{2}$.}
    \label{fig:reducing tables 6-3 ex}
\end{figure}

\begin{figure}
     \centering
     \begin{subfigure}[t]{0.23\textwidth}
         \centering
         \includegraphics[width=\textwidth]{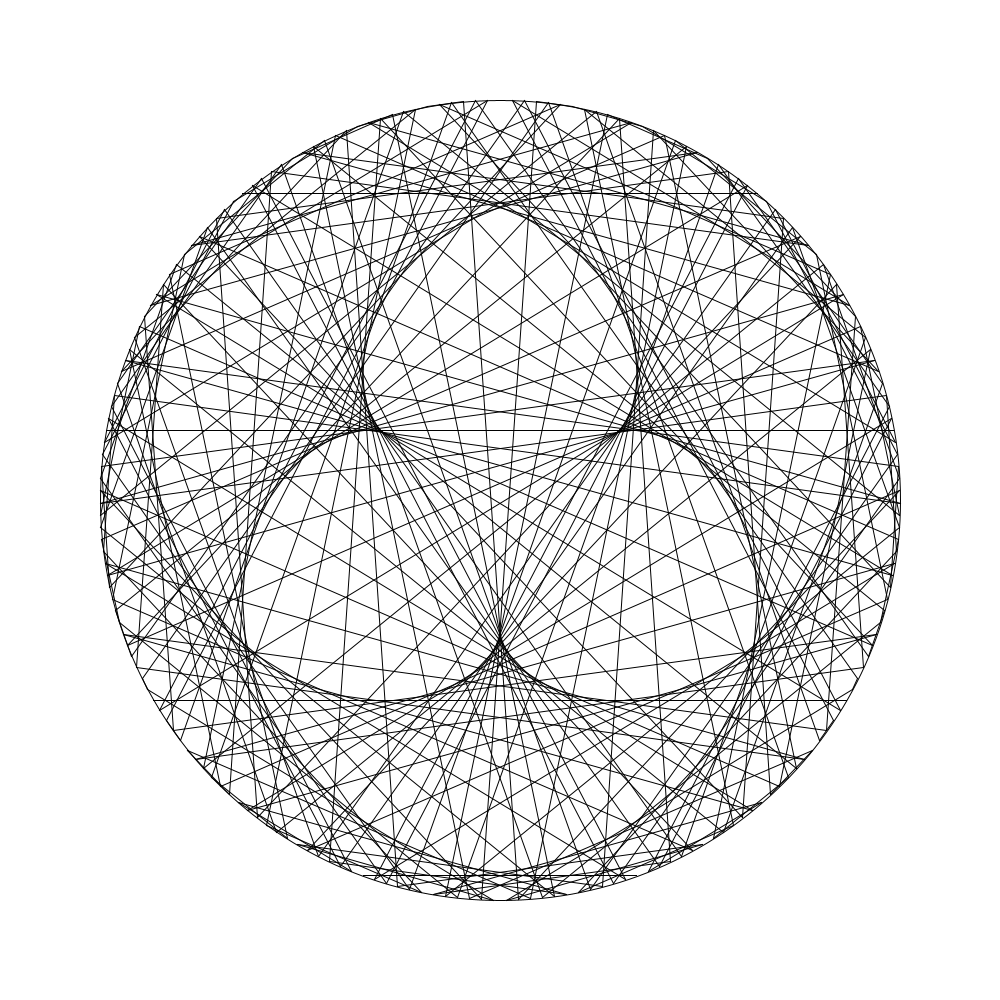}
         \caption{$\MMT(207,35)$}
         \label{fig:MMT-207-35}
     \end{subfigure}
     \begin{subfigure}[t]{0.3\textwidth}
         \centering
         \includegraphics[width=\textwidth]{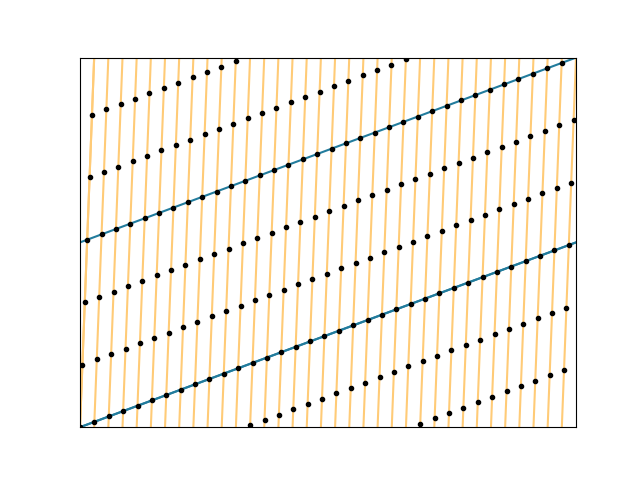}
         \caption{The corresponding planet dance and its discrete sampling.}
         \label{fig:knots-207-35}
     \end{subfigure}
     \hfill \\
     \begin{subfigure}[t]{0.23\textwidth}
         \centering
         \includegraphics[width=\textwidth]{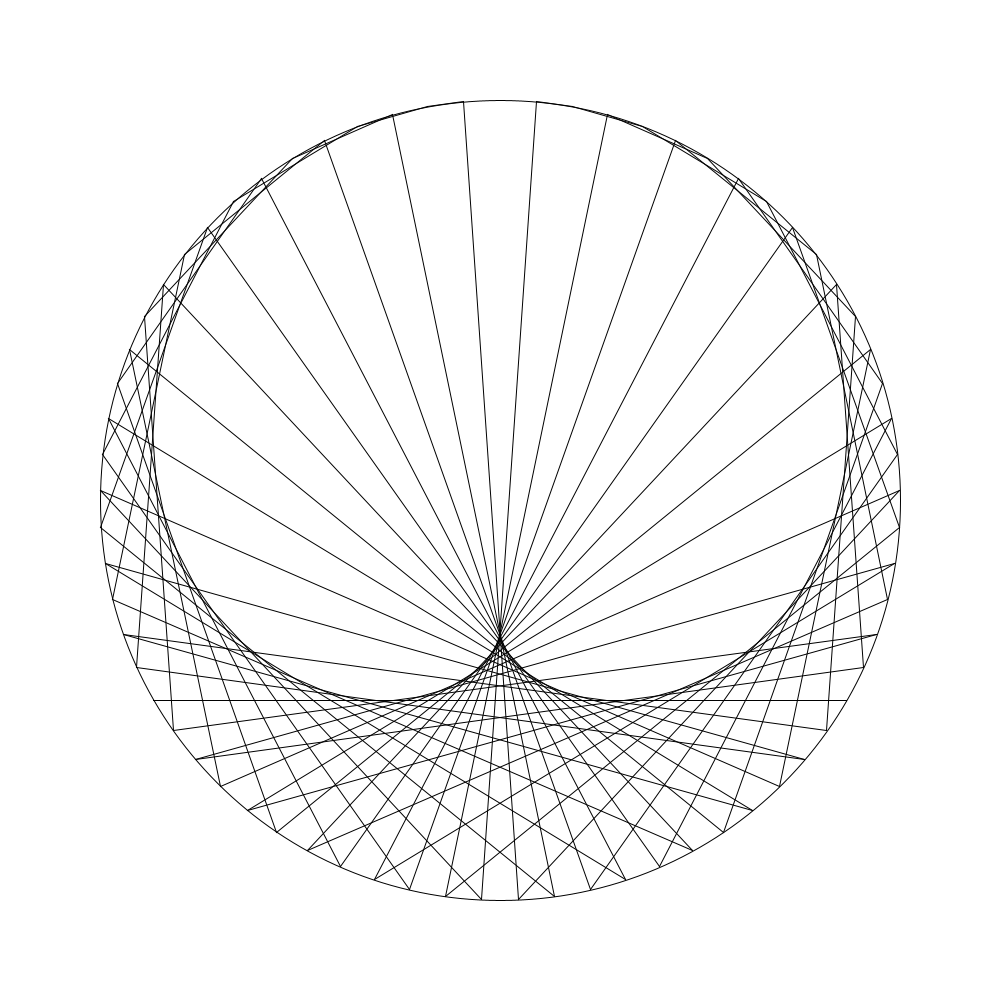}
         \caption{All chords with initial point $p \equiv 0 \mod 3$}
         \label{fig:MMT-207-35-0mod3}
     \end{subfigure}
     \begin{subfigure}[t]{0.23\textwidth}
         \centering
         \includegraphics[width=\textwidth]{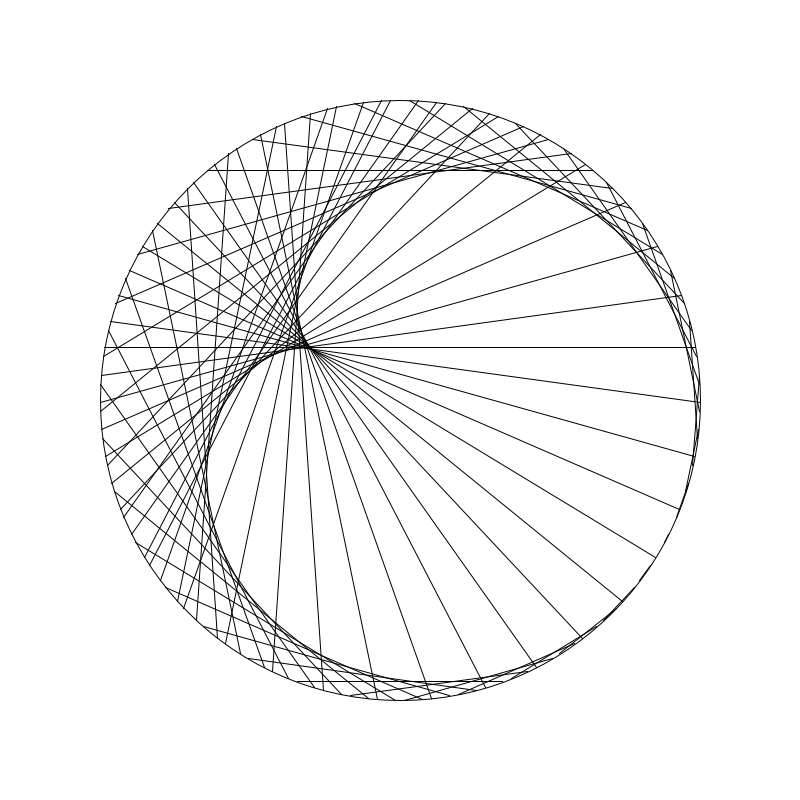}
         \caption{All chords with initial point $p \equiv 1 \mod 3$}
         \label{fig:MMT-207-35-1mod3}
     \end{subfigure}
      \begin{subfigure}[t]{0.23\textwidth}
         \centering
         \includegraphics[width=\textwidth]{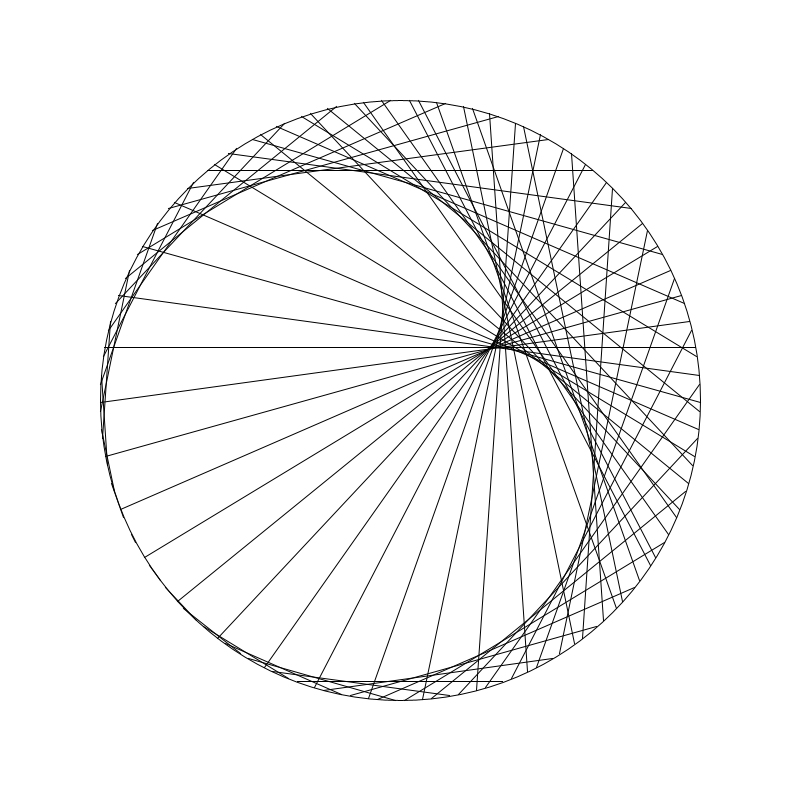}
         \caption{All chords with initial point $p \equiv 2 \mod 3$.}
         \label{fig:MMT-207-35-2mod3}
     \end{subfigure}
    \caption{$\MMT(207, 35)$ ``reduces'' to three copies of $\MMT(69, 35)$. Figure~\ref{fig:knots-207-35} depicts the corresponding linear torus loops. The yellow lines represent $\moon{1}{35}$, with the black dots showing a sampling at rate $207$. The blue lines represent $\moon{2}{1}$.}
    \label{fig:reducing tables 6-2 ex}
\end{figure}

I discovered these patterns before considering modular stitch graphs from the perspective of planet dances. But now, with this other perspective, the mysteries are unraveled. Let $m \equiv r \mod b$. Then we can write $a = \lceil \frac{m}{b} \rceil$ as 
\[a = \frac{m + (b - r)}{b}.\]
If we do some algebraic rearranging, then we have the following relationship.
\begin{align*}
    ba &= m + b - r \\
    ba - (b-r) &= m 
\end{align*}
With our knowledge about the intersection number of linear loops on the torus, we recognize that the above equation says that $\moon{b}{b - r}$ and $\moon{1}{a}$ will intersect $m$ times. Except, there is one missing link: this is only true if both planet dances are in reduced form, i.e. if $\gcd(b, b - r) = \gcd(b, r) = 1$. In this case, $\MMT(m,a)$ will be an $m$-sampling of $\moon{b}{b - r}$ by Proposition \ref{prop:aliasing-planet-dances}. We have already seen one example of this case: when $b = 3$ and $r = 1$. If we set $a = 34$, then $m = 3\cdot 34 - 2 = 100$. Following the discussion in Section \ref{sec:planet-dance-aliasing}, $\MMT(100, 34)$ is a $100$-sampling of $\moon{3}{2}$.

\begin{cor}[of Proposition \ref{prop:aliasing-planet-dances}]\label{cor:upper-odd-table-reduced}
    Let $m$ and $b$ be positive integers with $b < m$ such that $m \equiv r \mod{b}$ for $0 < r < b$. If $\gcd(r,b) = 1$, then  $\MMT(m, \lceil \frac{m}{b} \rceil)$ will be an $m$-sampling of $\moon{b}{b - r}$.
\end{cor}

However, it may be the case that $\gcd(b, b-r) = d > 1$. Upon examination of the ``grid of graphs'', one might notice that many of the images look like overlays of other graphs. For example, when $r = 2$ and $b = 6$, the design looks like two copies of a graph with $r = 1$ and $b = 3$, but one copy is rotated halfway around the circle. When I first found this family of graphs, I noticed this visually and I decided to test my idea experimentally. Figure~\ref{fig:reducing tables 6-3 ex} shows a member of the $r = 2$, $b = 6$ family ($\MMT(206, 35)$) along with two subsets of this modular stitch graph. The first shows only the chords originating from even points (Figure \ref{fig:MMT-206-35-0mod2}) and the second shows only the chords originating from odd points (Figure \ref{fig:MMT-206-35-1mod2}). Another example of this overlay phenomenon is shown in Figure \ref{fig:reducing tables 6-2 ex}.

Every example that I tested gave the same result. When $\gcd(b,b-r) = d$, the modular stitch graph  $\MMT(m, \lceil \frac{m}{b} \rceil)$ looked like $d$ copies of $\MMT(\frac{m}{d}, \lceil \frac{m}{b} \rceil)$, rotated symmetrically around the circle. I managed to devise a number theoretic proof of this fact but it was not very illuminating. With our topological perspective, we can see what is going on more clearly. When we draw the corresponding linear torus loops for the examples in Figures \ref{fig:reducing tables 6-3 ex} and \ref{fig:reducing tables 6-2 ex}, we see that the sampling points trace out copies of $\moon{b}{b-r}$ which are shifted up. Although we do not get these ``phantom copies'' of linear loops from the planet dance construction, they do appear by sampling.

The following theorem addresses a general case of overlaying tables, which we will then apply to our specific example.

\begin{thm}[Overlaying planet dances]\label{thm:overlays}
    Let $\moon{\alpha}{\beta}$ be a planet dance in reduced form with $\alpha \neq 0$ and let $a \in \Z$ be any integer. For $m = |\alpha a - \beta|$ we have $\sample{m}{\alpha, \beta} = \sample{m}{1, a}$ by Proposition \ref{prop:aliasing-planet-dances}. Then, for any positive integer $d$, $\sample{dm}{1, a}$ consists of $d$ disjoint sets of chords $\Sc_0, \dots, \Sc_{d-1}$. For each $k$, $|\Sc_k| = m$ and $\Sc_k$ is contained in a copy of $\moon{\alpha}{\beta}$ rotated by $\frac{k}{d(\alpha - \beta)}$ around the circle.
\end{thm}

Figure \ref{fig:overlaying} encapsulates the statement of Theorem \ref{thm:overlays} and its proof. The red points are the initial sample points, those in $\sample{m}{1,a}$. When we increase the sampling by a factor of 3 and plot $\sample{3m}{1, a}$, we get the additional points in the picture. All the red points lie on the red line, representing $\moon{2}{1}$. The additional points in the picture can be divided into two sets, those that lie on the blue line and those on the yellow line. These lines are copies of the red line shifted on the vertical axis of the square, and this corresponds to a rotation of the planet dance around the circle. 

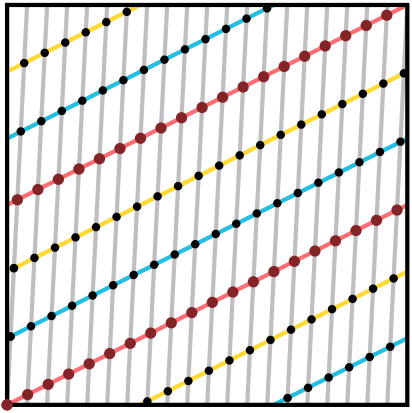
\begin{figure}[H]
    \centering
    \resizebox{0.4\textwidth}{!}{
    \begin{tikzpicture}[line width=0.3pt]  
    \foreach \i in {0,...,19} {
        \pgfmathsetmacro\xstart{\i/20}
        \pgfmathsetmacro\xend{(\i+1)/20}
        \draw[lightgray] (\xstart,0) -- (\xend,1);
    }
    
    \draw[red] (0,0) -- (1,0.5);
    \draw[red] (0,0.5) -- (1,1);

    \draw[blue] (0, 0.16667) -- (1,0.666);
    \draw[blue] (0,0.666) -- (0.666,1);
    \draw[blue] (0.666,0) -- (1,0.16667);

    \draw[yellow] (0, 0.333) -- (1,0.8333);
    \draw[yellow] (0,0.8333) -- (0.333, 1);
    \draw[yellow] (0.333, 0) -- (1,0.333);

    \draw[black, line width=0.3pt] (0,0) rectangle (1,1);

     \foreach \i in {0,...,39} {
        \pgfmathsetmacro\x{mod(\i/39,1)}
        \pgfmathsetmacro\y{mod((20*\i)/39,1)}
        \fill[darkred] (\x,\y) circle (0.4pt);
    }

    \foreach \i in {0,...,39} {
        \pgfmathsetmacro\x{mod(\i/39 + (1/117),1)}
        \pgfmathsetmacro\y{mod(((20*\i)/39) + (20/117),1)}
        \fill[black] (\x,\y) circle (0.3pt);
    }

    \foreach \i in {0,...,39} {
        \pgfmathsetmacro\x{mod(\i/39 + (2/117),1)}
        \pgfmathsetmacro\y{mod(((20*\i)/39) + (40/117),1)}
        \fill[black] (\x,\y) circle (0.3pt);
    }

\end{tikzpicture}}
    \caption{A figure summarizing Theorem \ref{thm:overlays}. Each set of colored lines represent a copy of $\moon{2}{1}$ rotated around the circle.}
    \label{fig:overlaying}
\end{figure}

\begin{proof}[proof of Theorem \ref{thm:overlays}]
    Recall that $\moon{\alpha}{\beta}$ is given by a line in $\R^2/ \Z^2$ with the following parametric equations.
    \[x = \alpha t, \hspace{10pt} y = \beta t, \hspace{10pt} \text{or} \hspace{10pt} y = \tfrac{\beta}{\alpha}x \hspace{10pt} \text{since } \alpha \neq 0\]
    From the discussion in Section \ref{sec:moon-systems-as-paths}, we know that $\sample{dm}{1,a}$ is an additive group generated by $(\tfrac{1}{dm}, \tfrac{a}{dm})$. We can realize $\sample{m}{1, a}$ as a subgroup of $\sample{dm}{1, a}$ generated by the point $(\frac{d}{dm}, \frac{ad}{dm})~=~(\frac{1}{m}, \frac{a}{m})$. Let $\Sc_0$ be this copy of $\sample{m}{1, a}$. The cosets of this subgroup are given by
    \[\Sc_k = \left\{ \left(\frac{j}{m}, \frac{a j}{m}\right) + \left(\frac{k}{dm}, \frac{ak}{dm}\right) \: \bigg| \: 0 \leq j \leq m -1 \right\}\]
    for $1 \leq k \leq d-1$. We know that $\Sc_0$ lies on $\moon{\alpha}{\beta}$ as given. We then claim that $\Sc_k$ lies on the line given by 
    \begin{equation}\label{eq:line-for-Qk}
        x = \alpha t, \hspace{10pt} y = \beta t + \frac{k}{d\alpha}.
    \end{equation}
    Notice that all the points in $\Sc_k$ are points in $\Sc_0$ translated by the vector $\left(\frac{k}{dm}, \frac{ak}{dm}\right)$ by definition. In particular, this vector is the same for all points in $\Sc_0$. Then, since all the points in $\Sc_0$ are colinear, and translating a line leaves the slope invariant, we conclude that $\Sc_k$ lies on a line of the form
    \[y = \frac{\beta}{\alpha}x + C_k\]
    where $C_k$ is a constant. We then use the Triangle Proportionality theorem to find that the $y$-intercept of the line is $\frac{k}{d\alpha}$ as explained in the caption of Figure \ref{fig:triangle-thm}. 

    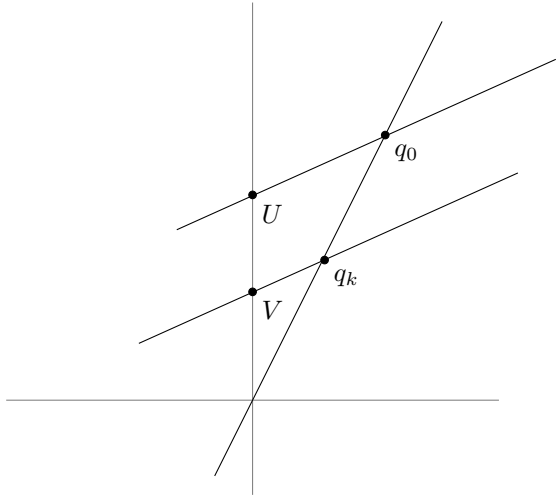
\begin{figure}[H]
        \centering
        \resizebox{0.5\textwidth}{!}{
        \begin{circuitikz}
        \tikzstyle{every node}=[font=\normalsize]
        \draw [color={gray}, short] (6.25,9.75) -- (6.25,3.25);
        \draw [color={gray}, short] (3,4.5) -- (9.5,4.5);
        \draw [short] (5.75,3.5) -- (8.75,9.5);
        \draw [ fill={black} ] (7.2,6.35) circle (0.05cm) node[anchor=north west] {\normalsize $q_k$} ;
        \draw [ fill={black} ] (8,8) circle (0.05cm) node[anchor=north west] {\normalsize $q_0$} ;
        \draw [short] (5.25,6.75) -- (10.25,9);
        \draw [short] (4.75,5.25) -- (9.75,7.5);
        \draw [ fill={black} ] (6.25,7.21) circle (0.05cm) node[anchor=north west] {\normalsize $U$} ;
        \draw [ fill={black} ] (6.25,5.93) circle (0.05cm) node[anchor=north west] {\normalsize $V$} ;
        \end{circuitikz}
        }
        \caption{Let $q_k = (\frac{k}{dm}, \frac{ak}{dm})$ and let $q_0 = (\frac{1}{m}, \frac{a}{m})$. The point $q_{k}$ lies between $(0,0)$ and $q_0$ on the line $y = ax$. The Triangle Proportionality Theorem says that $\frac{\norm{U - V}}{\norm{V}} = \frac{\norm{q_{0} - q_{k}}}{\norm{q_k}}$. Since $U = (0, \frac{1}{\alpha})$, we find that $V = (0, \frac{k}{d\alpha}$).}
    \label{fig:triangle-thm}
\end{figure}
    The last step is to show why the line given by (\ref{eq:line-for-Qk}) is a copy of $\moon{\alpha}{\beta}$ rotated by $\frac{k}{d(\alpha - \beta)}$. We notice that every planet dance $\moon{\alpha}{\beta}$ starts at $t = 0$ with both planets in the same position. This is exactly an intersection with the line $\moon{1}{1}$. But since (\ref{eq:line-for-Qk}) has a non-zero $y$-intercept, the corresponding planets start at different positions at $t = 0$. So the rotation amount of the corresponding planet dance is determined by the point of intersection between the line (\ref{eq:line-for-Qk}) with $\moon{1}{1}$. We solve for such a $t$.
    \begin{align*}
        \alpha t &= \beta t + \frac{k}{d\alpha} \\
        (\alpha - \beta)t &= \frac{k}{d\alpha} \\
        t &= \frac{k}{d\alpha(\alpha - \beta)}
    \end{align*}
    And then we find $x = y = \frac{k}{d(\alpha - \beta)}$ at this value for $t$.
\end{proof}

Now we apply this general theory of overlays to extend Corollary \ref{cor:upper-odd-table-reduced} to cases where $\gcd(b,r) > 1$. We give an explicit description for modular stitch graphs of the form $\MMT(m, \lceil \frac{m}{b} \rceil)$ and $\MMT(m, \lfloor \frac{m}{b} \rfloor)$. 

\begin{prop}\label{prop:odd-tables-non-reduced}
    Let $m$ and $b$ be positive integers such that $b < m$ and $b \nmid m$. Let $m \equiv r \mod b$ for $0 < r < b$ and let $d = \gcd(b,r)$. Then $\MMT(m, \lceil \frac{m}{b} \rceil)$ and $\MMT(m, \lfloor \frac{m}{b}\rfloor)$ are each composed of $d$ sets $\Sc_0, \dots \Sc_{d-1}$ each consisting of $\frac{m}{d}$ chords.
    \begin{itemize}
        \item For $\MMT(m, \lceil \frac{m}{b} \rceil)$, each $\Sc_k$ lies on a copy of $\moon{\frac{b}{d}}{\frac{b-r}{d}}$ rotated by $\frac{k}{r}$ around the circle.
        \item For $\MMT(m, \lfloor \frac{m}{b} \rfloor)$, each $\Sc_k$ lies on a copy of $\moon{\frac{b}{d}}{\frac{-r}{d}}$ rotated by $\frac{k}{r}$ around the circle.
    \end{itemize}
\end{prop}

\begin{proof}
    We prove this result for a graph of the form $\MMT(m, \lceil \frac{m}{b} \rceil)$. The proof for $\MMT(m, \lfloor \frac{m}{b} \rfloor)$ follows similarly.

    Note that $d = \gcd(r, b) = \gcd(m, b)$ so we can write $m = dm'$ for $m' \in \Z^{>0}$. Then we rewrite $a = \lceil \frac{m}{b} \rceil$.
    \[a = \frac{dm' + (b - r)}{b} = \frac{m' + \frac{b-r}{d}}{\frac{b}{d}}\]
    So we can write the relation
    \[a\left(\frac{b}{d}\right) - \left(\frac{b-r}{d}\right) = m'\]
    where all terms are integers. Next, we apply Corollary \ref{cor:upper-odd-table-reduced} and the preceding discussion to acquire $\sample{m'}{1, a} = \sample{m'}{\frac{b}{d}, \frac{b-r}{d}}$. Then we apply Theorem \ref{thm:overlays} to get the desired result for $\sample{dm'}{1, a} = \sample{m}{1, a} = \MMT(m, a)$.
\end{proof}

\section{Wrapping up and moving forward}
We will take a moment to recall where we have been and fit the pieces together. In Section 1, we introduced modular stitch graphs and asked the Big Question:
\begin{quote}
    \centering \emph{Given values for $m$ and $a$, what design will $\MMT(m, a)$ create?}
\end{quote}
In Section 2 we defined planet dances and we saw that modular stitch graphs are discrete samplings of these objects. We established a fundamental correspondence in Lemma \ref{lem:fundamental-correspondence} which gives a planet dance for each modular stitch graph. However, sometimes this correspondence doesn't capture the visual design of the modular stitch graph. We solved this conundrum with planet dance intersection in Section 3. We saw that if two planet dances intersect $m$ times, they will do so at regular intervals and an $m$-sampling of each will look identical. At this point, we have reframed our Big Question. 
\begin{quote}
    \centering \textit{A modular stitch graph is most naturally sampled from which planet dance?}
\end{quote}
By drawing planet dances on a flat torus, we find the answer. After graphing the chords of $\MMT(m,a)$ as points in the torus, we look for the closest sample point to $(0,0)$ given by $(\frac{\alpha}{m}, \frac{a\alpha}{m})$ with coordinates taken modulo 1. Then we connect $(0,0)$ to this point with the shortest path on the square torus and extend the line. This line will have a rational slope given by $\frac{\beta}{\alpha}$ for some $\beta \in \Z$ so it will form a closed loop and it will pass through all the points $(\frac{\alpha k}{m}, \frac{a(\alpha k)}{m})$ by construction. As we saw in Section \ref{sec:planet-dance-aliasing}, if $\gcd(\alpha, m) = 1$, then these will be all the points of $\MMT(m,a)$. However, if $\gcd(\alpha, m) = d > 1$, then it will pass through exactly $\frac{m}{d}$ of them. Letting $m = dm'$, they will be the points of the form $(\frac{dk}{m}, \frac{a (dk)}{m}) = (\frac{k}{m'}, \frac{ak}{m'})$. Now we have reached the discussion from Section \ref{sec:table-of-tables} in Theorem \ref{thm:overlays}. The $m'$-sampling $\sample{m'}{1, a}$ lies on the line $\moon{\alpha}{\beta}$ in $\T^2$. So the original modular stitch graph $\MMT(m,a) = \sample{dm'}{1,a}$ is composed of $d$ sets of chords which each lie on a copy of $\moon{\alpha}{\beta}$ rotated around the circle. With this, we have a complete answer to our Big Question.

All this mathematical theory may obscure the simplicity of the answer. At the end of the day, we have only to connect the dots. Given a modular stitch graph $\MMT(m,a)$, we can simply graph the set of chords as points in the square torus. And then, the most natural way to connect the dots will produce the planet dance that most closely visually represents the design of $\MMT(m,a)$.

To close, let's look back with fresh eyes at the menagerie of modular stitch graphs from Figure \ref{fig:examples-of-MMTs}. Figures 17-24 each show a modular stitch graph from Figure \ref{fig:examples-of-MMTs}, together with its sampled planet dance. Now it's up to you to connect the dots!

\begin{figure}[H]
    \centering
    \includegraphics[scale=0.4]{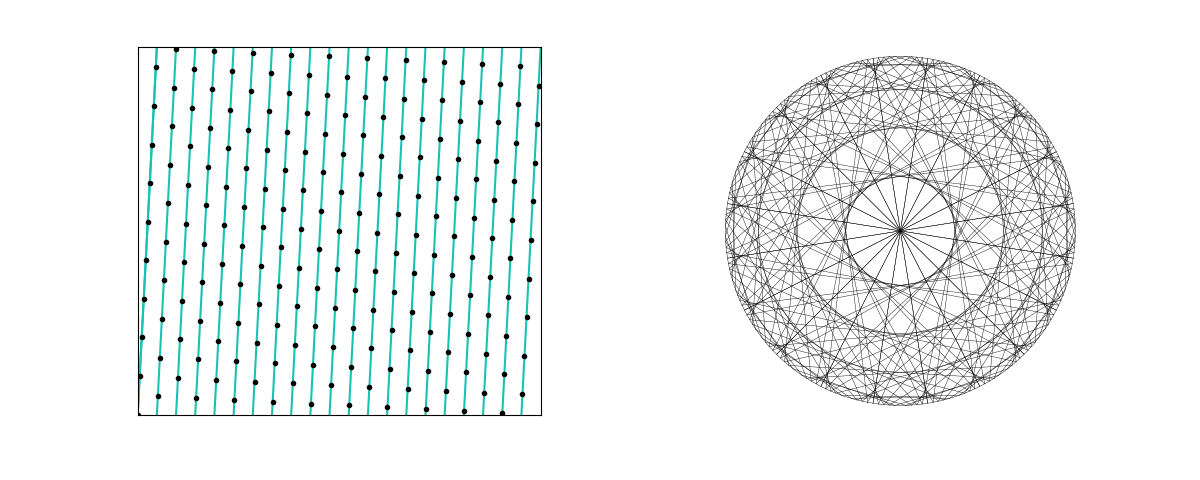}
    \caption{The planet dance $\moon{1}{21}$ sampled at rate 200 along with $\MMT(200, 21)$}
\end{figure}

\begin{figure}[H]
    \centering
    \includegraphics[scale=0.4]{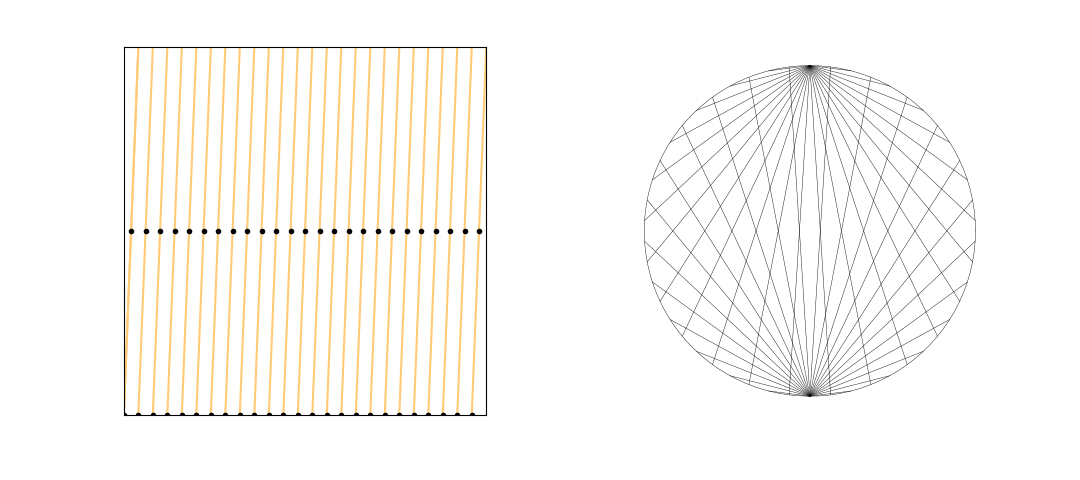}
    \caption{The planet dance $\moon{1}{25}$ sampled at rate 50 along with $\MMT(50, 25)$}
\end{figure}

\begin{figure}[H]
    \centering
    \includegraphics[scale=0.4]{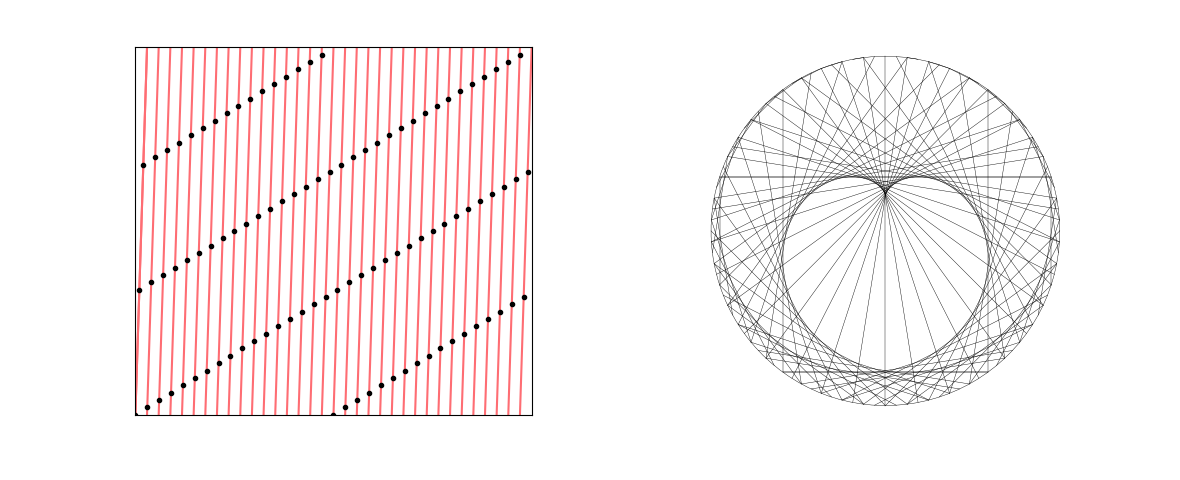}
    \caption{The planet dance $\moon{1}{34}$ sampled at rate 100 along with $\MMT(100, 34)$}
\end{figure}

\begin{figure}[H]
    \centering
    \includegraphics[scale=0.4]{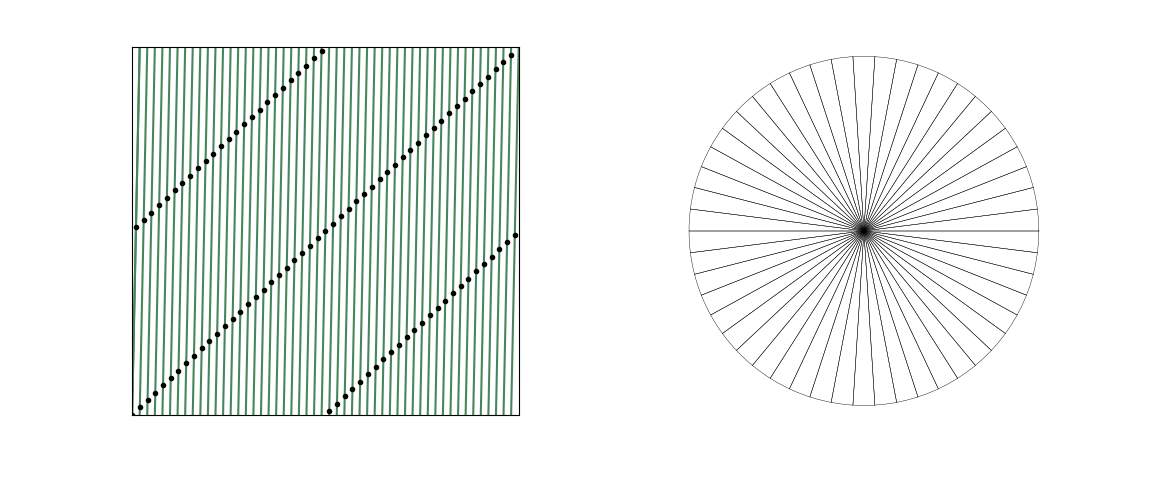}
    \caption{The planet dance $\moon{1}{51}$ sampled at rate 100 along with $\MMT(100, 51)$}
\end{figure}

\begin{figure}[H]
    \centering
    \includegraphics[scale=0.4]{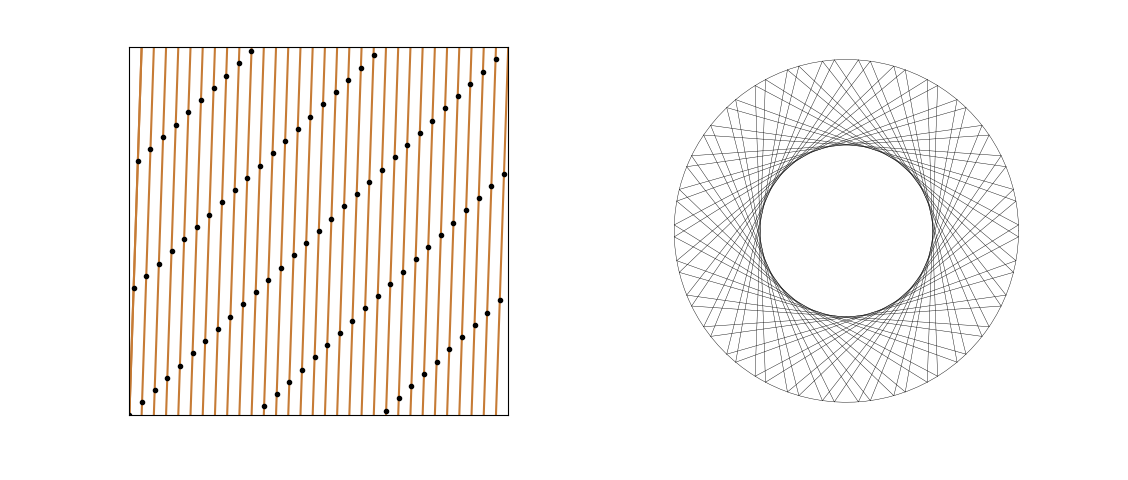}
    \caption{The planet dance $\moon{1}{31}$ sampled at rate 90 along with $\MMT(90, 31)$}
\end{figure}

\begin{figure}[H]
    \centering
    \includegraphics[scale=0.4]{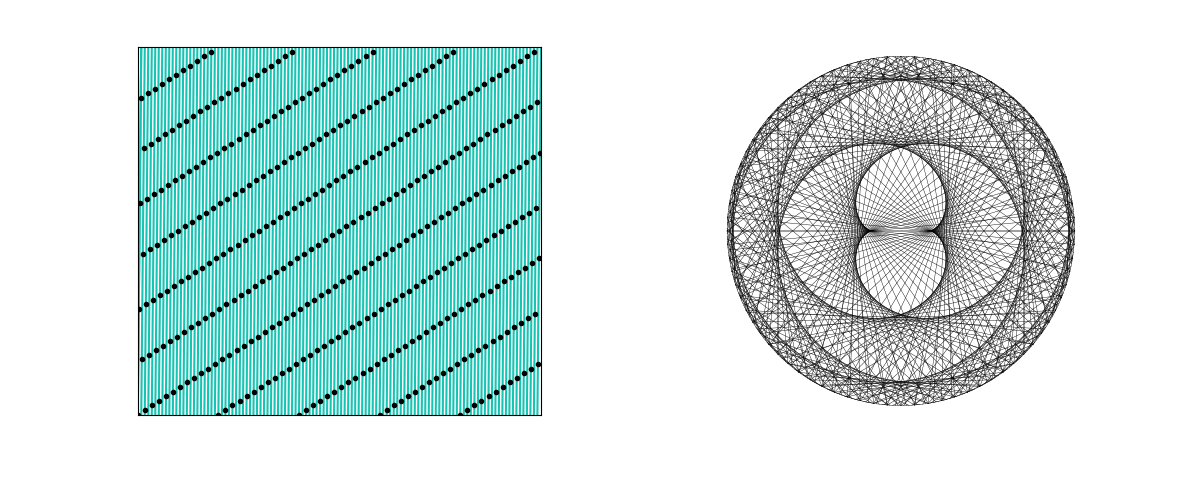}
    \caption{The planet dance $\moon{1}{115}$ sampled at rate 400 along with $\MMT(400, 115)$}
\end{figure}

\begin{figure}[H]
    \centering
    \includegraphics[scale=0.4]{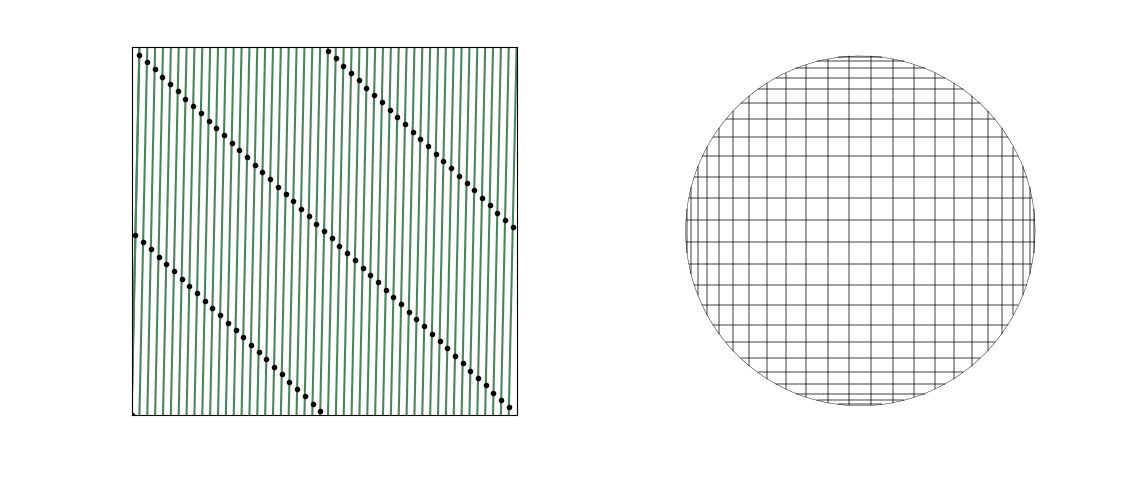}
    \caption{The planet dance $\moon{1}{49}$ sampled at rate 100 along with $\MMT(100, 49)$}
\end{figure}

\begin{figure}[H]
    \centering
    \includegraphics[scale=0.4]{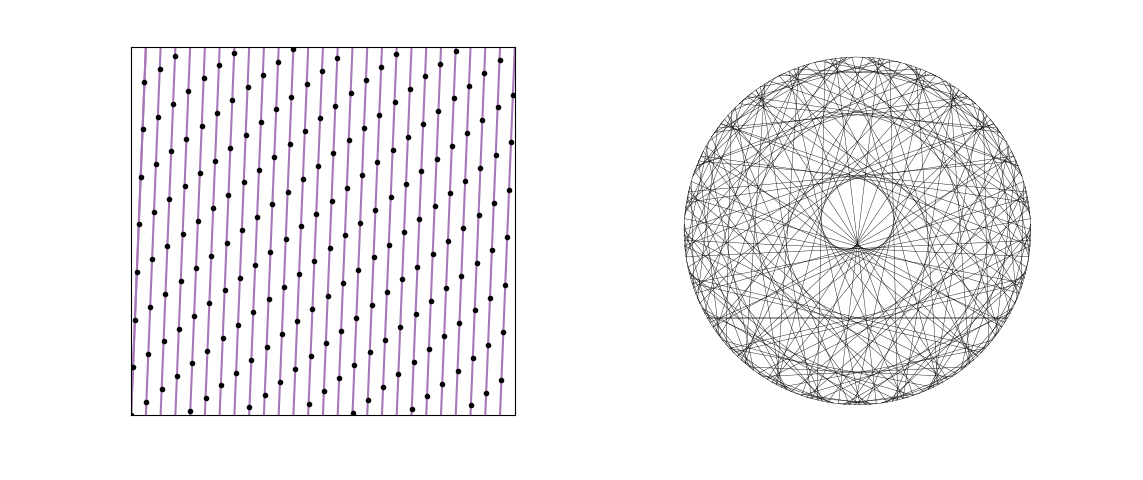}
    \caption{The planet dance $\moon{1}{26}$ sampled at rate 201 along with $\MMT(201, 26)$}
\end{figure}

\section*{Acknowledgments}
This article evolved through many helpful conversations with friends. First, I thank Jeffrey Wack for a pivotal discussion which provided the key insight. I also thank Jayadev Athreya, professor of mathematics, University of Washington, for his support through the writing process. Lastly, I thank Elliot Kienzle who provided invaluable feedback, unwavering encouragement, and boundless excitement for these pretty pictures. I would also like to thank the reviewers for their comments and revisions which much improved the quality of this article. 

I dedicate this article to my father, Alan Herr. He was the first person to whom I showed the ``grid of graphs''. Dad, you will always be my inspiration.
\pagebreak

\nocite{*}
\bibliography{bib}
\bibliographystyle{plain}

\end{document}